\documentclass[12pt,reqno]{amsart}

\usepackage{amssymb,amsmath,amsthm,cases,color}
\usepackage[]{bm}
\usepackage[]{colonequals}
\usepackage{bbm}

\newcommand{\calU}{\mathcal{U}}

\newcommand{\sgn}{\operatorname{sgn}}

\def\wh{\widehat}
\def\wt{\widetilde}

\def\EE{\mathbb{E}}
\def\VV{\mathbb{V}}
\def\PP{\mathbb{P}}
\def\VV{\mathbb{V}}
\def\1{\mathbbm{1}}

\def\VV{\mathbb{V}}
\def\wh{\widehat}
\def\sumsum{\mathop{\sum\sum}}

\def\calF{\mathcal{F}}

\newtheorem{theorem}{Theorem}[section]
\newtheorem{lemma}[theorem]{Lemma}
\newtheorem{proposition}[theorem]{Proposition}
\newtheorem{corollary}[theorem]{Corollary}
\newtheorem{assumption}[theorem]{Assumption}
\newtheorem{definition}[theorem]{Definition}

\usepackage[margin=1in]{geometry}

\author[Zhao]{Yue Zhao}
\address{Yue Zhao. Department of Statistical Science, Cornell University}

\author[Wegkamp]{Marten H. Wegkamp}
\address{Marten H. Wegkamp. Department of Mathematics \&
Department of Statistical Science, Cornell University}
\title{Semiparametric Gaussian copula classification}

\begin{document}

\baselineskip=18pt

\begin{abstract}
\noindent
This paper studies the binary classification of two distributions with the same Gaussian copula in high dimensions.  Under this semiparametric Gaussian copula setting, we derive an accurate semiparametric estimator of the log density ratio, which leads to our empirical decision rule and a bound on its associated excess risk.  Our estimation procedure takes advantage of the potential sparsity as well as the low noise condition in the problem, which allows us to achieve faster convergence rate of the excess risk than is possible in the existing literature on semiparametric Gaussian copula classification.  We demonstrate the efficiency of our empirical decision rule by showing that the bound on the excess risk nearly achieves a convergence rate of $n^{-1/2}$ in the simple setting of Gaussian distribution classification.\\

%\noindent
%{\sc MSC2000 Subject classification}: {62G07, 62H12.}\\

\end{abstract}

%----------------------------------------------------------------------------------------------------

\maketitle

\section{Introduction}
\label{sec:Introduction}

\subsection{Background}
\label{sec:background}

This paper studies the binary classification of semiparametric Gaussian copulas in high dimensions.  We first briefly review the general classification setting.  We assume throughout that the random vector $(X,Y)\in\mathbb{R}^d\times \{0,1\}$, with $X=(X_1,\dots,X_d)^T$, unless otherwise specified.  The goal of classification is to determine the value of the unobserved $Y$ based on an observed realization of $X$.  The optimal decision rule, namely the Bayes rule $\delta^*:\mathbb{R}^d\rightarrow\{0,1\}$, predicts $Y=1$ if and only if the logarithm of the ratio of densities of $(X|Y=0)$ to $(X|Y=1)$, $\log(f^1/f^0):\mathbb{R}^d\rightarrow\mathbb{R}$, at $X$ satisfies
\begin{align}
\log(f^1/f^0)(X) = \log \dfrac{f^1(X)}{f^0(X)} \ge 0, \nonumber
\end{align}
or equivalently if and only if $\eta(X)\ge 1/2$.  Here $f^y:\mathbb{R}^d\rightarrow\mathbb{R}$ is the multivariate density for the random vector $(X|Y=y)$, and $\eta:\mathbb{R}^d\rightarrow[0,1]$ defined as $\eta(x) = \PP(Y=1|X=x)$ is the regression function.  For simplicity here and throughout the paper we assume that $\PP(Y=0)=\PP(Y=1)=1/2$, and $\mathbb{R}$ denotes the extended real number line.

In practice, the Bayes rule is unavailable to us.  Instead, we have at our disposal a \text{training set} $\left\{(X^i,Y^i),1\le i\le n\right\}$ such that each $(X^i,Y^i)$ is an independent copy of $(X,Y)$.  From the training set, we wish to construct an efficient empirical decision rule $\wh\delta_n:\mathbb{R}^d\rightarrow\{0,1\}$.  In this paper our construction of $\wh\delta_n$ will be based on an estimator $\wh{\log(f^0/f^1)}$ of $\log(f^0/f^1)$ such that the rule $\wh \delta_n$ predicts $1$ at $X$ if and only if $\wh{\log(f^0/f^1)}(X)\ge 0$.

One of the most popular classification methods is linear discriminant analysis (LDA).  Here we first consider this method in Gaussian distribution classification.  Suppose the random vector $(X,Y)$ satisfies $(X|Y=0)\sim N(\mu_0,\Sigma)$ and $(X|Y=1)\sim N(\mu_1,\Sigma)$, for the mean vectors $\mu_0,\mu_1\in\mathbb{R}^d$ and the common covariance matrix $\Sigma\in\mathbb{R}^{d\times d}$.  In this case, the Bayes LDA rule predicts $Y=1$ if and only if $(X-\mu)^T \Omega\mu_d\ge 0$; here $\mu=(\mu_0+\mu_1)/2$, $\mu_d=\mu_1-\mu_0$, and $\Omega$ is the precision matrix, i.e., $\Omega=\Sigma^{-1}$.  Then, in the traditional fixed $d$ setting, the classical empirical LDA rule, or Fisher's rule, makes prediction by replacing $\mu_0$, $\mu_1$ and $\Omega$ in the Bayes rule with their empirical versions $\wh\mu_0$, $\wh\mu_1$ and $(\wh\Sigma)^{-1}$ respectively, and this rule has been well studied \cite{McLachlan92}.

%$\wh\Sigma$ is not invertible, so the traditional empirical decision rule is no longer applicable.

In the high dimensional setting when $d\gtrsim n$, it is well known that the classical empirical LDA rule often performs poorly without additional assumptions \cite{Bickel04,Shao11}.  Considerable progress has been made toward devising efficient empirical LDA rules in the high dimensional setting, typically by exploiting the potential sparsity in the problem, typically by assuming that $\Omega\mu_d\in\mathbb{R}^d$ is sparse \cite{Cai11,Clemmensen11,Fan12,KolarLiu13,Mai12,Shao11}.  In an orthogonal research direction, the traditional LDA under the Gaussian setting has been extended to tackle non-Gaussian distributions in the semiparametric LDA (SeLDA) model \cite{Lin03}.  More recently, the two aforementioned directions have been combined to further extend the LDA to classify non-Gaussian distributions in high dimensions by exploiting sparsity in the SeLDA model \cite{Han13,Mai13}.

Because the framework of SeLDA is closely related to our study in this paper, we will describe it in some details.  As in \cite{Lin03}, the SeLDA model assumes that there exists a $d$-variate transformation function $\alpha=(\alpha_1,\dots,\alpha_d)^T:\mathbb{R}^d\rightarrow\mathbb{R}^d$ that is strictly increasing (i.e., each univariate component $\alpha_i:\mathbb{R}\rightarrow\mathbb{R}$, for $i\in\{1,\dots,d\}$, is a strictly increasing function), such that $(\alpha(X)|Y=0)\sim N(\mu_0,\Sigma)$ and $(\alpha(X)|Y=1)\sim N(\mu_1,\Sigma)$ for some mean vectors $\mu_0,\mu_1\in\mathbb{R}^d$ and the common covariance matrix $\Sigma\in\mathbb{R}^{d\times d}$.  Here we use the convention that, for a vector $x=(x_1,\dots,x_d)^T\in\mathbb{R}^d$, $\alpha(x) = (\alpha_1(x_1),\dots,\alpha_d(x_d))^T$.  Then, as a natural generalization of the Bayes LDA rule under the traditional Gaussian setting, the Bayes rule under the SeLDA model predicts $Y=1$ if and only if $\left(\alpha(X)-\mu\right)^T \Omega\mu_d\ge 0$, with the same definitions for $\mu$, $\mu_d$ and $\Omega$ as described earlier.  Then, an efficient empirical decision rule under the SeLDA model is derived by replacing the unknown quantities in the Bayes rule, namely $\alpha$, $\mu$ and $\Omega\mu_d$, by their accurate estimates.  We emphasize that, under the SeLDA model, the transformation function $\alpha$ is required to be the same independent of the value of $Y$, because when classifying a new observation we have no prior knowledge of the value of $Y$ (which is what we would like to predict in a classification problem).

Because under the SeLDA model, $(X|Y=0)$ and $(X|Y=1)$ have the same Gaussian copula, SeLDA can also be regarded as an special instance of the \textit{semiparametric Gaussian copula classification problem}, or simply the Gaussian copula classification problem, which we define as classifying two distributions whose dependence structures are described by the same Gaussian copula but whose marginals are not explicitly specified.

%We say a random variable $X=(X_1,\dots,X_d)^T\in\mathbb{R}^d$ has a Gaussian copula if

%Specifically, instead of assuming that $X$ given $Y=y$, $y\in\{0,1\}$ follows a Gaussian distribution, we assume that there exists some $d$-variate strictly increasing function $f=(f_1,\dots,f_d)$ such that $f(X)$ given $Y=y$, $y\in\{0,1\}$ follows a $N(\mu_y,\Sigma)$ distribution.

%Note that it is required the function $f$ to be the same independent of distribution from which $X$ is drawn, and it is easy to see why such a requirement has to be in place, because otherwise, i.e., if different $f$'s are needed, when classifying a new observed $X$ we do not know which $f$ to use.

%This problem can be seen as an extension of the classical linear discriminant analysis (LDA) problem.

%----------------------------------------------------------------------------------------------------

\subsection{Limitation of the existing method}

Even though the SeLDA model is an instance of the Gaussian copula classification problem, it is in fact applicable only to a quite restrictive collection of distributions on $(X,Y)$ such that $(X|Y=y)$, $y\in\{0,1\}$ have the same Gaussian copula.  The assumption of SeLDA that the transformation function $\alpha$ must be the same independent of the class $y\in\{0,1\}$ already implies that some restriction must exist between the marginals of $(X|Y=0)$ and $(X|Y=1)$.  Here we show that the implied restriction is quite strong, perhaps even unnatural.  For simplicity, we assume here that $d=1$.  Then, the assumption of the SeLDA model states that there exists a strictly increasing univariate function $\alpha$ such that $(\alpha(X)|Y=y)\sim N(\mu_y,\sigma^2)$ for $y\in\{0,1\}$, which implies that $(\alpha(X)-\mu_y|Y=y)\sim N(0,\sigma^2)$ for $y\in\{0,1\}$.  Hence, recalling that $\alpha$ is strictly increasing, we derive the following relationship between the distributions of $(X|Y=0)$ and $(X|Y=1)$: for an arbitrary $t\in\mathbb{R}$, we have
\begin{align}
\PP(X\le t|Y=0) &= \PP(\alpha(X)\le \alpha(t)|Y=0 ) = \PP\left(\left.\alpha(X)-\mu_0\le\alpha(t)-\mu_0\right\vert Y=0\right) \nonumber \\
&=\PP\left(\left.\alpha(X)-\mu_1\le\alpha(t)-\mu_0\right\vert Y=1\right) = \PP\left(\left. \alpha(X)-\mu_1+\mu_0 \le \alpha(t) \right\vert Y=1\right) \nonumber \\
&= \PP(\alpha^{-1}(\alpha(X)-\mu_1+\mu_0)\le t|Y=1 ). \nonumber
\end{align}
Hence, SeLDA imposes a rather bizarre requirement that $(X|Y=0)$ and $(\alpha^{-1}(\alpha(X)-\mu_1+\mu_0)|Y=1)$ must have the same distribution.  This requirement would become more interpretable if the function $\alpha$ satisfies $\alpha^{-1}(\alpha(t)-\mu_1+\mu_0) = t-\mu_1+\mu_0$ for all $t\in\mathbb{R}$, which would imply that the random variables $(X|Y=0)$ and $(X|Y=1)$ are a constant shift $\mu_1-\mu_0$ from each other.  However, this is typically not the case unless $\alpha$ is the identity function, but then we simply revert back to the traditional case of classifying two Gaussian distributions with the same variance.  To put it somewhat differently, as one example of the strong restriction it places on the distribution of $(X,Y)$, SeLDA typically cannot accommodate the very natural scenario where $(X|Y=0)$ and $(X|Y=1)$ are a constant shift from each other, unless $(X|Y=0)$ and $(X|Y=1)$ are already normally distributed.

%----------------------------------------------------------------------------------------------------

\subsection{Proposed research}

In this paper, we study the classification of two random vectors $(X|Y=0), (X|Y=1)\in\mathbb{R}^{d}$ that have the same Gaussian copula but that are otherwise completely arbitrary (except for certain regularity conditions) --- in short, we allow each class $y\in\{0,1\}$ to have their own transformation function $\alpha_y$ --- and develop a genuine and efficient Gaussian copula classification method in high dimensions.  We will make the blanket assumption that $(X|Y=0)$ and $(X|Y=1)$ have continuous marginals, and the Gaussian copula characterizing $(X|Y=0)$ and $(X|Y=1)$ has copula correlation matrix $\Sigma$.

%At this stage, $(X|Y=0)$ and $(X|Y=1)$ are otherwise arbitrary.

As the starting point of our study, and also to describe our general strategy, we derive in this section the explicit from of the log density ratio $\log(f^0/f^1)$, which directly translates into an explicit Bayes rule for the Gaussian copula classification problem.  For the rest of the paper, we will construct a precise estimator of this ratio to establish an efficient corresponding empirical rule.

In the following we let $i\in\{1,\dots,d\}$, $y\in\{0,1\}$, $t\in\mathbb{R}$ and $x=(x_1,\dots,x_d)^T\in\mathbb{R}^d$.  Throughout the paper, we let $F_{i|y}$ and $f_{i|y}$ be, respectively, the distribution function and the density function of the $i$th coordinate for class $y$, and let $F_i = \left(F_{i|0} + F_{i|1}\right)/2$ and $f_i = \left(f_{i|0} + f_{i|1}\right)/2$ be, respectively, the marginal distribution function and the marginal density function of the $i$th coordinate (when $\PP(Y=0)=\PP(Y=1)=1/2$ as we are assuming).  We let $\Phi$ be the distribution function and $\Phi^{-1}$ the quantile function of $N(0,1)$.  We let the function $\alpha_{i|y}:\mathbb{R}\rightarrow\mathbb{R}$ be
\begin{align}
\alpha_{i|y}(t) = \Phi^{-1} (F_{i|y}(t)),
\label{def_alpha_y_i}
\end{align}
and we let the function $\alpha_y:\mathbb{R}^d\rightarrow\mathbb{R}^d$ be
\begin{align}
\alpha_y(x) &= \left( \alpha_{1|y}(x_1),\dots,\alpha_{d|y}(x_d) \right)^T.
\label{eq:def_alpha_x}
\end{align}
Then, we let the function $\Delta\alpha = \left( \Delta\alpha_1, \cdots, \Delta\alpha_d \right)^T:\mathbb{R}^d\rightarrow\mathbb{R}^d$ be
\begin{align}
\Delta\alpha(x) &= \alpha_0(x) - \alpha_1(x) = \left( \alpha_{1|0}(x_1)-\alpha_{1|1}(x_1),\dots,\alpha_{d|0}(x_d)-\alpha_{d|1}(x_d) \right)^T \nonumber \\
&= \left( \Delta\alpha_1(x_1), \cdots, \Delta\alpha_d(x_d) \right)^T,
%\Delta\alpha &= \alpha_0 - \alpha_1 = \left( \Delta\alpha_1, \cdots, \Delta\alpha_d \right)^T
\label{eq:def_Delta_alpha_x}
\intertext{and the function $\Delta\log f = (\Delta\log f_1,\dots,\Delta\log f_d)^T:\mathbb{R}^d\rightarrow\mathbb{R}^d$ be}
\Delta\log f(x) & = \left( \log f_{1|0}(x_1) - \log f_{1|1}(x_1), \cdots, \log f_{d|0}(x_d) - \log f_{d|1}(x_d) \right)^T \nonumber \\
&= (\Delta\log f_1(x_1), \cdots, \Delta\log f_d(x_d) )^T.
\label{eq:def_Delta_log_f_x}
\end{align}

We state in Theorem~\ref{thm:Gaussian_copula_density_ratio} the explicit from of the log density ratio $\log(f^0/f^1)$.
\begin{theorem}
\label{thm:Gaussian_copula_density_ratio}
For all $x\in\mathbb{R}$, we have
\begin{align}
\log(f^0/f^1)(x) &= -\dfrac{1}{2} (\alpha_0(x)+ \alpha_1(x) )^T \left( \Omega - I_d \right) ( \alpha_0(x) - \alpha_1(x)) + \sum_{i=1}^d \log\dfrac{ f_{i|0}( x_i) }{ f_{i|1}( x_i) } \nonumber \\
&= -\dfrac{1}{2} (\alpha_0(x) + \alpha_1(x) )^T \beta^*(x) + \sum_{i=1}^d \Delta\log f_i(x_i).
\label{eq:thm_density_ratio}
\end{align}
Here, $I_d$ denotes the $d\times d$ identity matrix, and for brevity (and analogous to the notation of \cite{Cai11}), we define the function $\beta^*=(\beta_1^*,\dots,\beta_d^*)^T=(\Omega-I_d)\Delta\alpha:\mathbb{R}^d\rightarrow\mathbb{R}^d$ as, for $x\in\mathbb{R}^d$,
\begin{align}
\beta^*(x)  = (\beta^*_1(x),\dots,\beta^*_d(x))^T = \left( \Omega - I_d \right) \Delta\alpha(x).
\label{eq:def_beta_star_x}
\end{align}
\end{theorem}

\begin{proof}[Proof of Theorem~\ref{thm:Gaussian_copula_density_ratio}]
The proof can be found in Section~\ref{sec:proof_thm:Gaussian_copula_density_ratio}.
\end{proof}

It is clear from Equation~(\ref{eq:thm_density_ratio}) in Theorem~\ref{thm:Gaussian_copula_density_ratio} that the log density ratio $\log(f^0/f^1)$ at $x$ is decomposed as the sum of the term
\begin{align}
\left[(\alpha_0 + \alpha_1)^T \beta^*\right](x),
\label{eq:copula_part}
\end{align}
which we refer to as the copula part, and the term
\begin{align}
\sum_{i=1}^d \Delta\log f_i(x_i),
\label{eq:naive_Bayes_part}
\end{align}
which we refer to as the naive Bayes part.  Note that the copula part and the naive Bayes part are thus named because the former arises from the particular multivariate dependence structure described by the Gaussian copula, while the latter would arise even in the case of the classification of two multivariate distributions each with independent individual coordinates.  The estimation of the copula part and the naive Bayes part will involve different techniques.  Thus we will derive their estimators separately; in particular, we will derive the deviation properties of these estimator.

%We will study the estimation of the copula part in Section~\ref{sec_copula}, and the estimation of the naive Bayes part in Section~\ref{sec_naive_bayes}.  In particular, we will derive tight large deviation inequalities for our estimators.

The estimators of the copula part and the naive Bayes part combined yield our semiparametric estimator $\wh {\log(f^0/f^1)}$ of the log density ratio $\log(f^0/f^1)$, which directly translates into our empirical decision rule $\wh\delta_n$.  By the aforementioned deviation properties, we can straightforwardly calculate the main result of our paper, a bound on the \textit{excess risk}
%in Section~\ref{sec:risk_bound}
\begin{align}
\PP(\wh\delta_n(X)\neq Y)-\PP(\delta^*(X)\neq Y)
\label{eq:excess_risk}
\end{align}
associated with the empirical decision rule $\wh\delta_n$.  In words, the excess risk, which is a canonical benchmark for evaluating the efficiency of a decision rule, is the probability of misclassification associated with the empirical rule $\wh\delta_n$ in excess of that associated with the optimal Bayes rule $\delta^*$.  Moreover, by the same reason, we can easily incorporate in the excess risk calculation the \textit{margin assumption} (i.e., ``low noise'' condition) to take advantage of the potential low noise condition in the problem, which allows us to achieve faster convergence rate of the excess risk than is possible in the existing literature on Gaussian copula classification.

We will allow the dimension $d$ and certain other parameters (to be specified in more details throughout the paper) to grow with the sample size $n$.  To avoid error accumulation in high dimensions, throughout our studies, we will present explicit procedures that take advantage of the potential sparsity present in the problem, in particular the joint sparsity of $\Delta\alpha$ and $\Omega-I_d$ for the copula part, and the sparsity of $\Delta\log f$ in the naive Bayes part.

To demonstrate the efficiency of our empirical decision rule $\wh\delta_n$, we calculate the particular bound on excess risk that we achieve in the simple case of classifying two Gaussian distributions with common covariance, and show that our empirical decision rule nearly achieves the rate of $n^{-1/2}$ (with dimension $d$, sparsity indices, etc., all fixed).  (Of course, this simple case is more specifically and efficiently tackled by several well-developed high-dimensional LDA methods.  Our aim here is not to compete with these methods, but only to demonstrate the convergence rate, in particular with respect to $n$, of our method in this case.)

%----------------------------------------------------------------------------------------------------

\subsection{Outline of the paper}

To facilitate presentation, we collect in Section~\ref{sec:risk_bound} the major ingredients of our paper.  First, Section~\ref{sec:sparsity} describes the types of sparsity that we exploit in our Gaussian copula classification framework.  Then, Section~\ref{sec:main_result_copula_part} describes the estimation procedure for the copula part and Section~\ref{sec:main_result_Bayes_part} describes the estimation procedure for the naive Bayes part.  Then, Section~\ref{sec:main_result_decision_rule} describes the feature of the resultant empirical decision rule $\wh\delta_n$, and presents the aforementioned main result of the paper, a bound on the excess risk associated with the rule $\wh\delta_n$, in Theorem~\ref{thm:excess_risk}.  Section~\ref{sec:case_study} presents the particular bound on excess risk we achieve when classifying two Gaussian distributions with common covariance.  More detailed, ``step-by-step'' studies of the estimation of the copula part and the naive Bayes part are deferred to Sections~\ref{sec_copula} and \ref{sec_naive_bayes} respectively.

For brevity of presentation, we defer the detailed proofs for Sections~\ref{sec:Introduction}, \ref{sec:risk_bound}, \ref{sec_copula} and \ref{sec_naive_bayes} to Sections~\ref{sec:proof_for_sec:Introduction}, \ref{sec:proof_for_sec:risk_bound}, \ref{sec:proof_for_sec_copula}, \ref{sec:proof_for_sec_naive_bayes} respectively.

%----------------------------------------------------------------------------------------------------
%----------------------------------------------------------------------------------------------------
%----------------------------------------------------------------------------------------------------
%----------------------------------------------------------------------------------------------------

\subsection{Conventions and notations}

For brevity of presentation, we assume that we have $n$ independent copies $X^{y,j}=(X^{y,j}_1,\dots,X^{y,j}_d)^T\in\mathbb{R}^d$, $j\in\{1,\dots,n\}$, of $(X|Y=y)$ for each class $y\in\{0,1\}$.  We can easily accommodate unequal sample sizes for the two classes.

%v\in\mathbb{R}^d:
For any vector $v$, we will use $[v]_k$ to denote its $k$th element, and for any matrix $A$, we will use $[A]_{k\ell}$ to denote the $k,\ell$th element of $A$, and $[A]_{k\cdot}$ to denote the $k$th row of $A$.  For matrices, we let $\|\cdot\|_q$ denote the induced $q$-matrix norm, i.e., $\|A\|_q = \sup_{\|v\|_{\ell_q}=1}\|Av\|_{\ell_q}$, and let $\|A\|_{\max} =\max_{i,j} |[A]_{i,j}|$; in particular, $\|A\|_{\infty}$ is the maximum row sum of the matrix $A$.  We let $\lambda_{\max}(\cdot)$ denote the largest eigenvalue of the argument.  Typically, we let $t\in\mathbb{R}$ , and $x\in\mathbb{R}^d$.  We let $I_d$ denote the $d\times d$ identity matrix.

We let $Z$ denote a standard normal random variable.  As stated earlier, $\Phi$ and $\Phi^{-1}$ denote the distribution function and the quantile function of $Z$.  We let $\phi$ denote the probability density functions of $Z$, and $\Phi_{\mu}$ and $\Phi_{\mu}^{-1}$ denote the distribution function and the quantile function of $Z+\mu$ respectively.  We note the basic fact that $\Phi_{\mu}^{-1}(\cdot) = \Phi^{-1}(\cdot)+\mu$.

For any absolute (i.e., numerical) constant $a$, we let $a^+$ denote an arbitrary but throughout the paper fixed absolute constant that is strictly greater than $a$.  We let $C$ denote a constant whose value may change from line to line of even within the same line, but is always an absolute constant that doesn't depend on any parameter in the problem (e.g., sample size, dimension, sparsity indices, locations $r\in\mathbb{R}, x\in\mathbb{R}^d$), unless otherwise specified.  We let $C$ and $J$ with subscripts denote constants with particular chosen values.

%----------------------------------------------------------------------------------------------------

\section{Construction and performance summary of the empirical decision rule $\wh\delta_n$}
\label{sec:risk_bound}

\subsection{Exploiting potential sparsity in the problem}
\label{sec:sparsity}

In this section, we describe the types of sparsity we exploit in our Gaussian copula classification framework.

We first focus on the copula part as defined in (\ref{eq:copula_part}).  As can be seen from (\ref{eq:copula_part}), because the vector-valued output of the function $\alpha_0+\alpha_1$ is clearly non-sparse and is monotone in $x$, the potential sparsity in the copula part should come from $\beta^*$.  Instead of directly exploiting the sparsity induced by $\beta^*$, however, we aim to study the following sparsity sets and indices induced by the function $|\Omega-I_d| |\Delta\alpha(x)|$: for $x\in\mathbb{R}^d$, we let
\begin{align}
\label{eq:def_bar_S_x}
S_x'& = \{i: |[\Omega-I_d]_{i\cdot}| |\Delta\alpha(x)| \neq 0\}, \quad s_x'= |S_x'|.
\end{align}
Here and throughout the paper, $|\cdot|$ with vector or matrix as argument returns the absolute value component-wise, and with set as argument returns cardinality.  In words, $i\in S_x'$ if and only if the two vectors $[\Omega-I_d]_{i\cdot}^T$ and $\Delta\alpha(x)$ have some overlapping nonzero components.  Then, estimating the sparsity set $S_x'$ becomes equivalent to estimating the sparsity patterns of $\Omega-I_d$ and $\Delta\alpha(x)$ separately.

One may be curious why we do not exploit the sparsity directly induced by the function $\beta^*$, namely the sparsity represented by the following sparsity sets and indices: for $x\in\mathbb{R}^d$,
\begin{align}
\label{eq:def_S_x}
S_x =\{i:\beta_i^*(x) = [\Omega-I_d]_{i\cdot} \Delta\alpha(x) \neq 0\}, \quad s_x=|S_x|;
\end{align}
note that $S_x\subset S_x'$ for all $x\in\mathbb{R}^d$.  We provide motivation for our choice here.  A sparsity pattern analogous to that represented by (\ref{eq:def_S_x}), namely the sparsity of the vector $\Omega\mu_d$ as described in Section~\ref{sec:background}, is indeed commonly exploited when classifying two Gaussian distributions $(X|Y=y)\sim N(\mu_y,\Sigma)$, $y\in\{0,1\}$ in high dimensions (e.g., see \cite{Cai11}).  To contrast this setting and in particular the sparsity pattern of $\Omega\mu_d$ to our Gaussian copula classification framework, here we briefly consider Gaussian distribution classification.  For simplicity we first assume that all the diagonal elements of $\Sigma$ are equal to one.  In this case, $\Delta\alpha=\Delta\alpha(x)$ is a constant function equal to $\mu_d=\mu_1-\mu_0$ for all $x\in\mathbb{R}^d$.  Then, the sparsity pattern analogous to that represented by (\ref{eq:def_S_x}) is the sparsity of the \textit{constant} vector $\Omega\Delta\alpha=\Omega\mu_d$, which prominently appears in the Bayes LDA rule.

The rationale behind exploiting the sparsity of the vector $\Omega\mu_d$, instead of the separate sparsity patterns of $\Omega$ and $\mu_d$, is that the $i$th component of the vector $\Omega\mu_d$, namely $[\Omega]_{i\cdot}^T\mu_d$, can be zero even if the vectors $[\Omega]_{i\cdot}^T$ and $\mu_d$ have overlapping nonzero components, if the latter two vectors are orthogonal.  However, this rationale is largely lost in our more general Gaussian copula classification framework.  Here, typically, $\Delta\alpha(X)$ is a continuous, rather than a constant, random vector (and the nonzero components of $\Delta\alpha(X)$ are typically not constant scalings of each other).  As such, up to an event of probability zero, the event on which $\Delta\alpha(X)$ is orthogonal to the constant vector $[\Omega-I_d]_{i\cdot}^T$ is equal to the event on which $\Delta\alpha(X)$ and $[\Omega-I_d]_{i\cdot}^T$ have no overlapping nonzero components.  Equivalently, $S_X'=S_X$ with probability one.  For illustration, we provide a simple but extreme example.  We again consider classifying two Gaussian distributions $(X|Y=y)\sim N(\mu_y,\Sigma_y)$, but this time we assume that $\Sigma_0$ has all diagonal elements equal to one, but $\Sigma_1=a^2\Sigma_0$ for $a\neq 1$ (which results in a \textit{quadratic discriminant analysis} problem, and which in this particular instance still falls under our Gaussian copula classification framework because $(X|Y=y)$, $y\in\{0,1\}$ still have the same Gaussian copula).  Then, the $X$-dependent component of $\Delta\alpha(X)$ becomes $(1-1/a)X$, and $((\Omega-I_d)\Delta\alpha(X)|Y=y)$ follows a $d$-variate Gaussian distribution with covariance $(1-1/a)^2(\Omega-2I_d+\Sigma)$.  Hence, $S_X'=S_X=\{1,\dots,d\}$ with probability one unless $\Omega=\Sigma=I_d$, in which case $S_X'=S_X=\emptyset$ with probability one, i.e., the sparsity sets $S_X'$ and $S_X$ are equal with probability one.

As stated immediately following (\ref{eq:def_bar_S_x}), the sparsity induced by the function $|\Omega-I_d| |\Delta\alpha|$ as in (\ref{eq:def_bar_S_x}) is in turn induced by the separate sparsity patterns induced by the function $\Delta\alpha$, represented by the sets and indices, for $x\in\mathbb{R}^d$,
\begin{align}
\label{eq:S_x_p}
S_x''& =\{i:\Delta\alpha_i(x_i) \neq 0\},  \quad s_x''=|S_x''|,
\end{align}
and the matrix $\Omega-I_d$.  We will consider the sparse estimation of $\Delta\alpha$ in Section~\ref{sec:sparse_estimation_Delta_alpha}, and the sparse estimation of $\Omega-I_d$ in Section~\ref{sec:sparse_estimation_Omega}.

Analogous to (\ref{eq:S_x_p}), we let the sparsity sets and indices for the naive Bayes part induced by the function $\Delta\log f$ be, for $x\in\mathbb{R}^d$,
\begin{align}
\label{eq:S_f_x}
S^f_x &=\{i: \Delta\log f_i(x_i) \neq 0\}, \quad s^f_x=|S^f_x|.
\end{align}

A typical model that induces sparsities for both the functions $\Delta\alpha$ and $\Delta\log f$ is the classification of two distributions $(X|Y=y)$, $y\in\{0,1\}$ such that the marginals $(X_i|Y=y)$, $i\in\{1,\dots,d\}$ of the two distributions are identical except at a subset $S\subset\{1,\dots,d\}$ of coordinates.  For concreteness we assume $S=\{1,\dots,s\}$ and so $|S|=s$, and $s<d$.  In this case, $S_x'',S^f_x\subset S$ for all $x\in\mathbb{R}^d$.  Then, if furthermore $\Omega-I_d$ is appropriately sparse, then the function $|\Omega-I_d||\Delta\alpha|$ is sparse.  For instance, if the first $s$ coordinates of $(\alpha_y(X)|Y=y)$ are independent and are furthermore independent with the remaining $d-s$ coordinates, then the first $s$ columns of $\Omega-I_d$ are identically zero, which implies that $|\Omega-I_d||\Delta\alpha|$ is identically zero and $S_x'$ is identically the empty set at all $x\in\mathbb{R}^d$.  Having considered such an example, we emphasize that our Gaussian copula classification framework does not require that the sets $S_x', S_x'', S^f_x$ are constant over $x\in\mathbb{R}^d$.

%----------------------------------------------------------------------------------------------------

\subsection{Estimation of the copula part}
\label{sec:main_result_copula_part}

\subsubsection{Sparse estimation of $\Delta\alpha$}
\label{sec:sparse_estimation_Delta_alpha}

We let, for some $0<\gamma<2$,
\begin{align}
\label{eq:a_n}
a_n &= \sqrt{\gamma \log n}, \\
\label{eq:g_n_gamma}
g(n,\gamma) &= \dfrac{\phi(a_n)}{2a_n} = \dfrac{1}{2\sqrt{2\pi}} \dfrac{n^{-\gamma/2}}{\sqrt{\gamma\log n}}.
\end{align}
The parameter $\gamma$ will eventually be chosen to minimize our bound on the excess risk according to the discussion following Theorem~\ref{thm:excess_risk}; at present we let it be arbitrary.  We will make the blanket assumption that $n$ is large enough such that $a_n\ge 1$.

We let $\wh F_{i|y}:\mathbb{R}\rightarrow\mathbb{R}$ be the empirical distribution function of the $i$th coordinate for class $y$, i.e., for $t\in\mathbb{R}$,
\begin{align}
\wh F_{i|y}(t) = \dfrac{1}{n} \sum_{j=1}^n \1\left\{ X^{y,j}_i\le t \right\}, \nonumber
\end{align}
and let $\wh F_i:\mathbb{R}\rightarrow\mathbb{R}$ be the empirical marginal distribution function of the $i$th coordinate, i.e.,
\begin{align}
\wh F_i = \dfrac{1}{2} \left[ \wh F_{i|0} + \wh F_{i|1} \right]. \nonumber
\end{align}
We let $\wh\alpha_{i|y}:\mathbb{R}\rightarrow\mathbb{R}$ and $\wh\alpha_y:\mathbb{R}^d\rightarrow\mathbb{R}^d$ be, respectively, the estimator of $\alpha_{i|y}$ and $\alpha_y$ defined as: for $t\in\mathbb{R}$ and $x\in\mathbb{R}^d$,
\begin{align}
\label{def_wh_alpha_y_i}
\wh\alpha_{i|y}(t) &= \Phi^{-1} ( \wh F_{i|y} (t)), \\
\wh\alpha_y(x) &=(\wh\alpha_{1|y}(x_1),\dots,\wh\alpha_{d|y}(x_d) )^T. \nonumber
\end{align}
The property of the estimator $\wh\alpha_{i|y}$ will be discussed in more details in Section~\ref{sec_est_tran}.  Here we only note that, as we will see in Lemma~\ref{thm:delta_alpha}, we focus on the estimation of $\alpha_{i|y}$ over the regime specified by $t:\alpha_{i|y}(t)=\Phi^{-1}(F_{i|y}(t))\in [-a_n, a_n]$, i.e., we focus on the estimation of $\alpha_{i|y}$ for moderate values of $F_{i|y}(t)$.  By Proposition~\ref{prop:A_n_prob_bound}, up to a log factor in $n$, the complement of this region has probability $n^{-\gamma/2}$ with respect to the random variable $(X_i|Y=y)$.  We will loosely refer to the rate $n^{-\gamma/2}$ as the ``exclusion probability,'' and will match some other probability bounds to this rate in the rest of the paper.

Next, we let $\wt\Delta\alpha=(\wt\Delta\alpha_1, \dots, \wt\Delta\alpha_d)^T:\mathbb{R}^d\rightarrow\mathbb{R}^d$ with
\begin{align}
\wt\Delta\alpha(x) = (\wt\Delta\alpha_1(x_1), \dots, \wt\Delta\alpha_d(x_d))^T
\label{eq:wt_Delta_alpha}
\end{align}
be our sparse estimator of $\Delta\alpha$, whose construction consists of two potential steps; we fix arbitrary $i\in\{1,\dots,d\}$ and arbitrary $t\in\mathbb{R}$:
\begin{enumerate}
\item
First, we check whether
\begin{align}
\wh F_{i}(t) \le 4 g(2n,\gamma)\quad \text{or} \quad \wh F_{i}(t) \ge 1 - 4 g(2n,\gamma).
\label{eq:F_test_1}
\end{align}
(Note that the test involves the empirical marginal distribution function $\wh F_{i}$.  The constant $4$ in (\ref{eq:F_test_1}) is chosen for convenience.)  At the same time, we also check whether
\begin{align}
\max\left\{ \dfrac{ \max\{ \wh F_{i|0}(t),\wh F_{i|1}(t) \} }{ \min\{ \wh F_{i|0}(t),\wh F_{i|1}(t) \} } ,  \dfrac{ \max\{ 1-\wh F_{i|0}(t),1-\wh F_{i|1}(t) \} }{ \min\{ 1-\wh F_{i|0}(t),1-\wh F_{i|1}(t) \} }  \right\} \le \dfrac{1+\bar\delta_{n,d,\gamma}}{1-\bar\delta_{n,d,\gamma}}.
\label{eq:F_test_2}
\end{align}
Here
\begin{align}
\bar\delta_{n,d,\gamma} = \left[ 3 n^{-1} g^{-1}(2n,\gamma) \log(d\cdot n^{\frac{\gamma}{2}}) \right]^{1/2}.
\label{eq:def_bar_delta}
\end{align}
If either inequality in (\ref{eq:F_test_1}) holds, or if Inequality~(\ref{eq:F_test_2}) holds, we set $\wt\Delta\alpha_i(t)=0$.
\item
Otherwise (i.e., if both (\ref{eq:F_test_1}) and (\ref{eq:F_test_2}) are violated) we set
\begin{align}
\wt\Delta\alpha_i(t) &= \wh\alpha_{i|0}(t) - \wh\alpha_{i|1}(t) = \Phi^{-1} ( \wh F_{i|0} (t)) - \Phi^{-1} ( \wh F_{i|1} (t)).
\label{eq:def_wt_Delta_alpha_i_t}
\end{align}
(Here we have invoked the form of $\wh\alpha_{i|y}$ as defined in (\ref{def_wh_alpha_y_i}).)  It is apparent that in this case $\wt\Delta\alpha_i(t)\neq 0$ because if (\ref{eq:F_test_2}) is violated then necessarily $\wh F_{i|0}(t) \neq \wh F_{i|1}(t)$.
\end{enumerate}
%, and

The basic intuition behind our two-step construction is as follows.  First, test (\ref{eq:F_test_1}) checks whether the value of $F_i(t)$ is likely close to $0$ or $1$.  If so, then the value of at least one of $F_{i|y}(t)$, $y\in\{0,1\}$ is also likely close to $0$ or $1$, and hence the estimation of the corresponding $\alpha_{i|y}(t)$ is likely poor (see the discussion following Lemma~\ref{thm:delta_alpha}).  In this case, we do not try to estimate $\Delta\alpha_i(t)$ at all and so set $\wt\Delta\alpha_i(t)=0$.  Next, test (\ref{eq:F_test_2}) checks whether the values of $F_{i|0}(t)$ and $F_{i|1}(t)$ are likely close, i.e., whether the signal strength is likely small.  If so, we again set $\wt\Delta\alpha_i(t)=0$.  Otherwise we estimate $\Delta\alpha_i(t)$ as in (\ref{eq:def_wt_Delta_alpha_i_t}) (as one normally would in the absence of sparsity).  The property of the estimator $\wt\Delta\alpha$ will be discussed in more details in Section~\ref{sec_est_Delta_alpha}.

%if the value of $F_i(t)$ is likely close to $0$ or $1$ (this is the test (\ref{eq:F_test_1})), where the value of at least one of $F_{i|y}(t)$, $y\in\{0,1\}$ is also likely close to $0$ or $1$ and hence the estimation of at least one of $\alpha_{i|y}(t)$, $y\in\{0,1\}$ is likely poor (see the discussion following Lemma~\ref{thm:delta_alpha}), we do not try to estimate $\Delta\alpha_i(t)$ at all and so set $\wt\Delta\alpha_i(t)=0$.  For likely moderate values of $F_{i|0}(t)$ and $F_{i|1}(t)$, if their values are too close (this is test (\ref{eq:F_test_2})), we again set $\wt\Delta\alpha_i(t)=0$.  Otherwise we estimate $\wt\Delta\alpha_i(t)$ as in (\ref{eq:def_wt_Delta_alpha_i_t}) (as one normally would in the absence of sparsity).  The property of the estimator $\wt\Delta\alpha$ will be discussed in more details in Section~\ref{sec_est_Delta_alpha}.

%----------------------------------------------------------------------------------------------------

\subsubsection{Sparse estimation of $\Omega-I_d$}
\label{sec:sparse_estimation_Omega}

In this section, we collect some existing results on the sparse estimation of $\Omega$, the precision matrix associated with the copula correlation matrix $\Sigma$, which will lead to our sparse estimation of $\Omega-I_d$.

The literature on sparse precision matrix estimation is rapidly growing (see \cite{Cai14} for a recent review), although many of the recent strong results work under (sub-)Gaussian or moment conditions.  It remains to be seen how these results can be generalized to the Gaussian copula setting where a rank-based pilot estimator, such as Kendall's tau matrix, is usually taken as input.  In this paper we simply quote a result working explicitly with Kendall's tau from \cite{Zhao14}.  Our aim is to demonstrate how the sparse estimation of $\Omega$ can be incorporated into our efficient estimation of the copula part, keeping in mind that stronger results may become available in the future.  For concreteness, as in \cite{Zhao14}, in this paper we will concentrate on the sparse estimation of precision matrices within a particular class $\calU(s,M,\kappa)$, defined as
\begin{align}
\calU(s,M,\kappa) = \Bigg\{&\Omega\in\mathbb{R}^{d\times d}:\Omega\succ0,\text{diag}(\Omega^{-1})=\mathbf{1}, \lambda_{\max}(\Omega)\le\kappa, \nonumber \\
&\left.\max_{\ell}\sum_{k=1}^d \1\left\{[\Omega]_{k\ell}\neq0\right\}\le s,\|\Omega\|_{\infty}\le M\right\}.
\label{eq:def_mathcal_U}
\end{align}
Here $\Omega\succ0$ denotes that $\Omega$ is positive definite, and $\kappa$, $s$ and $M$ may scale with $n$ and $d$.

We let $\wh\Sigma$ be the empirical \textit{plug-in estimator} of $\Sigma$ constructed from Kendall's tau statistic, as we describe below.  We first recall the definition of Kendall's tau statistic associated with class $y\in\{0,1\}$: for $1\le k,\ell\le d$, we have
\begin{eqnarray}
\label{tauhat}
\wh \tau^y_{k\ell} = \frac{2}{n(n-1)} \sumsum_{1\le i<j\le n} \text{sgn}\left( (X^{y,i}_{k}-X^{y,j}_{k})(X^{y,i}_{\ell}-X^{y,j}_{\ell}) \right).
\end{eqnarray}
Then, we let $\wh T^y$ be the empirical Kendall's tau matrix associated with class $y$ with entries
\begin{align}
[\widehat{T}^y]_{k\ell} = \wh \tau_{k\ell}^y \quad \text{ for all }  1\le k,\ell\le d,
\label{eq:Kendall_tau_matrix_est}
\end{align}
and form $\wh \Sigma^y$, the plug-in estimator of $\Sigma$ from class $y$, constructed from $\wh T^y$ as
\begin{align}
\label{R_hattau}
\wh\Sigma^y = \sin \left( \frac{\pi}{2} \wh T^y \right).
\end{align}
Here the sine function acts component-wise.  Finally, we let the overall plug-in estimator of $\Sigma$ from both classes be
\begin{align}
\label{R_hattau_overall}
\wh\Sigma = (\wh\Sigma^0+\wh\Sigma^1)/2.
\end{align}
The conceptual justification of employing the plug-in estimator $\wh\Sigma^y$ to estimate $\Sigma$ is provided by the elegant relationship $\Sigma=\sin \left( \frac{\pi}{2} T \right)$ for semiparametric elliptical copulas (which includes semiparametric Gaussian copulas), for $T=\EE \wh T^y$ the matrix of population version of Kendall's tau (i.e., the $k,\ell$th element of $T$ is Kendall's tau coefficient between the $k$th and $\ell$th elements of $(X|Y=y)$).  We refer the readers to the extensive references in Section~1.1 of \cite{WZ14} for more detailed discussion.  Analytically, the plug-in estimator $\wh\Sigma^y$ has proven to be an efficient estimator of $\Sigma$ in terms of both the element-wise $\|\cdot\|_{\max}$ norm and the operator norm \cite{HCL13,Liu12,Mitra14,WZ14,Xue12b}, and such results easily generalize to the (overall) plug-in estimator $\wh\Sigma$.

We let $\wh\Omega'$ be the solution of \cite[Algorithm~(III.6)]{Zhao14} with tuning parameter $\lambda_n$ specified by
\begin{align}
\lambda_n = \dfrac{2}{\sqrt{n}}\log^{\frac{1}{2}}(2n^{\frac{\gamma}{2}}d^2),
\label{eq:tuning_parameter_Omega}
\end{align}
and $\wh\Omega$ be the result of the symmetrization step \cite[(III.11)]{Zhao14} with $\wt\Omega$ replaced by $\wh\Omega'$ and $\|\cdot\|_{*}$ replaced by $\|\cdot\|_{\infty}$.  Then, we construct our sparse estimator $\wt\Omega$ of $\Omega$ by thresholding $\wh\Omega$ as
\begin{align}
[\wt\Omega]_{k\ell} &= [\wh\Omega]_{k\ell} \cdot \left( \1\{k\neq\ell, |[\wh\Omega]_{k\ell}|> \tau_n \} + \1\{k=\ell, [\wh\Omega]_{kk}> 1+\tau_n \} \right) \nonumber \\
& + \1\{k=\ell, [\wh\Omega]_{kk} \le 1+\tau_n \}
\label{eq:def_wt_Sigma}
\end{align}
for some $\tau_n\ge J_2 \kappa M s \lambda_n$; here $J_2$ is some absolute constant that is precisely introduced in Proposition~\ref{prop:Zhao14_Thm_IV.5}.  In words, to obtain $\wt\Omega$, we shrink the off-diagonal elements of $\wh\Omega$ toward zero, while shrink the diagonal elements of $\wh\Omega$ toward one.  The difference between the treatments of the diagonal and off-diagonal elements in (\ref{eq:def_wt_Sigma}) results from the consideration that we would like $\wt\Omega-I_d$, rather than $\wt\Omega$ itself, to be sparse, as should be the case if $\Omega-I_d$ is sparse, and the basic fact that the diagonal elements of an inverse correlation matrix are bounded below by one (instead of zero as is the case for the off-diagonal elements).  The property of the estimator $\wt\Omega$ will be discussed in more details in Section~\ref{sec_est_Omega}.

%----------------------------------------------------------------------------------------------------

\subsubsection{Estimation of $\beta^*$ and the copula part}
\label{sec:est_copula_part}

With our separate sparse estimators $\wt\Delta\alpha$ of $\Delta\alpha$ in Section~\ref{sec:sparse_estimation_Delta_alpha} and $\wt\Omega$ of $\Omega$ in Section~\ref{sec:sparse_estimation_Omega}, we now let $\wh\beta=(\wh\beta_1,\dots,\wh\beta_d)^T:\mathbb{R}^d\rightarrow\mathbb{R}^d$ defined as, for $x\in\mathbb{R}^d$,
\begin{align}
\label{eq:wh_beta}
\wh\beta(x) = (\wh\beta_1(x),\dots,\wh\beta_d(x))^T = \left( \wt\Omega - I_d \right) \wt\Delta\alpha(x)
\end{align}
be our sparse estimator of $\beta^* = \left( \Omega - I_d \right) \Delta\alpha$.  Then, finally, we let $(\wh\alpha_0 + \wh\alpha_1)^T \wh\beta:\mathbb{R}^d\rightarrow\mathbb{R}$ defined as, for $x\in\mathbb{R}^d$,
$$(\wh\alpha_0(x) + \wh\alpha_1(x) )^T \wh\beta(x)$$
be our estimator of the copula part $(\alpha_0+ \alpha_1)^T \beta^*$.
%as introduced in (\ref{eq:def_beta_star_x})

%----------------------------------------------------------------------------------------------------

\subsection{Estimation of the naive Bayes part}
\label{sec:main_result_Bayes_part}

%\subsubsection{Preliminaries}

\subsubsection{Construction of the kernel density estimator of $f_{i|y}$}
\label{sec:construction_kernel_density_estimator}

Recall from (\ref{eq:naive_Bayes_part}) that for the naive Bayes part we need to estimate
\begin{align}
\sum_{i=1}^d \Delta\log f_i(x_i) = \sum_{i=1}^d \left( \log f_{i|0}(x_i) - \log f_{i|1}(x_i) \right). \nonumber
\end{align}
(Recall that $f_{i|y}$, $i\in\{1,\dots,d\}$, $y\in\{0,1\}$ is the probability density function of the $i$th coordinate for class $y$.)  Hence, naturally, our estimation of the naive Bayes part will be based on the estimation of the density functions $f_{i|y}$, for which we opt to use kernel density estimators.

We let $K_i:\mathbb{R}\rightarrow\mathbb{R}$ be the kernel and $h_{n,i}$ be the bandwidth for the $i$th coordinate, and let $\wh f_{i|y}$ be the kernel density estimator of $f_{i|y}$, constructed from the $n$ samples $X^{y,j}_i$, $j\in\{1,\dots,n\}$:
\begin{align}
\wh f_{i|y}(t) = \dfrac{1}{n h_{n,i}} \sum_{j=1}^{n} K_i \left(\dfrac{X^{y,j}_i-t}{h_{n,i}}\right).
\label{eq:kernel_density_estimator_i_y}
\end{align}
In addition, we let $\wh f_i$ be the kernel density estimator of the marginal density $f_i$, constructed from the $2n$ samples $X^{y,j}_i$, $j\in\{1,\dots,n\}$, $y\in\{0,1\}$, with the same kernel and bandwidth:
\begin{align}
\wh f_i = \dfrac{1}{2} \left[ \wh f_{i|0} + \wh f_{i|1} \right].
\label{eq:kernel_density_estimator_i}
\end{align}
The specifics of the kernel $K_i$ and its order, the bandwidth $h_{n,i}$, as well as a quantity $\underline{f_{n,i}}$ that we need later, depend on the smoothness condition of $f_{i|y}$, and will be specified in details in Section~\ref{sec:choice_kernel_density_estimator}.  The impatient readers are encouraged to jump directly to Section~\ref{eq:sparse_estimation_naive_Bayes_part}.

%----------------------------------------------------------------------------------------------------

\subsubsection{Choosing the kernel, the bandwidth, and the quantity $\underline{f_{n,i}}$}
\label{sec:choice_kernel_density_estimator}

We will make the blanket assumption that we have at our disposal a sequence of kernels $\{K^{(l)},l\ge 1\}$ of varying orders, such that $K^{(l)}$ is a kernel of order $l$ and is constructed as in \cite[Proposition~1.3]{TsybakovBook09}.  Hence, the kernel $K^{(l)}$ is compactly supported on $[-1,1]$, and satisfies $\|K^{(l)}\|_{L^{\infty}}\le C_K\cdot l^{3/2}$ for an absolute constant $C_K$ independent of $l$ and $\|K^{(l)}\|_{L^2}^2\le l$.  Here and below, for a function $f:\mathbb{R}\rightarrow\mathbb{R}$, we denote $\|f\|_{L^p} = \left( \int_{\mathbb{R}} |f(t)|^{p} dt \right)^{1/p}$.  (We can substitute the sequence $\{K^{(l)}, l\ge 1\}$ by any other sequence of kernels that are compactly supported on $[-1,1]$ and that have comparable bound on the growth rate of $\|K^{(l)}\|_{L^{\infty}}$ and $\|K^{(l)}\|_{L^2}^2$ with $l$, although for concreteness we avoid such generalization.)

We will always choose the kernel $K_i$ from the sequence $\{K^{(l)},l\ge 1\}$.  We opt not to employ ``kernels of infinite order'' (e.g., \cite{Devroye92}), because such kernels don't have compact support, while the derivation of Inequality~(\ref{eq:f_rel_dev_variance}) in Proposition~\ref{prop:f_rel_dev_variance} requires a kernel with compact support to eliminate an extra factor $f_{i|y}$ in the exponent through condition (\ref{eq:t_in_B_f}).

As is typical in kernel density estimation, we assume that the density functions $f_{i|y}$ satisfy certain smoothness conditions.  For simplicity we assume that, for each $i\in\{1,\dots,d\}$, the two density functions $f_{i|y}$, $y\in\{0,1\}$ have comparable smoothness, and hence we use the same kernel and bandwidth for the two classes $y\in\{0,1\}$.

We will consider the canonical case of densities belonging to a H\"{o}lder class.  On the other hand, it may turn out that it is too restrictive to have a H\"{o}lder class characterize the smoothness of certain densities, such as Gaussian densities.  Here we consider one class of such densities, which we will call \text{super-smooth densities}, and obtain improved convergence rate and weakened assumption of their estimation (as compared to densities that merely belong to some H\"{o}lder class), if we allow the order of the kernel to increase with the sample size as $\lceil\log(n)\rceil$.  First we introduce our precise definition of super-smooth densities.

\begin{definition}[Super-smooth densities]
\label{def:super_smooth_densities}
We say the class of continuous density functions $\calF$ is super-smooth with respect to the sequence of constants $\{c_l, l\ge 1\}$ with $c_l\rightarrow 0$ as $l\rightarrow\infty$, if for any $f\in\calF$, any $t\in\mathbb{R}$ and any $l\ge 1$, the bias satisfies
\begin{align}
\left| \EE \left[ \wh f_{K^{(l)}}(t) \right] - f(t) \right| \le c_l h^l. \nonumber
\end{align}
Here $\wh f_{K^{(l)}}$ is the kernel density estimator of $f$ constructed using the kernel $K^{(l)}$ (or order $l$) and some arbitrary bandwidth $h$.
\end{definition}

Our next result shows that appropriate class of (univariate) Gaussian density functions are super-smooth.
\begin{proposition}
\label{prop:Gaussian_super_smooth}
The class $F_{\sigma_0^2}$, with $\sigma_0^2>0$, of Gaussian density functions with variance $\sigma^2$ bounded below by $\sigma_0^2$ is super-smooth with respect to the sequence of constants $c_l=\dfrac{C_{\textnormal{Cram\'{e}r}}\|K^{(l)}\|_{L^{\infty}}}{\sqrt{\pi/2}(l!)^{1/2} \sigma_0^{l+1} }$.  Here the absolute constant $C_{\textnormal{Cram\'{e}r}}<1.09$.
\end{proposition}

\begin{proof}
The proof can be found in Section~\ref{sec:proof_prop:Gaussian_super_smooth}.
\end{proof}

%We now specify the the order $l_i$ of the kernel $K_i$ and the value of the bandwidth $h_{n,i}$.

From now on, we make the blanket assumption that for each $i\in\{1,\dots,d\}$, the density functions $f_{i|y}$, $y\in\{0,1\}$ either belong to the same H\"{o}lder class, or to the same class of super-smooth densities, and we choose appropriate order $l_i$ of the kernel $K_i$ and the value of the bandwidth $h_{n,i}$ for their estimation, as well as the quantity $\underline{f_{n,i}}$, according to our specification below.  We define
\begin{align}
\label{eq:choice_epsilon}
\epsilon_n = (2J_1)^{-\frac{1}{2}} \left[\gamma\log(n)\right]^{\frac{3}{4}} n^{-\left(\frac{1}{2}-\frac{\gamma}{4}\right)}.
\end{align}
Here $J_1$ is the particular constant that appears in (\ref{eq:ineq_delta_master_1}).

We first consider the case where the density functions $f_{i|y}$, $y\in\{0,1\}$ merely belong to the H\"{o}lder class $\Sigma(\beta_i,L_i)$.  We set
\begin{align}
C_i &=\left(\dfrac{l!}{ 2^+ 2 L_i \|K_i\|_{L^{\infty}}}\right)^{1/\beta_i}, \nonumber \\
\label{eq:f_i_lower_bound}
\underline{f_{n,i}} &= \left( J_{\beta_i,\gamma,C_d} \cdot \dfrac{\max\left\{ 3 \|K_i\|_{L^{\infty}} \epsilon_n,\|K_i\|_{L^2}^2 \right\}}{C_i}\right)^{\frac{\beta_i}{\beta_i+1}} \cdot \log^{-\frac{2\beta_i+3}{4(\beta_i+1)}}(n) \cdot n^{-\left(-\frac{1}{2(\beta_i+1)}+\frac{2\beta_i+1}{\beta_i+1}\frac{\gamma}{4} \right)}.
\end{align}
Here $J_{\beta_i,\gamma,C_d}$ is a finite but large enough constant to ensure that Inequality~(\ref{eq:f_rel_dev_concrete}) in Theorem~\ref{thm:f_rel_dev} holds, and it depends only on $\beta_i$, $\gamma$ and $C_d$, for the constant $C_d$ to be introduced in Assumption~\ref{ass_d_n}.  Then, we let the kernel $K_i$ have order $l_i=\lfloor\beta_i\rfloor$, i.e. we let $K_i=K^{(l_i)}$, and let the bandwidth $h_{n,i}$ be (recall $\epsilon_n$ as defined in (\ref{eq:choice_epsilon}))
\begin{align}
\label{eq:h_n_i_concrete}
h_{n,i} = C_i \left( \epsilon_n \underline{f_{n,i}} \right)^{1/\beta_i}.
\end{align}

Alternatively, we assume that the density functions $f_{i|y}$, $y\in\{0,1\}$ belong to a class of super-smooth densities with respect to the sequence of constants $\{c_l,l\ge 1\}$.  We then let the order of the kernel $K_i$ to vary with the sample size $n$, and in particular we set $K_i=K_i(n)=K^{(\lceil\log(n)\rceil)}$.  We let the bandwidth $h_{n,i}$ be
\begin{align}
\label{eq:h_n_i_concrete_super_smooth}
h_{n,i} = H_i \log^{-\frac{1}{2}}(n)
\end{align}
for a constant $H_i$ satisfying
\begin{align}
\label{eq:H_i_condition_super_smooth}
H_i \le \log(2)/\sqrt{\gamma}.
\end{align}
We also set, in this case,
\begin{align}
\label{eq:f_i_lower_bound_super_smooth}
\underline{f_{n,i}} = J_{\gamma,C_d} \cdot H_i^{-1} \cdot \log(n) \cdot n^{-\frac{\gamma}{2}}.
\end{align}
Here again $J_{\gamma,C_d}$ is a finite but large enough constant to ensure that Inequality~(\ref{eq:f_rel_dev_concrete}) in Theorem~\ref{thm:f_rel_dev} holds, and it depends only on $\gamma$ and $C_d$.  Note that the dependence on $n$ in (\ref{eq:f_i_lower_bound}) is, up to a log factor in $n$, identical to the dependence on $n$ in (\ref{eq:f_i_lower_bound_super_smooth}) in the limit $\beta_i\rightarrow\infty$ but is slower for finite $\beta_i$, which implies that the condition required for the accurate estimation of $\Delta\log f_i$ in the H\"{o}lder case is stronger, as we will see in Sections~\ref{sec:kernel_density_estimation} and \ref{sec:sparse_estimation_Bayes_part}.

%\noindent
%{\em Remark:}

%----------------------------------------------------------------------------------------------------

\subsubsection{Sparse estimation of the naive Bayes part}
\label{eq:sparse_estimation_naive_Bayes_part}

We let $\wt\Delta \log f=(\wt\Delta \log f_1,\dots,\wt\Delta \log f_d)^T:\mathbb{R}^d\rightarrow\mathbb{R}^d$ with, for $x\in\mathbb{R}^d$,
\begin{align}
\wt\Delta \log f(x) = (\wt\Delta\log f_1(x_1), \dots, \wt\Delta\log f_d(x_d))^T
\label{eq:wt_Delta_log_f}
\end{align}
be our sparse estimator of $\Delta \log f$.  Analogous to the construction of $\wt\Delta\alpha$ in Section~\ref{sec:sparse_estimation_Delta_alpha}, the construction of $\wt\Delta \log f$ consists of two potential steps; we fix arbitrary $i\in\{1,\dots,d\}$ and arbitrary $t\in\mathbb{R}$:
\begin{enumerate}
\item
First, we check whether
\begin{align}
\wh f_{i}(t) \le 3 \underline{f_{n,i}}
\label{eq:density_test_1}
\end{align}
(Note that the test involves the marginal empirical density function $\wh f_{i}$.)  At the same time, we also check whether
\begin{align}
\left| \log \wh f_{i|0}(t) - \log \wh f_{i|1}(t) \right| \le \tilde\delta_{n,\gamma}.
\label{eq:density_test_2}
\end{align}
Here
\begin{align}
\tilde\delta_{n,\gamma} = 2\dfrac{\epsilon_n}{1-\epsilon_n}.
\label{eq:def_density_delta}
\end{align}
If either Inequality~(\ref{eq:density_test_1}) or Inequality~(\ref{eq:density_test_2}) holds, we set $\wt\Delta\log f_i(t)=0$.
\item
Otherwise (i.e., if both (\ref{eq:density_test_1}) and (\ref{eq:density_test_2}) are violated), we set
\begin{align}
\wt\Delta\log f_i(t) = \log \wh f_{i|0}(t) - \log \wh f_{i|1}(t).
\label{eq:def_wt_Delta_log_f_i_t}
\end{align}
\end{enumerate}

The basic intuition behind our two-step construction is analogous to that of the construction of $\wt\Delta\alpha$ in Section~\ref{sec:sparse_estimation_Delta_alpha} and is as follows.  First, test (\ref{eq:density_test_1}) checks whether the value of $f_{i}(t)$ is likely small.  If so, then the value of at least one of $f_{i|y}(t)$, $y\in\{0,1\}$ is also likely small, and hence the estimation of the corresponding $\log f_{i|y}(t)$ is likely poor (because the error when estimating the logarithm of the density is roughly scaled by the inverse of the density; see Section~\ref{sec:sparse_estimation_Bayes_part}).  In this case, we do not try to estimate $\Delta\log f_i(t)$ at all and so set $\wt\Delta\log f_i(t)=0$.  Next, test (\ref{eq:density_test_2}) checks whether the values of $\log f_{i|0}(t)$ and $\log f_{i|1}(t)$ are likely close, i.e., whether the signal strength is likely small.  If so, we again set $\wt\Delta\log f_i(t)=0$.  Otherwise we estimate $\Delta\log f_i(t)$ as in (\ref{eq:def_wt_Delta_log_f_i_t}) (as one normally would in the absence of sparsity).  The property of the estimator $\wt\Delta\log f$ will be discussed in more details in Section~\ref{sec:sparse_estimation_Bayes_part}.

%----------------------------------------------------------------------------------------------------

\subsection{Performance of the empirical decision rule $\wh\delta_n$, and discussion}
\label{sec:main_result_decision_rule}

We put together our estimators for the copula part and the naive Bayes part to construct $\wh{\log(f^0/f^1)}$, our estimator of the log density ratio $\log(f^0/f^1)$, as follows: for $x\in\mathbb{R}^d$, we let
\begin{align}
\wh{\log(f^0/f^1)}(x) = (\wh\alpha_0(x) + \wh\alpha_1(x) )^T \wh\beta(x) + \sum_{i=1}^d \wt\Delta \log f_i(x_i).
\label{eq:estimator_log_density}
\end{align}
Then, based on (\ref{eq:estimator_log_density}), our empirical classification rule $\wh\delta_n$ predicts $Y=1$ if and only if $\wh{\log(f^0/f^1)}(X)\ge 0$.

%We summarize in Corollary~\ref{corollary_master_error_bound} the pointwise performance of the estimator $\wh{\log(f^0/f^1)}$.  Before doing that,

We collect in Section~\ref{sec:assumptions} the relevant assumptions we need for the pointwise performance guarantee of the estimator $\wh{\log(f^0/f^1)}$.  Their necessity will only be explained in details later in Sections~\ref{sec_copula} and \ref{sec_naive_bayes}, and some of these assumptions are rather technical.  Hence, most readers may want to jump directly to Section~\ref{sec:excess_risk_bound}.

\subsubsection{Collection of assumptions}
\label{sec:assumptions}

The first assumption ensures the accurate estimation and support recovery of $\Omega-I_d$.
\begin{assumption}
\label{ass_Omega}
The precision matrix $\Omega$ satisfies $\Omega\in\calU$, for the class $\calU$ as defined in (\ref{eq:def_mathcal_U}).  In addition, for all $k,\ell\in\{1,\dots,d\}$ such that $k\neq\ell$, if $[\Omega]_{k\ell}\neq 0$, then $|[\Omega]_{k\ell}| > 2\tau_n$, while for all $k\in\{1,\dots,d\}$, if $[\Omega]_{kk}> 1$, then $[\Omega]_{kk} > 1 + 2\tau_n$.  (We recall $\tau_n$ as introduced in (\ref{eq:def_wt_Sigma}).)
\end{assumption}
We also assume that the dimension $d$ grows at most with a polynomial rate in $n$, as specified by Assumption~\ref{ass_d_n}.  (Although moderate exponential growth of $d$ with $n$ can be accommodated, in this paper we do not treat such situations in order to avoid complicated-looking exponent in $n$ when displaying convergence rates.)  We also impose in Assumption~\ref{ass_d_n} the condition that the product $\kappa s$ (recall the definitions of $\kappa, s$ from (\ref{eq:def_mathcal_U})) does not scale too rapidly with $n$, which simplifies certain bounds on convergence rates.
\begin{assumption}
\label{ass_d_n}
$d\le n^{C_d}$ for some absolute constant $C_d>0$, and $\kappa s \sqrt{\log(n)}\lambda_n = o(\epsilon_n)$.  (We recall $\lambda_n$ and $\epsilon_n$ as introduced in (\ref{eq:tuning_parameter_Omega}) and (\ref{eq:choice_epsilon}) respectively.)
\end{assumption}

The next four assumptions concern the location $x\in\mathbb{R}^d$ at which we can estimate the log density ratio $\log(f^0/f^1)(x)$ accurately.  Of these, the first two concern the estimation of the copula part and the remaining two concern the estimation of the naive Bayes part.

For the copula part, we define the sets
\begin{align}
\label{eq:B_n_gamma_i_y}
B_{n,\gamma,i,y} &= \left\{t: t~\text{satisfies Inequality}~(\ref{eq:F_t_large_2}) \right\}, \quad y\in\{0,1\} \\
\label{eq:B_delta_n_d_gamma_i}
B^{\delta}_{n,d,\gamma,i} &= \left\{t: t~\text{satisfies at least one of Inequalities}~(\ref{eq:F_ratio_1}), (\ref{eq:F_ratio_2}), (\ref{eq:F_ratio_3}), (\ref{eq:F_ratio_4}) \right\}
\end{align}
for the inequalities
\begin{align}
8 g(2n,\gamma) \le F_{i|y}(t) \le 1 - 8 g(2n,\gamma),
\label{eq:F_t_large_2}
\end{align}
and
\begin{eqnarray}
\label{eq:F_ratio_1}
\dfrac{ F_{i|0}(t) }{ F_{i|1}(t) } &>& \dfrac{ \left(1+\bar\delta_{n,d,\gamma}\right) \left(1+\bar\delta_{n,1,\gamma}\right) }{ \left(1-\bar\delta_{n,d,\gamma}\right) \left(1-\bar\delta_{n,1,\gamma}\right) }, \\
\label{eq:F_ratio_2}
\dfrac{ F_{i|1}(t) }{ F_{i|0}(t) } &>& \dfrac{ \left(1+\bar\delta_{n,d,\gamma}\right) \left(1+\bar\delta_{n,1,\gamma}\right) }{ \left(1-\bar\delta_{n,d,\gamma}\right) \left(1-\bar\delta_{n,1,\gamma}\right) }, \\
\label{eq:F_ratio_3}
\dfrac{ 1 - F_{i|0}(t) }{ 1 - F_{i|1}(t) } &>& \dfrac{ \left(1+\bar\delta_{n,d,\gamma}\right) \left(1+\bar\delta_{n,1,\gamma}\right) }{ \left(1-\bar\delta_{n,d,\gamma}\right) \left(1-\bar\delta_{n,1,\gamma}\right) }, \\
\label{eq:F_ratio_4}
\dfrac{ 1 - F_{i|1}(t) }{ 1 - F_{i|0}(t) } &>& \dfrac{ \left(1+\bar\delta_{n,d,\gamma}\right) \left(1+\bar\delta_{n,1,\gamma}\right) }{ \left(1-\bar\delta_{n,d,\gamma}\right) \left(1-\bar\delta_{n,1,\gamma}\right) }.
\end{eqnarray}
Here the constant $8$ in (\ref{eq:F_t_large_2}) is chosen for convenience, and $\bar\delta_{n,1,\gamma}$ is just $\bar\delta_{n,d,\gamma}$ as defined in (\ref{eq:def_bar_delta}) but with $d$ replaced by $1$, i.e.,
\begin{align}
\bar\delta_{n,1,\gamma} \colonequals \left[ 3 n^{-1} g^{-1}(2n,\gamma) \log(n^{\frac{\gamma}{2}}) \right]^{1/2}.
\label{eq:def_bar_delta_s_x_p}
\end{align}
Then, we define
\begin{align}
\label{eq:wt_A_F_1}
A^{F,1}_{n,d,\gamma} &= \left\{ x\in\mathbb{R}^d: \forall i\in S_x'', x_i \in B_{n,\gamma,i,0} \cap B_{n,\gamma,i,1} \right\}, \\
\label{eq:wt_A_F_2}
A^{F,2}_{n,d,\gamma} &= \left\{ x\in\mathbb{R}^d: \forall i\in S_x'', x_i \in B^{\delta}_{n,d,\gamma,i} \right\}, \\
\label{eq:wt_A_F_n_d_gamma}
A^F_{n,d,\gamma} &= A^{F,1}_{n,d,\gamma} \cap A^{F,2}_{n,d,\gamma}.
\end{align}
Next, we define
\begin{align}
A^F_{n,\beta^*,\gamma} = \{x\in\mathbb{R}^d: \forall i\in S_x', \forall y\in\{0,1\}, \alpha_{i|y}(x_i)\in [-a_n,a_n]\}.
\label{eq:wt_A_F_beta_n_gamma}
\end{align}
Our first two assumptions regarding $x\in\mathbb{R}^d$ are
\begin{assumption}
\label{ass_x}
$x\in\mathbb{R}^d$ satisfies $x\in A^F_{n,d,\gamma}$ .
\end{assumption}
\begin{assumption}
\label{ass_x_beta_star}
$x\in\mathbb{R}^d$ satisfies $x\in A^F_{n,\beta^*,\gamma}$ .
\end{assumption}
Essentially, when $x\in\mathbb{R}^d$ satisfies Assumption~\ref{ass_x}, then for $i\in S_x''$, where we have $\Delta\alpha_i(x_i)\neq 0$, the values of $F_{i|y}(x_i)$, $y\in\{0,1\}$ are moderate so that $\alpha_{i|y}(x_i)$, $y\in\{0,1\}$ can be estimated accurately, and the signal strength, i.e., the difference between $F_{i|0}(x_i)$ and $F_{i|1}(x_i)$, is large enough so that we do not mistaken $\Delta\alpha_i(x_i)$ to be zero.  Similarly, when $x\in\mathbb{R}^d$ satisfies Assumption~\ref{ass_x_beta_star}, then $\alpha_{i|y}(x_i)$, $y\in\{0,1\}$ can be estimated accurately at those coordinates $i\in S_x'$.

For the naive Bayes part, we define, for $h_{n,i}$ the bandwidth, $i\in\{1,\dots,d\}$ and $y\in\{0,1\}$, the sets
\begin{align}
B^f_{h_{n,i},i,y} = \Big\{ t\in\mathbb{R}: &\text{if}~f_{i|y}(t)<\underline{f_{n,i}},~\text{then}~\max_{t'\in\left[t-h_{n,i},t+h_{n,i}\right]} f_{i|y}(t') \le 2 \underline{f_{n,i}}; \nonumber \\
&\text{if}~f_{i|y}(t)\ge\underline{f_{n,i}},~\text{then}~\max_{t'\in\left[t-h_{n,i},t+h_{n,i}\right]} f_{i|y}(t') \le 2 f_{i|y}(t) \Big\}
\label{def_B_f_c}
\end{align}
and
\begin{align}
A^f_{n,i} = \left\{t\in\mathbb{R}: t~\text{satisfies Inequality}~(\ref{eq:density_t_large_2}) \right\}
\end{align}
for
\begin{align}
f_{i|y}(t) \ge \dfrac{3}{1-\epsilon_n} \underline{f_{n,i}}, \forall y\in\{0,1\}.
\label{eq:density_t_large_2}
\end{align}
Then, we define the sets
\begin{align}
\label{eq:wt_A_f_eq_n_d_gamma}
A^{f,=}_{n,d,\gamma} &= \left\{ x\in\mathbb{R}^d: \forall i\notin S^f_x, x_i\in \cap_{y\in\{0,1\}}B^f_{h_{n,i},i,y} \right\}, \\
\label{eq:wt_A_f_neq_n_d_gamma}
A^{f,\neq}_{n,d,\gamma} &= \left\{ x\in\mathbb{R}^d: \forall i\in S^f_x, x_i \in A^f_{n,i} \cap \left( \cap_{y\in\{0,1\}} B^f_{h_{n,i},i,y} \right) \right\}.
\end{align}
Our remaining two assumptions regarding $x\in\mathbb{R}^d$ are
\begin{assumption}
\label{ass_density_x_S_f_x_c}
$x\in\mathbb{R}^d$ satisfies $x\in A^{f,=}_{n,d,\gamma}$.
%$x\in\mathbb{R}^d$ satisfies that, for all $i\notin S^f_x$, we have $x_i\in B^f_{h_{n,i},i}$.
\end{assumption}

\begin{assumption}
\label{ass_density_x_S_f_x}
$x\in\mathbb{R}^d$ satisfies $x \in A^{f,\neq}_{n,d,\gamma}$.
\end{assumption}
Roughly speaking, when $x\in\mathbb{R}^d$ satisfies Assumptions~\ref{ass_density_x_S_f_x_c} and \ref{ass_density_x_S_f_x}, then $\wt\Delta\log f(x)$ is an accurate sparse estimator of $\Delta\log f(x)$.

%then the densities $f_{i|y}(x_i)$ or the log densities $\log f_{i|y}(x_i)$ for $i\in\{1,\dots,d\}$, $y\in\{0,1\}$ can be estimated accurately., and if furthermore $i\in S^f_x$, then the signal strength, i.e., the difference between $\log f_{i|0}(x_i)$ and $\log f_{i|1}(x_i)$, is large enough so that we do not mistaken $\Delta\log f_i(x_i)$ to be zero.

%----------------------------------------------------------------------------------------------------

\subsubsection{Bound on the excess risk}
\label{sec:excess_risk_bound}

We are now ready to state the pointwise performance of the estimator $\wh{\log(f^0/f^1)}$.  We define
\begin{align}
\Delta(x) = J_0 \left[ \|\beta^*(x)\|_{\ell_1}  + s_x' M \sqrt{\log(n)} + s_x^f \right] \log^{\frac{3}{4}}(n) n^{-\left(\frac{1}{2}-\frac{\gamma}{4}\right)}.
\label{eq:Delta}
\end{align}
Here $J_0$ is a finite but large enough absolute constant to ensure that Inequality~(\ref{eq:master_error_bound}) in Corollary~\ref{corollary_master_error_bound} holds.

\begin{corollary}
\label{corollary_master_error_bound}
Suppose that Assumptions~\ref{ass_Omega} and \ref{ass_d_n} holds, and that $n$ is large enough.  Suppose that an arbitrary $x\in\mathbb{R}^d$ satisfies Assumptions~\ref{ass_x}, \ref{ass_x_beta_star}, \ref{ass_density_x_S_f_x_c} and \ref{ass_density_x_S_f_x}.  Then, on an event $L$ with
\begin{align}
\PP\left( L \right) \ge 1 - (6 s_x' + 9 s_x'' + 11) n^{-\gamma/2},
\end{align}
we have, for $\Delta(x)$ as defined in (\ref{eq:Delta}),
\begin{align}
&\left|\left[\wh{\log(f^0/f^1)} - \log(f^0/f^1)\right](x)\right| \le \Delta(x).
\label{eq:master_error_bound}
\end{align}
\end{corollary}

\begin{proof}
From the construction of $\wh{\log(f^0/f^1)(x)}$ as in (\ref{eq:estimator_log_density}), we have
\begin{align}
&\left|\left[\wh{\log(f^0/f^1)} - \log(f^0/f^1)\right](x)\right| \nonumber \\
& \le \left| \left[ (\wh\alpha_0 + \wh\alpha_1 )^T \wh\beta - (\alpha_0 + \alpha_1 )^T \beta^*\right](x) \right| + \left\| \left[\wt \Delta\log f - \Delta\log f\right](x) \right\|_{\ell_1}.
\label{eq:wh_eta_error}
\end{align}
We let $L=L^{\text{copula}}_{x,n} \cap L^{\text{bayes}}_{x,n}$, for the events $L^{\text{copula}}_{x,n}$ introduced in (\ref{eq:def_L_copula}) and $L^{\text{bayes}}_{x,n}$ introduced in (\ref{eq:L_n_b_x_n_gamma}).  The corollary then follows straightforwardly from Inequality~(\ref{eq:wh_eta_error}), Corollary~\ref{cor_est_copula_part} and Theorem~\ref{thm:est_naive_bayes_part}.  (Note that for $n$ large enough we are free to replace $\tilde\delta_{n,\gamma}$ by $C \epsilon_n$, which is in turn bounded as in (\ref{eq:choice_epsilon_simple}).)
\end{proof}

%----------------------------------------------------------------------------------------------------

Because Corollary~\ref{corollary_master_error_bound} states a deviation inequality for the estimator $\wh{\log(f^0/f^1)}$ of the log density ratio, we can straightforwardly calculate the excess risk, defined in (\ref{eq:excess_risk}), associated with the empirical decision rule $\wh\delta_n$.  Moreover, by the same reason, we can easily incorporate the margin assumption, introduced in \cite{Audibert07}, to take advantage of the potential low noise condition in the problem.  We state a slight variant of the margin assumption from \cite[Relationship~(1.7)]{Audibert07} in terms of the log density ratio instead of the regression function, which is more suited for our Gaussian copula classification framework.
\begin{assumption}[The margin assumption] There exist constants $C_0>0$ and $\alpha\ge 0$ s.t.
\label{assumption:margin_condition}
\[
\PP(0<|\log(f^0/f^1)(X)|\le t)\le C_0 t^{\alpha}, \forall t>0.
\]
\end{assumption}

As a concrete example, in the canonical case of classifying two Gaussian distributions with the same covariance, the margin assumption is fulfilled with $\alpha=1$, e.g., see Appendix~\ref{sec:margin_assumption_Gaussian}.

We define the set of $x\in\mathbb{R}^d$ simultaneously satisfying Assumptions~\ref{ass_x}, \ref{ass_x_beta_star}, \ref{ass_density_x_S_f_x_c} and \ref{ass_density_x_S_f_x} as
\begin{align}
A_{n,d,\gamma} = A^{F}_{n,d,\gamma} \cap A^F_{n,\beta^*,\gamma} \cap A^{f,=}_{n,d,\gamma} \cap A^{f,\neq}_{n,d,\gamma}
\label{eq:wt_A_n_d_gamma}
\end{align}
(for $A^{F}_{n,d,\gamma}$, $A^F_{n,\beta^*,\gamma}$, $A^{f,=}_{n,d,\gamma}$, $A^{f,\neq}_{n,d,\gamma}$ as in (\ref{eq:wt_A_F_n_d_gamma}),  (\ref{eq:wt_A_F_beta_n_gamma}), (\ref{eq:wt_A_f_eq_n_d_gamma}) and (\ref{eq:wt_A_f_neq_n_d_gamma}) respectively).  We also state one more piece of assumption under which we can simplify our bound on the excess risk to be presented in Theorem~\ref{thm:excess_risk}.

\begin{assumption}
\label{assumption:fixed_sets}
For all $x\in\mathbb{R}^d$, the cardinalities of $S_x'$, $S_x''$ and $S^f_x$, i.e., $s_x'$, $s_x''$ and $s^f_x$, are upper bounded by constants $s'$, $s''$ and $s^f$ respectively, and $\|\beta^*(x)\|_{\ell_1}$ is upper bounded by a constant $C_{\beta^*}$.
\end{assumption}

\begin{theorem}
\label{thm:excess_risk}
Suppose that Assumptions~\ref{ass_Omega}, \ref{ass_d_n}, and the margin assumption \ref{assumption:margin_condition} hold, and that $n$ is large enough.  Then the excess risk satisfies
\begin{align}
\PP(\wh\delta_n(X)\neq Y)-\PP(\delta^*(X)\neq Y) & \le \PP\left(X\notin A_{n,d,\gamma}\right) + \EE\left[6 s_X' + 9 s_X'' + 11 \right] n^{-\gamma/2} \nonumber \\
& + \dfrac{1}{2} \EE\left[ \Delta(X) \1\left\{ \left|\log(f^0/f^1)(X)\right|\le \Delta(X) \right\}  \right].
\label{eq:excess_risk_1}
\end{align}
Hence, if in addition Assumption~\ref{assumption:fixed_sets} holds, then the excess risk satisfies
\begin{align}
&\PP(\wh\delta_n(X)\neq Y)-\PP(\delta^*(X)\neq Y) \le \PP\left(X\notin A_{n,d,\gamma}\right) + (6 s'+9 s''+11) n^{-\frac{\gamma}{2}} \nonumber \\
&+  \dfrac{C_0}{2} \left\{ J_0 \left[ C_{\beta^*} + s' M \sqrt{\log(n)} + s^f \right] \log^{\frac{3}{4}}(n) n^{-\left(\frac{1}{2}-\frac{\gamma}{4}\right)} \right\}^{\alpha+1}.
\label{eq:excess_risk_2}
\end{align}
\end{theorem}

\begin{proof}
The proof can be found in Section~\ref{sec:proof_thm:excess_risk}.
\end{proof}

We elaborate on the results presented in Theorem~\ref{thm:excess_risk}.  We note that without the term $\PP\left(X\notin A_{n,d,\gamma}\right)$ in (\ref{eq:excess_risk_1}) and (\ref{eq:excess_risk_2}), we can choose $\gamma$ to optimize the convergence rate with respect to $n$.  For instance, on the right hand side of (\ref{eq:excess_risk_2}), the last term scales with $n$ as $n^{-\left(\frac{1}{2}-\frac{\gamma}{4}\right)(\alpha+1)}$ up to log factors.  To match this convergence rate in $n$ with that of the second term on the right hand side of (\ref{eq:excess_risk_2}) (up to log factors), we can choose $\gamma=(2\alpha+2)/(\alpha+3)$ so that the last two terms on the right hand side of (\ref{eq:excess_risk_2}) both scale with $n$ as $n^{-(\alpha+1)/(\alpha+3)}$ (up to log factors).  Therefore, for $\alpha=1$, we achieve a convergence rate of $n^{-1/2}$, while for larger values of $\alpha$, we obtain a convergence rate faster than $n^{-1/2}$.

This leaves us the task of bounding the first term on the right hand side of (\ref{eq:excess_risk_2}), namely the term $\PP\left(X\notin A_{n,d,\gamma}\right)$.  The collection $A_{n,d,\gamma}^c$, the complement of (\ref{eq:wt_A_n_d_gamma}), is the set on which it is difficult to estimate the log density ratio accurately.  This set is explicitly dependent on the particular distribution functions and the density functions of $(X|Y=0)$ and $(X|Y=1)$.  Hence, we cannot explicitly calculate the term $\PP\left(X\notin A_{n,d,\gamma}\right)$ unless we specify explicit distributions, although we can easily establish a lower bound on this term that scales with $n$ as $g(2n,\gamma)$ for all possible distributions (e.g., through the set $A^{F,1}_{n,d,\gamma}$ as defined in (\ref{eq:wt_A_F_1})), and it is straightforward to construct a toy example where this lower bound is achieved.

To demonstrate a concrete upper bound on the term $\PP\left(X\notin A_{n,d,\gamma}\right)$, we consider in Section~\ref{sec:case_study} the canonical case of classifying two Gaussian distributions with the same covariance but different means, specifically under the scenario stated in Definition~\ref{def:simple_Gaussian}.  Then, we have
\begin{align}
\PP\left(X\notin A_{n,d,\gamma}\right) \le C_{\gamma,C_d,\mu} (s'+s'') e^{C \mu \sqrt{\gamma\log(n)}} g(2n,\gamma).
\label{eq:case_study_P_A}
\end{align}
Here $C_{\gamma,C_d,\mu}$ is some constant dependent only on $\gamma,C_d,\mu$, and we refer the readers to Section~\ref{sec:case_study} for the exact meanings of the parameters $\mu$, $s'$ and $s''$ in (\ref{eq:case_study_P_A}).  Thus, the convergence rate of the term $\PP\left(X\notin A_{n,d,\gamma}\right)$ with respect to $n$ is just slightly slower than that of the second term on the right hand side of (\ref{eq:excess_risk_2}) (we note that $e^{C\mu\sqrt{\gamma \log(n)}}=o(n^{\varepsilon})$ for all $\varepsilon>0$).  As stated following Assumption~\ref{assumption:margin_condition}, here the margin assumption is fulfilled with $\alpha=1$, and so we choose $\gamma=(2\alpha+2)/(\alpha+3)=1$ as discussed earlier.  Then, in this particular scenario, the excess risk associated with our empirical decision rule $\wh\delta_n$ based on semiparametric method achieves a convergence rate of $e^{C \mu \sqrt{\gamma\log(n)}} n^{-1/2}$ with respect to $n$, which is nearly the rate of $n^{-1/2}$.

%At the other extreme, this term equals one for $(X|Y=0)$ and $(X|Y=1)$ that lack common support, in which case the classification problem is arguably trivial but our bound (\ref{eq:excess_risk_2}) becomes vacuous.

%----------------------------------------------------------------------------------------------------

\subsection{Case study: Gaussian distribution classification}
\label{sec:case_study}

In this section we assume that $(X,Y)$ follows a simple model, which we will casually refer to as the simple $(d,s',\mu,\Sigma)$ Gaussian classification model and which is described in Definition~\ref{def:simple_Gaussian}.  We will calculate the term $\PP\left(X\notin A_{n,d,\gamma}\right)$ explicitly under this model, and state our result in Theorem~\ref{thm:exclusion_set_simple_Gaussian}.
\begin{definition}
\label{def:simple_Gaussian}
We let $Z_1,\dots,Z_d$ be $d$ standard normal random variables with correlation matrix $\Sigma$.  We fix some $1\le s''\le d$ and some $\mu\in\mathbb{R}^+$.  We say that $(X,Y)\in\mathbb{R}^d\times\{0,1\}$ is a simple $(d,s'',\mu,\Sigma)$ Gaussian classification model if $(X|Y=0)\stackrel{d}{=}(Z_1,\dots,Z_d)^T$ and $(X|Y=1)\stackrel{d}{=}(Z_1+\mu,\dots,Z_{s''}+\mu,Z_{s''+1},\dots,Z_d)^T$.
\end{definition}

Under the simple $(d,s'',\mu,\Sigma)$ Gaussian classification model, the marginal distributions of $(X|Y=0)$ and $(X|Y=1)$ are identical except for the first $s''$ coordinates, and $\Delta\alpha$ is a constant function that returns a vector with the first $s''$ components equal to $\mu$ and the remaining components equal to zero.  We let $S''=\{1,\dots,s''\}$, which has cardinality $s''$.  Then, for all $x\in\mathbb{R}^d$, $S_x'' = S''$ (for $S_x''$ as defined in (\ref{eq:S_x_p})) and $S^f_x\subset S''$ (for $S^f_x$ as defined in (\ref{eq:S_f_x})); in addition, $S_x'$ (as defined in (\ref{eq:def_bar_S_x})) is a constant set $S'$, which we assume has cardinality $s'$.

We also recall that, because Gaussian densities are super-smooth densities, for all $i\in\{1,\dots,d\}$ and all $y\in\{0,1\}$, the density function $f_{i|y}$ is estimated with the kernel $K_i=K_i(n)=K^{(\lceil\log(n)\rceil)}$ and with the bandwidth $h_{n,i}$ as in (\ref{eq:h_n_i_concrete_super_smooth}), and additionally the quantity $\underline{f_{n,i}}$ is chosen according to (\ref{eq:f_i_lower_bound_super_smooth}), as we discussed in Section~\ref{sec:choice_kernel_density_estimator}.

\begin{theorem}
\label{thm:exclusion_set_simple_Gaussian}
Suppose that Assumption~\ref{ass_d_n} holds.  Under the simple $(d,s'',\mu,\Sigma)$ Gaussian classification model, for $n$ large enough, Inequality~(\ref{eq:case_study_P_A}) holds.
\end{theorem}

\begin{proof}
The proof can be found in Section~\ref{sec:proof_thm:exclusion_set_simple_Gaussian}.
\end{proof}

Therefore, as explained in details in the discussion following Theorem~\ref{thm:excess_risk}, for classifying two Gaussian distributions under the simple $(d,s'',\mu,\Sigma)$ Gaussian classification model, the excess risk associated with our empirical decision rule $\wh\delta_n$ nearly achieves the rate of $n^{-1/2}$.

%----------------------------------------------------------------------------------------------------

\section{Detailed study of the copula part}
\label{sec_copula}

\subsection{Outline}
In Section~\ref{sec_est_tran}, we study the estimation of the transformation functions $\alpha_{i|y}$.  This serves as one of the building blocks for our sparse estimation of $\Delta\alpha$ in Section~\ref{sec_est_Delta_alpha}, which in turn elaborates our earlier Section~\ref{sec:sparse_estimation_Delta_alpha}.  In Section~\ref{sec_est_Omega}, we elaborate our earlier Section~\ref{sec:sparse_estimation_Omega}.  Sections~\ref{sec_est_Delta_alpha} and \ref{sec_est_Omega} combined lead to our estimation of $\beta^*$ in Section~\ref{sec_est_beta} and further the copula part in Section~\ref{sec_est_copula_part}, elaborating our earlier Section~\ref{sec:est_copula_part}.

%----------------------------------------------------------------------------------------------------

\subsection{Estimation of the transformation function $\alpha_{i|y}$}
\label{sec_est_tran}

Recall $\alpha_{i|y}$ as defined in (\ref{def_alpha_y_i}) and its estimate $\wh\alpha_{i|y}$ as defined in (\ref{def_wh_alpha_y_i}), and $a_n$ as defined in (\ref{eq:a_n}).  In this section we provide a tight, pointwise deviation inequality of $|\wh\alpha_{i|y}(t)-\alpha_{i|y}(t)|$ for $t$ over the interval $[-a_n,a_n]$ on $\mathbb{R}$ that expands with $n$.

\begin{lemma}
\label{thm:delta_alpha}
Let $0<\epsilon\le\sqrt{2\pi}$ but otherwise be arbitrary.  Then, for all $t\in\mathbb{R}$ such that $\alpha_{i|y}(t)\in [-a_n,a_n]$, we have
\begin{align}
&\PP(|\wh\alpha_{i|y}(t)-\alpha_{i|y}(t)|\ge\epsilon) \nonumber \\
\label{eq:ineq_delta_alpha_1}
&\le 2 \exp\left( -\dfrac{\min\left\{F_{i|y}(t),1-F_{i|y}(t)\right\}}{6\pi} n \epsilon^2  \right) + 6\log(g^{-1}(n,\gamma)/2) \exp\left( -\dfrac{1}{32}n\cdot g(n,\gamma) \right) \\
\label{eq:ineq_delta_master_1}
&\le 2 \exp\left( - J_1 \dfrac{n^{1-\gamma/2} \epsilon^2}{\sqrt{\gamma\log n}}  \right) + 6\log(g^{-1}(n,\gamma)/2) \exp\left( -\dfrac{1}{32}n\cdot g(n,\gamma) \right).
\end{align}
Here $J_1$ is some absolute constant which we can take to be $J_1=(12\pi\sqrt{2\pi})^{-1}$.
\end{lemma}

\begin{proof}%[Proof of Lemma~\ref{thm:delta_alpha}]
The proof can be found in Section~\ref{sec:proof_thm:delta_alpha}.
\end{proof}

We elaborate on the results presented in Lemma~\ref{thm:delta_alpha}.  First, we observe from (\ref{eq:ineq_delta_alpha_1}) that the estimator $\wh\alpha_{i|y}(t)$ of $\alpha_{i|y}(t)$ is the most accurate when the value of $F_{i|y}(t)$ is moderate, i.e., close to $1/2$ instead of close to $0$ or $1$.  Next, we compare our Lemma~\ref{thm:delta_alpha} to some related results in existing literature, in particular \cite[Theorem~2]{Han13} and \cite[Lemma~1]{Mai13}.  Both these results are, roughly speaking, versions of our Inequality~(\ref{eq:ineq_delta_master_1}), but for $t$ uniformly over the interval $t:\alpha_{i|y}(t)\in [-a_n,a_n]$, instead of our pointwise result.  The advantage of our result is that it is a tight deviation inequality, which will allow us to straightforwardly calculate the excess risk and incorporate the margin assumption in Section~\ref{sec:risk_bound}.  This is in contrast to the convergence in probability result in \cite[Theorem~2]{Han13}, and our pointwise convergence rate is distinctly faster than that implied by \cite[Lemma~1]{Mai13}.  Moreover, our proof of Lemma~\ref{thm:delta_alpha} can be easily modified to obtain a version of our Inequality~(\ref{eq:ineq_delta_master_1}) that is uniform over $t:\alpha_{i|y}(t)\in [-a_n,a_n]$.  Because Inequality~(\ref{eq:ineq_delta_master_1}) already suffices for our purpose and offers a somewhat faster convergence rate than the uniform version, we leave the detailed derivation of the latter to future studies.

%, if we incorporate \cite[Chapter~11, Section~2, Corollary~1]{Shorack09} (which we already invoked when proving Lemma~\ref{lemma:P_E_F}, indicating how such derivation could work).

%Lemma~\ref{thm:delta_alpha} does not address the estimation of $\alpha_{i|y}$ over the regime specified by $t:\alpha_{i|y}(t)\in [-a_n,a_n]^c$.  By Proposition~\ref{prop:A_n_prob_bound}, up to a log factor in $n$, this region has probability proportional to $n^{-\gamma/2}$ with respect to the random variable $(X_i|Y=y)$.  We will loosely refer to the rate $n^{-\gamma/2}$ as the ``exclusion probability,'' and will match some other probability bounds to this rate in the rest of the paper.

%----------------------------------------------------------------------------------------------------

\subsection{Estimation of $\Delta\alpha$ in a sparse setting}
\label{sec_est_Delta_alpha}

For $x\in\mathbb{R}^d$, we let $\wt S_x''$ be the estimator of $S_x''$ (defined in (\ref{eq:S_x_p})) based on the estimator $\wt\Delta\alpha$, introduced in Section~\ref{sec:sparse_estimation_Delta_alpha}, of $\Delta\alpha$, and $\wt s_x''$ be its cardinality, that is,
\begin{align}
\wt S_x'' = \{i:\wt\Delta\alpha_i(x_i)\neq 0\}, \quad \wt s_x'' = |\wt S_x''|.
\label{eq:wh_S_x_p}
\end{align}

We discuss our estimator $\wt\Delta\alpha$ separately for the case $i\notin S_x''$, i.e., $\Delta\alpha_i(x_i) = 0$ and so $F_i(x_i)=F_{i|0}(x_i) = F_{i|1}(x_i)$ by (\ref{eq:def_alpha_x}) and (\ref{eq:def_Delta_alpha_x}), and the case $i\in S_x''$.  We define, for $i\in\{1,\dots,d\}$ and $t\in\mathbb{R}$, the event
\begin{align}
H_{i,t} = \left\{ \wt\Delta\alpha_i(t) \neq 0 \right\}.
\label{eq:def_H_i_t}
\end{align}

%----------------------------------------------------------------------------------------------------

\subsubsection{The case $i\notin S_x''$}

We show in Theorem~\ref{thm:S_x_p_c_est} that, with high probability, and for all $x\in\mathbb{R}^d$, we correctly identify all components of $\Delta\alpha(x)$ that are zero.  We first state a weak condition on the sample size $n$ in Assumption~\ref{ass_n}, which is technical and is in place to facilitate our presentation.
\begin{assumption}
\label{ass_n}
$n$ satisfies
\begin{align}
\max\Bigg\{ & \dfrac{d}{8}\cdot \exp\left( - 3 n \cdot g(2n,\gamma) \right), 4 \exp\left( - n \cdot g(2n,\gamma) \right), \nonumber \\
&6\log(g^{-1}(n,\gamma)/2) \exp\left( -\dfrac{1}{32}n\cdot g(n,\gamma) \right) \Bigg\} \le n^{-\gamma/2}.
\label{eq:thm:S_x_p_c_est_n_condition}
\end{align}
\end{assumption}

\begin{theorem}
\label{thm:S_x_p_c_est}
%Recall the definitions of $S_x''$ and $\wt S_x''$ from (\ref{eq:S_x_p}) and (\ref{eq:wh_S_x_p}).
Suppose that Assumption~\ref{ass_n} holds.  For all $x\in\mathbb{R}^d$ and all $i\notin S_x''$, we have
\begin{align}
\PP(H_{i,x_i}^c) \ge 1 - 8\dfrac{1}{d} n^{-\gamma/2}.
\label{eq:H_i_prob_S_x_p_c_individual}
\end{align}
Hence, by the union bound, for all $x\in\mathbb{R}^d$, we have
\begin{align}
\PP(\cap_{i\notin S_x''} H_{i,x_i}^c) \ge 1 - 8 n^{-\gamma/2}.
\label{eq:H_i_prob_S_x_p_c_union}
\end{align}
\end{theorem}

\begin{proof}
The proof can be found in Section~\ref{sec:proof_thm:S_x_p_c_est}.
\end{proof}

%----------------------------------------------------------------------------------------------------

\subsubsection{The case $i\in S_x''$}

We show in Theorem~\ref{thm:S_x_p_est_1} that, with high probability, under Assumption~\ref{ass_x} on the distribution functions at $x\in\mathbb{R}^d$, we also correctly identify all the nonzero components of $\Delta\alpha(x)$.  Then, combined with Theorem~\ref{thm:S_x_p_c_est}, Theorem~\ref{thm:S_x_p_est_1} presents the performance guarantee of our sparse estimator $\wt\Delta\alpha$ of $\Delta\alpha$.  We define the event $H'_{x,\epsilon}$
\begin{align}
\label{eq:H_x_p}
H'_{x,\epsilon} &= \{\wt S_x'' = S_x''\} \nonumber \\
&\cap \left( \cap_{i\in S_x''} \left( \left( \cap_{y\in\{0,1\}} \left\{ |\wh\alpha_{i|y}(x_i)-\alpha_{i|y}(x_i)|<\epsilon \right\} \right) \cap \{ |\wt\Delta\alpha_i(x_i) - \Delta\alpha_i(x_i) | < 2\epsilon \}\right) \right).
\end{align}

Here we record some simple observations regarding Assumption~\ref{ass_x}.  It is trivial to see that at most one of the two Inequalities~(\ref{eq:F_ratio_1}) and (\ref{eq:F_ratio_2}) holds, and at most one of the two Inequalities~(\ref{eq:F_ratio_3}) and (\ref{eq:F_ratio_4}), but for brevity of presentation we do not emphasize this point.  It is also easy to see from (\ref{eq:F_edge_bound}) (for $t$ such that $\alpha_{i|y}(t)=a_n$) and its mirror version (for $t$ such that $\alpha_{i|y}(t)=-a_n$) that, for $y\in\{0,1\}$,
\begin{align}
B_{n,\gamma,i,y} \subset \{t\in\mathbb{R}:\alpha_{i|y}(t)\in [-a_n,a_n]\}.
\label{eq:B_subset_A_n}
\end{align}
Hence, Lemma~\ref{thm:delta_alpha} on the estimation of $\alpha_{i|y}(t)$ by $\wh\alpha_{i|y}(t)$ applies for $t\in B_{n,\gamma,i,y}$.
%To see this, note that on the one hand we have, for $t$ such that $\alpha_{i|y}(t)=-a_n$, we have
%\begin{align}
%g(n,\gamma) \le \Phi(\alpha_{i|y}(t)) = F_{i|y}(t) \le 2g(n,\gamma), \nonumber
%\end{align}
%while on the other hand we have
%\begin{align}
%8g(2n,\gamma) =\dfrac{4\cdot2^{-\gamma/2}}{\sqrt{2\pi}}\dfrac{n^{-\gamma/2}}{\sqrt{\gamma\log(2n)}}\ge\dfrac{2}{\sqrt{2}\sqrt{2\pi}}\dfrac{n^{-\gamma/2}}{\sqrt{\gamma\log(n)}} \ge 2g(n,\gamma). \nonumber
%\end{align}

\begin{theorem}
\label{thm:S_x_p_est_1}
Suppose that Assumption~\ref{ass_n} holds and that an arbitrary $x\in\mathbb{R}^d$ satisfies Assumption~\ref{ass_x}.  Then, for all $i\in S_x''$, we have
\begin{align}
\label{eq:H_i_prob_S_x_p_individual}
\PP(H_{i,x_i}) \ge 1 - 3 n^{-\gamma/2}.
\end{align}
Hence, by the union bound and Theorem~\ref{thm:S_x_p_c_est}, we conclude that
\begin{align}
\PP\left(\wt S_x'' = S_x''\right) \ge 1 - (3 s_x'' + 8) n^{-\gamma/2}.
\label{eq:S_x_p_consistency}
\end{align}
Furthermore, the event $H'_{x,\epsilon}$ introduced in (\ref{eq:H_x_p}) satisfies
\begin{align}
\label{eq:H_x_p_prob}
\PP (H'_{x,\epsilon}) \ge 1 - (5 s_x'' +8) n^{-\gamma/2} - 4 s_x'' \exp\left( - J_1 \dfrac{n^{1-\gamma/2} \epsilon^2}{\sqrt{\gamma\log n}} \right).
\end{align}
\end{theorem}

\begin{proof}
The proof can be found in Section~\ref{sec:proof_thm:S_x_p_est_1}.
\end{proof}

%----------------------------------------------------------------------------------------------------

\subsection{Sparse estimation of $\Omega$}
\label{sec_est_Omega}

For the estimator $\wh\Omega$ of $\Omega$ introduced in Section~\ref{sec:sparse_estimation_Omega}, we have the following proposition, which is a slight variant of \cite[Theorem~IV.5]{Zhao14}.
\begin{proposition}
\label{prop:Zhao14_Thm_IV.5}
Suppose that $\Omega\in\calU$, and $\kappa s \lambda_n\rightarrow 0$.  Then, there exists an event $E_n$, with
\begin{align}
\PP\left( E_n \right) \ge 1 - n^{-\gamma/2},
\label{eq:P_Omega_estimation}
\end{align}
and some absolute constant $J_2$ such that, for $n$ large enough, on the event $E_n$ we have
\begin{align}
\|\wh\Omega-\Omega\|_{\infty} \le J_2 \kappa M s \lambda_n.
\label{eq:Omega_estimation_bound}
\end{align}
\end{proposition}

\begin{proof}
By slightly modifying the argument leading to \cite[Inequality~(4.26)]{WZ14}, we have
\begin{align}
\PP\left(\|\wh\Sigma-\Sigma\|_{\max}\ge \lambda_n \right) \le \sum_{y\in\{0,1\}}\PP\left(\|\wh\Sigma^y-\Sigma\|_{\max}\ge \lambda_n \right) \le n^{-\gamma/2}.
\label{eq:P_wh_Sigma_max_norm}
\end{align}
The rest of the proof follows from the proof of \cite[Theorem~IV.5]{Zhao14}.  (In fact, the necessary proof here is simpler because $\Sigma$ is a correlation matrix with unit diagonal.)
\end{proof}

For the rest of this paper we fix the event $E_n$ and the absolute constant $J_2$ as the ones appearing in Proposition~\ref{prop:Zhao14_Thm_IV.5}.  We now state the estimation and support recovery guarantees of $\wt\Omega$, the thresholded version of $\wh\Omega$ introduced in (\ref{eq:def_wt_Sigma}), in Proposition~\ref{thm_Omega}.
\begin{proposition}
\label{thm_Omega}
Suppose that Assumptions~\ref{ass_Omega} and \ref{ass_d_n} hold.  Then, on the event $E_n$ (whose probability satisfies Inequality~(\ref{eq:P_Omega_estimation})), for $n$ large enough,
\begin{align}
\|\wt\Omega-\Omega\|_{\infty} &\le J_2 \kappa M s \lambda_n, \nonumber \\
\sgn(\wt\Omega - I_d) &=\sgn(\Omega - I_d) \nonumber
\end{align}
hold simultaneously.  Here the sign function acts component-wise.
\end{proposition}

\begin{proof}
With the condition on the growth rate of $\kappa s \lambda_n$ imposed by Assumption~\ref{ass_d_n}, $\kappa s \lambda_n\rightarrow 0$ as is required by Proposition~\ref{prop:Zhao14_Thm_IV.5}.  The conclusions of the proposition follow immediately from Proposition~\ref{prop:Zhao14_Thm_IV.5} and Assumption~\ref{ass_Omega}.
\end{proof}

We mention here that recent study from \cite{Ren13} provides very strong result on the estimation of \textit{individual entries} (rather than through matrix norm) of $\Omega$ under the Gaussian setting, which, as noted in \cite{Cai14}, leads to much weakened assumption on $\Omega$ for accurate support recovery.  (We also mention the result from \cite{Ravikumar11} on the estimation of individual entries of $\Omega$; this result can take the empirical Kendall's tau matrix as input, but at the same times requires strong irrepresentability condition on the Hessian matrix $\Sigma\otimes\Sigma$.)  In fact, as can be seen from the sparse estimation of $\left( \Omega - I_d \right) \Delta\alpha$ which we will undertake in Section~\ref{sec_est_beta}, we only need to estimate accurately, within the matrix $\Omega$, the entries of the rows $[\Omega]_{i\cdot}$, $i\in\{1,\dots,d\}$, whose locations correspond to the set $S_x''$ (and we already have an accurate estimator $\wt S_x''$ of $S_x''$ as demonstrated in Section~\ref{sec_est_Delta_alpha}).  We leave the potential generalization of \cite{Ren13} and related methods to the semiparametric Gaussian copula setting to future studies.

%----------------------------------------------------------------------------------------------------

\subsection{Sparse estimation of $\left( \Omega - I_d \right) \Delta\alpha(x)$}
\label{sec_est_beta}

With our separate sparse estimators $\wt\Delta\alpha$ of $\Delta\alpha$ and $\wt\Omega$ of $\Omega$ as constructed in Sections~\ref{sec:sparse_estimation_Delta_alpha} and \ref{sec:sparse_estimation_Omega}, and their properties described in Sections~\ref{sec_est_Delta_alpha} and \ref{sec_est_Omega}, we recall that $\wh\beta$, introduced in (\ref{eq:wh_beta}), is our sparse estimator of $\beta^* = \left( \Omega - I_d \right) \Delta\alpha$ introduced in (\ref{eq:def_beta_star_x}).  Then, we let $\wt S_x'$ be an estimator of $S_x'$ (as defined in (\ref{eq:def_bar_S_x})) as follows
\begin{align}
\label{eq:def_wt_S_x}
\wt S_x' = \left\{i: | [\wt\Omega-I_d]_{i\cdot}^T | | \wt\Delta\alpha(x)| \neq 0 \right\}.
\end{align}
Here, as in (\ref{eq:def_bar_S_x}), $|\cdot|$ takes the absolute value component-wise.  It is easy to see that
\begin{align}
\left\{i:\wh\beta_i(x)\neq 0\right\}\subset\wt S_x'.
\label{eq:wh_beta_support}
\end{align}

Recall the event $H'_{x,\epsilon}$ as introduced in (\ref{eq:H_x_p}), the absolute constant $J_1$ as introduced in Lemma~\ref{thm:delta_alpha}, the event $E_n$ and the absolute constant $J_2$ as introduced in Proposition~\ref{prop:Zhao14_Thm_IV.5}.  Then, we define the event
\begin{align}
\label{eq:def_L_x_gamma}
L_{x,\epsilon} &= \left\{ \wt S_x' = S_x' \right\} \cap H'_{x,\epsilon} \cap E_n \\
&\cap \left( \cap_{i\in S_x'} \left\{ |\wh\beta_i(x)-\beta^*_i(x)| \le  2 (M-1) \epsilon + 2 J_2 \kappa M s \sqrt{\gamma \log(n)} \lambda_n + 2 J_2 \kappa M s \lambda_n \epsilon \right\} \right). \nonumber
\end{align}
Theorem~\ref{thm_L} presents the performance guarantee of our estimator $\wh\beta$ of $\beta^*$.
\begin{theorem}
\label{thm_L}
Suppose that Assumptions~\ref{ass_Omega} and \ref{ass_d_n} hold, and that $n$ large enough.  Suppose that an arbitrary $x\in\mathbb{R}^d$ satisfies Assumption~\ref{ass_x}.  Then $L_{x,\epsilon}$ as defined in (\ref{eq:def_L_x_gamma}) satisfies
\begin{align}
\label{eq:L_prob}
\PP(L_{x,\epsilon}) \ge 1 - (5 s_x'' + 9) n^{-\gamma/2} - 4 s_x'' \exp\left( - J_1 \dfrac{n^{1-\gamma/2} \epsilon^2}{\sqrt{\gamma\log n}} \right).
\end{align}
\end{theorem}

\begin{proof}
The proof can be found in Section~\ref{sec:proof_thm_L}.
\end{proof}

%----------------------------------------------------------------------------------------------------

\subsection{Estimation of the copula part}
\label{sec_est_copula_part}

Recall the event $L_{x,\epsilon}$ as defined in (\ref{eq:def_L_x_gamma}).  Then, we define the event
\begin{align}
L'_{x,\epsilon} = L_{x,\epsilon} \cap \left( \cap_{i\in S_x'} \cap_{y\in\{0,1\}}\left\{ |\wh\alpha_{i|y}(x_i)-\alpha_{i|y}(x_i)|<\epsilon \right\} \right).
\label{eq:def_L_p_x}
\end{align}

Assumption~\ref{ass_x_beta_star} states the last piece of condition we need for our performance guarantee of the estimation of the copula part, which we state in Theorem~\ref{thm_est_copula_part}.
%\begin{assumption}
%\label{ass_x_beta_star}
%$x\in\mathbb{R}^d$ satisfies $x\in A^F_{n,\beta^*,\gamma}$ .
%\end{assumption}

\begin{theorem}
\label{thm_est_copula_part}
Suppose that Assumptions~\ref{ass_Omega} and \ref{ass_d_n} hold, and that $n$ large enough.  In addition, suppose that an arbitrary $x\in\mathbb{R}^d$ satisfies Assumptions~\ref{ass_x} and \ref{ass_x_beta_star}.  Then, on the event $L'_{x,\epsilon}$ as defined in (\ref{eq:def_L_p_x}), we have
\begin{align}
& \left| \left[(\wh\alpha_0 + \wh\alpha_1)^T \wh\beta  - (\alpha_0 + \alpha_1 )^T \beta^*\right](x) \right| \nonumber \\
& \le 2 \epsilon \|\beta^*(x)\|_{\ell_1} + 4s_x' \left(\sqrt{\gamma\log(n)}+\epsilon\right) \left[ (M-1) \epsilon + J_2 \kappa M s \lambda_n \left( \sqrt{\gamma \log(n)} + \epsilon \right) \right].
\label{eq:L_p_x_error_bound}
\end{align}
Furthermore, the event $L'_{x,\epsilon}$ satisfies
\begin{align}
\PP(L'_{x,\epsilon}) \ge 1 - (2 s_x' + 5 s_x''+9) n^{-\gamma/2} - (4s_x'+4s_x'') \exp\left( - J_1 \dfrac{n^{1-\gamma/2} \epsilon^2}{\sqrt{\gamma\log(n)}}  \right).
\label{eq:L_p_x_gamma_prob}
\end{align}
\end{theorem}

\begin{proof}
The proof can be found in Section~\ref{sec:proof_thm_est_copula_part}.
\end{proof}

So far we have left $\epsilon$, which corresponds to the estimation error (as can be see from Theorem~\ref{thm_est_copula_part}), unspecified.  Now we fix our choice of $\epsilon$ by matching the exponential term in (\ref{eq:L_p_x_gamma_prob}), namely $\exp\left( - J_1 n^{1-\gamma/2} \epsilon^2 / \sqrt{\gamma\log(n)}\right)$, to $n^{-\gamma/2}$, the rate of the exclusion probability. Hence we set $\epsilon=\epsilon_n$ for $\epsilon_n$ as introduced in (\ref{eq:choice_epsilon}).  Recall that $\gamma<2$, so we have from (\ref{eq:choice_epsilon}) the simple bound that
\begin{align}
\label{eq:choice_epsilon_simple}
\epsilon_n < (2J_1)^{-\frac{1}{2}} \log^{\frac{3}{4}}(n) n^{-\left(\frac{1}{2}-\frac{\gamma}{4}\right)}.
\end{align}

With the choice (\ref{eq:choice_epsilon}) of $\epsilon=\epsilon_n$, we state in Corollary~\ref{cor_est_copula_part} a concrete instance of Theorem~\ref{thm_est_copula_part}.  We define the event
\begin{align}
L^{\text{copula}}_{x,n} = L'_{x,\epsilon_n};
\label{eq:def_L_copula}
\end{align}
that is, $L^{\text{copula}}_{x,n} = L'_{x,\epsilon}$, for $L'_{x,\epsilon}$ introduced in (\ref{eq:def_L_p_x}), with $\epsilon$ replaced by $\epsilon_n$ in the latter.

\begin{corollary}[Estimation of the copula part]
\label{cor_est_copula_part}
Suppose that Assumptions~\ref{ass_Omega} and \ref{ass_d_n} hold, and that $n$ large enough.  In addition, suppose that an arbitrary $x\in\mathbb{R}^d$ satisfies Assumptions~\ref{ass_x} and \ref{ass_x_beta_star}.  Then, on the event $L^{\textnormal{copula}}_{x,n}$ as defined in (\ref{eq:def_L_copula}), we have
\begin{align}
& \left| \left[(\wh\alpha_0 + \wh\alpha_1)^T \wh\beta  - (\alpha_0 + \alpha_1 )^T \beta^*\right](x) \right| \nonumber \\
&\le J_1' \left( \|\beta^*(x)\|_{\ell_1}   + s_x' M \sqrt{\log(n)} \right) \log^{\frac{3}{4}}(n) n^{-\left(\frac{1}{2}-\frac{\gamma}{4}\right)}.
\label{eq:L_copula_error_bound}
\end{align}
Here $J_1'$ is some absolute constant that depends only on the absolute constant $J_1$.  Furthermore, the event $L^{\textnormal{copula}}_{x,n}$ satisfies
\begin{align}
\PP(L^{\textnormal{copula}}_{x,n})\ge 1- (6s_x'+9s_x''+9) n^{-\gamma/2},
\label{eq:P_L_copula}
\end{align}
\end{corollary}

\begin{proof}
Inequality~(\ref{eq:P_L_copula}) follows immediately from (\ref{eq:L_p_x_gamma_prob}) by the choice (\ref{eq:choice_epsilon}) of $\epsilon=\epsilon_n$.  We have $\epsilon_n=o(\sqrt{\gamma\log(n)})$, and in addition with the choice $\epsilon=\epsilon_n$, the second term in the square bracket in (\ref{eq:L_p_x_error_bound}) is dominated by the first for large $n$ by the second half of Assumption~\ref{ass_d_n}.  Then, (\ref{eq:L_copula_error_bound}) follows immediately from (\ref{eq:L_p_x_error_bound}) by bounding $\epsilon=\epsilon_n$ as in (\ref{eq:choice_epsilon_simple}) and by bounding the remaining appearances of $\gamma$ by 2.
\end{proof}

%----------------------------------------------------------------------------------------------------

\section{Detailed study of the naive Bayes part}
\label{sec_naive_bayes}

\subsection{Outline}

Our estimation of the naive Bayes part in this section roughly parallels certain components of our estimation of the copula part in Section~\ref{sec_copula}.  In Section~\ref{sec:kernel_density_estimation}, paralleling Section~\ref{sec_est_tran}, we study the estimation of the density functions $f_{i|y}$ in a form that is suitable for the estimation of the log density ratio.  In Section~\ref{sec:sparse_estimation_Bayes_part}, paralleling Section~\ref{sec_est_Delta_alpha}, we study the sparse estimation of $\Delta\log f$, which leads to our estimation of the naive Bayes part.

%----------------------------------------------------------------------------------------------------

\subsection{Relative deviation property of the kernel density estimator}
\label{sec:kernel_density_estimation}

We recall, for $i\in\{1,\dots,d\}$ and $y\in\{0,1\}$, the kernel density estimator $\wh f_{i|y}$ of $f_{i|y}$ and $\wh f_i$ of $f_i$ as defined in (\ref{eq:kernel_density_estimator_i_y}) and (\ref{eq:kernel_density_estimator_i}) respectively, and the set $B^f_{h_{n,i},i,y}$ as defined in (\ref{def_B_f_c}).  In words, the second term on the right hand side of (\ref{def_B_f_c}) consists of those points $t$ such that the supremum of the density $f_{i|y}(t')$ are close to $f_{i|y}(t)$ in a relative sense (by a factor of two), where $t'$ can range over an interval of length $2h_{n,i}$ centered around $t$.  The constant $2$ appearing in (\ref{def_B_f_c}) is chosen for convenience and can be replaced by any other constant larger than one.

%Here $\underline{f_{n,i}}$ is defined either as in (\ref{eq:f_i_lower_bound}) or as in (\ref{eq:f_i_lower_bound_super_smooth}) depending on the context.

We first obtain an inequality regarding the relative deviation from the mean of our kernel density estimators.
\begin{proposition}
\label{prop:f_rel_dev_variance}
Suppose that $t\in\mathbb{R}$ satisfies
\begin{align}
\label{eq:t_in_B_f}
t\in B^f_{h_{n,i},i,y} \quad\text{and}\quad f_{i|y}(t)\ge\underline{f_{n,i}}.
\end{align}
Then, the kernel density estimator $\wh f_{i|y}$ satisfies
\begin{align}
\label{eq:f_rel_dev_variance}
\PP\left\{ \dfrac{ | \wh f_{i|y}(t) - \EE \wh f_{i|y}(t) | }{ f_{i|y}(t) } \ge \epsilon' \right\} \le 2 \exp\left( - \dfrac{3}{8\max\left\{ \|K_i\|_{L^{\infty}} \epsilon', 3\|K_i\|_{L^2}^2 \right\}} n \epsilon'^2 f_{i|y}(t) h_{n,i} \right).
\end{align}
\end{proposition}

\begin{proof}
The proof can be found in Section~\ref{sec:proof_prop:f_rel_dev_variance}.
\end{proof}

Note that, Proposition~\ref{prop:f_rel_dev_variance} suggests that $f_{i|y}(t)$ should not be too small, for otherwise the bound offered by (\ref{eq:f_rel_dev_variance}) is weak.  This, together with other considerations, lead us to concentrate on estimating the densities that satisfy a lower bound, such as that expressed by the second half of (\ref{eq:t_in_B_f}).  We also match $\epsilon'$ to $\epsilon_n$ as in (\ref{eq:choice_epsilon}).  Our relative deviation inequality for kernel density estimation is presented in Theorem~\ref{thm:f_rel_dev}.

\begin{theorem}
\label{thm:f_rel_dev}
Suppose that Assumption~\ref{ass_d_n} holds, and that $n$ large enough.  Suppose that $t\in\mathbb{R}$ satisfies condition (\ref{eq:t_in_B_f}).  Then we have
\begin{align}
\label{eq:f_rel_dev_concrete}
\PP\left\{ \dfrac{ | \wh f_{i|y}(t) - f_{i|y}(t) | }{ f_{i|y}(t) } \ge \epsilon_n \right\} \le \dfrac{1}{d} n^{-\gamma/2}.
\end{align}

\end{theorem}

\begin{proof}
The proof can be found in Section~\ref{sec:proof_thm:f_rel_dev}.
\end{proof}

%----------------------------------------------------------------------------------------------------

\subsection{Sparse estimation of the naive Bayes part}
\label{sec:sparse_estimation_Bayes_part}

%Our sparse estimation of the naive Bayes part parallels our earlier sparse estimation of the copula part in Section~\ref{sec_est_Delta_alpha}.

Recall that $\wt\Delta\log f$ as introduced in (\ref{eq:wt_Delta_log_f}) is the sparse estimator of $\Delta\log f$, and its construction is detailed in Section~\ref{eq:sparse_estimation_naive_Bayes_part}.  We also recall from (\ref{eq:S_f_x}) the sparsity sets and indices for the naive Bayes part.  We let
\begin{align}
\wh S^f_x = \{i: \wt\Delta\log f_i(x_i)\neq 0\} \nonumber
\end{align}
be the estimator of $S^f_x$ .  Similar to the sparse estimation of the copula part, we first consider the case $i\notin S^f_x$, i.e., $\Delta\log f_i(x_i) = 0$ and so $f_i(x_i)=f_{i|0}(x_i)=f_{i|1}(x_i)$.  Analogous to (\ref{eq:def_H_i_t}), we define the event
\begin{align}
G_{i,t} = \left\{ \wt\Delta\log f_i(t) \neq 0 \right\}.
\label{eq:def_G_i_x_i}
\end{align}

%----------------------------------------------------------------------------------------------------

\begin{theorem}
\label{thm:S_f_x_c_est}
Suppose that Assumption~\ref{ass_d_n} holds, and that $n$ is large enough.  Suppose that an arbitrary $x\in\mathbb{R}^d$ satisfies Assumption~\ref{ass_density_x_S_f_x_c}.  Then, we have, for all $i\notin S^f_x$,
\begin{align}
\PP(G_{i,x_i}^c) \ge 1 - \dfrac{2}{d} n^{-\gamma/2}.
\label{eq:G_i_prob_S_f_x_c_individual}
\end{align}
Hence, by the union bound, we have
\begin{align}
\PP\left(\wh S^f_x \subset S^f_x\right) = \PP\left(\cap_{i\notin S^f_x} G_{i,x_i}^c\right) &\ge 1 - \dfrac{2 (d-s^f_x)}{d} n^{-\gamma/2}.
\label{eq:G_i_prob_S_f_x_c_union}
\end{align}
\end{theorem}

\begin{proof}
The proof can be found in Section~\ref{sec:proof_thm:S_f_x_c_est}.
\end{proof}

%----------------------------------------------------------------------------------------------------

Next, our consideration of the case $i\in S^f_x$ leads to Theorem~\ref{thm:est_naive_bayes_part} (which strengthens Theorem~\ref{thm:S_f_x_c_est}) which states that, when combining our earlier Assumption~\ref{ass_density_x_S_f_x_c} with the additional Assumption~\ref{ass_density_x_S_f_x}, we can accurately estimate the naive Bayes part with high probability.  This result is based on a bound on the probability of the following event
\begin{align}
\label{eq:L_n_b_x_n_gamma}
L^{\text{bayes}}_{x,n} = \left\{\wh S^f_x \subset S^f_x\right\} \cap \left( \cap_{i\in S^f_x} \left\{ \left| \wt\Delta\log f_i(x_i) - \Delta\log f_i(x_i) \right| < 2 \tilde\delta_{n,\gamma} \right\} \right).
\end{align}
Note that, for technical reasons, we do not require accurate identification of all nonzero components of $\Delta\log f_i$, as can be seen from (\ref{eq:L_n_b_x_n_gamma}).

\begin{theorem}[Estimation of the naive Bayes part]
\label{thm:est_naive_bayes_part}
Suppose that Assumption~\ref{ass_d_n} holds, and that $n$ is large enough.  Suppose that an arbitrary $x\in\mathbb{R}^d$ satisfies Assumptions~\ref{ass_density_x_S_f_x_c} and \ref{ass_density_x_S_f_x}.  Then, on the event $L^{\textnormal{bayes}}_{x,n}$ as defined in (\ref{eq:L_n_b_x_n_gamma}), for $\tilde\delta_{n,\gamma}$ as defined in (\ref{eq:def_density_delta}), we have
\begin{align}
\left\| \left[\Delta\log f - \wt\Delta\log f\right](x) \right\|_{\ell_1} \le 2 s^f_x \tilde\delta_{n,\gamma}.
\label{eq:error_bound_bayes}
\end{align}
In addition, the event $L^{\textnormal{bayes}}_{x,n}$ satisfies
\begin{align}
\label{eq:L_n_b_x_n_gamma_prob}
\PP (L^{\textnormal{bayes}}_{x,n}) \ge 1 - 2 n^{-\gamma/2}.
\end{align}

\end{theorem}

\begin{proof}
The proof can be found in Section~\ref{sec:proof_thm:est_naive_bayes_part}.
\end{proof}

%For super-smooth densities, we can set the bandwidth to be a constant to obtain faster convergence rate.

%--------------------------------------------------------------------------------
%--------------------------------------------------------------------------------
%--------------------------------------------------------------------------------
%--------------------------------------------------------------------------------
%----------------------------------------------------------------------------------------------------

%--------------------------------------------------------------------------------
%--------------------------------------------------------------------------------
%--------------------------------------------------------------------------------
%--------------------------------------------------------------------------------
%--------------------------------------------------------------------------------

\section{Proofs for Section~\ref{sec:Introduction}}
\label{sec:proof_for_sec:Introduction}

\subsection{Proof of Theorem~\ref{thm:Gaussian_copula_density_ratio}}
\label{sec:proof_thm:Gaussian_copula_density_ratio}
By the assumption that $(X|Y=0)$ and $(X|Y=1)$ have the same Gaussian copula with the copula correlation matrix $\Sigma$, we have that
\begin{align}
(\alpha_y(X)|Y=y)\sim N(0,\Sigma).
\label{eq:alpha_y_X_normal}
\end{align}
We derive the density $f^y(x)$ for $y\in\{0,1\}$.  We let $\Phi_{\Sigma}$ denote the distribution function and $\phi_{\Sigma}$ denote the density function of a multivariate $N(0,\Sigma)$ distribution.  We have, for $x\in\mathbb{R}^d$,
\begin{align}
f^y(x) &= \dfrac{d}{dx} \PP(\left.X\le x\right|Y=y) = \dfrac{d}{dx} \PP(\left.\alpha_y(X)\le\alpha_y(x)\right|Y=y) \nonumber \\
%&= \dfrac{d}{dx} \Phi_{\Sigma}\left( \Phi^{-1}(F_{1|y}(x_1)),\dots,\Phi^{-1}(F_{d|y}(x_d)) \right) \nonumber \\
&= \dfrac{d}{dx} \Phi_{\Sigma}\left( \alpha_y(x) \right) = \phi_{\Sigma}\left( \alpha_y(x) \right) \prod_{i=1}^d \dfrac{d}{d x_i} \Phi^{-1}(F_{i|y}(x_i)) \nonumber \\
&= \dfrac{1}{\sqrt{(2\pi)^d |\Sigma|}} \exp\left(-\dfrac{1}{2} (\alpha_y(x))^T \Omega \alpha_y(x) \right) \prod_{i=1}^d \dfrac{1}{ \phi(\alpha_{i|y}(x_i)) } f_{i|y}(x_i) \nonumber \\
&= \dfrac{1}{\sqrt{(2\pi)^d |\Sigma|}} \exp\left(-\dfrac{1}{2} (\alpha_y(x))^T (\Omega - I_d) \alpha_y(x) \right) \prod_{i=1}^d f_{i|y}(x_i).
\label{eq:density_f_y_x}
\end{align}
Here in the third equality we have invoked (\ref{eq:alpha_y_X_normal}).  Then, from (\ref{eq:density_f_y_x}), we have
\begin{align}
\log \dfrac{f^0(x)}{f^1(x)} &= -\dfrac{1}{2} (\alpha_0(x))^T (\Omega - I_d) \alpha_0(x) + \dfrac{1}{2} (\alpha_1(x))^T (\Omega - I_d) \alpha_1(x) + \sum_{i=1}^d \left[ \log f_{i|0}(x_i) -  \log f_{i|1}(x_i) \right], \nonumber
\end{align}
from which Equation~(\ref{eq:thm_density_ratio}) easily follows.
\qed

%----------------------------------------------------------------------------------------------------
%----------------------------------------------------------------------------------------------------
%----------------------------------------------------------------------------------------------------
%----------------------------------------------------------------------------------------------------

\section{Proofs for Section~\ref{sec:risk_bound}}
\label{sec:proof_for_sec:risk_bound}

\subsection{Proof of Proposition~\ref{prop:Gaussian_super_smooth}}
\label{sec:proof_prop:Gaussian_super_smooth}
We let $f$ be the density function of a (univariate) normal distribution with mean $\mu$ and variance $\sigma^2\ge\sigma_0^2$.  We fix arbitrary $t\in\mathbb{R}$, and $l\ge1$.  In the following $f^{(l)}$ and $\phi^{(l)}$ denote the $l$th derivative of $f$ and (the standard normal density function) $\phi$ respectively, but $K^{(l)}$ is the kernel of order $l$.  We have
\begin{align}
&\EE\left[\wh f_{K^{(l)}}(t)\right]-f(t) = \int K^{(l)}(u) \dfrac{(uh)^{l}}{l!}f^{(l)}(t+\tau u h) du \nonumber \\
&= \int K^{(l)}(u) \dfrac{(uh)^{l}}{l!} \left(\dfrac{1}{\sigma}\right)^{l+1} \phi^{(l)}\left( \dfrac{t-\mu+\tau u h}{\sigma}\right) du \nonumber \\
&= \int \dfrac{1}{\sqrt{2\pi}}(-1)^l \dfrac{(uh)^{l}}{l!} \left(\dfrac{1}{\sigma}\right)^{l+1} \exp\left[-\dfrac{(t-\mu+\tau u h)^2}{2\sigma^2}\right] H_{e,l}\left(\dfrac{t-\mu+\tau u h}{\sigma}\right) K^{(l)}(u) du \nonumber \\
&= \int \dfrac{1}{\sqrt{2\pi}}(-1)^l 2^{-l/2} \dfrac{(uh)^{l}}{l!} \left(\dfrac{1}{\sigma}\right)^{l+1} \exp\left[-\dfrac{(t-\mu+\tau u h)^2}{2\sigma^2}\right] H_{l}\left(\dfrac{t-\mu+\tau u h}{\sigma\sqrt{2}}\right) K^{(l)}(u) du \nonumber \\
&= \int \dfrac{1}{\sqrt{2\pi}}(-1)^l 2^{-l/2} \dfrac{(uh)^{l}}{l!} \left(\dfrac{1}{\sigma}\right)^{l+1} \exp\left(-t'^2\right) H_{l}\left(t'\right) K^{(l)}(u) du.
\label{eq:super_smooth_1}
\end{align}
Here the first equality follows by standard derivation for $K^{(l)}$ a kernel of order $l$ (e.g., \cite[Proposition~1.2]{TsybakovBook09}), and in there $\tau$ is some number such that $0\le \tau\le 1$, in the third equality $H_{e,l}$ is ``probablist's'' Hermite polynomial of order $l$, in the fourth equality $H_{l}$ is ``physicist's'' Hermite polynomial of order $l$, and in the last equality we have let $t'=\dfrac{t-\mu+\tau u h}{\sigma\sqrt{2}}$.  We further derive from (\ref{eq:super_smooth_1}) that
\begin{align}
\left| \EE\left[\wh f_{K^{(l)}}(t)\right]-f(t) \right| &\le \dfrac{ h^l}{\sqrt{2\pi}(l!)^{1/2}} \left(\dfrac{1}{\sigma}\right)^{l+1} \int e^{-t'^2} | H_{l}\left(t'\right) | \left[2^{-l/2} (l!)^{-1/2} \right] |u|^l |K^{(l)}(u)| du \nonumber \\
&\le \dfrac{C_{\text{Cram\'{e}r}}}{\sqrt{2\pi}(l!)^{1/2}} \left(\dfrac{1}{\sigma}\right)^{l+1} h^l \int_{-1}^1 e^{-t'^2/2} |u|^l |K^{(l)}(u)| du \nonumber \\
&\le \dfrac{C_{\text{Cram\'{e}r}}}{\sqrt{2\pi}(l!)^{1/2}} \left(\dfrac{1}{\sigma}\right)^{l+1} h^l \int_{-1}^1 |K^{(l)}(u)| du \nonumber \\
&\le \dfrac{C_{\text{Cram\'{e}r}} \|K^{(l)}\|_{L^{\infty}} }{\sqrt{\pi/2}(l!)^{1/2}} \left(\dfrac{1}{\sigma}\right)^{l+1} h^l \le c_l h^l. \nonumber
\end{align}
Here in the second inequality we have used Cram\'{e}r's inequality stating that $| H_{l}\left(t'\right) | \le C_{\text{Cram\'{e}r}} e^{t'^2/2} 2^{l/2}\sqrt{l!}$ for the absolute constant $C_{\text{Cram\'{e}r}}\le1.09$ \cite[(19) in Section~10.18]{Erdelyi53}, \cite[(22.14.17)]{Abramowitz72}.  It is easy to show that we indeed have $c_l\rightarrow 0$ by Stirling approximation and the fact that $\|K\|_{L^{\infty}}\le C_K l^{3/2}$.
%Page~208
\qed

%----------------------------------------------------------------------------------------------------

\subsection{Proof of Theorem~\ref{thm:excess_risk}}
\label{sec:proof_thm:excess_risk}
With the fact that (for $\pi_0=\pi_1 = 1/2$)
\begin{align}
\eta = \dfrac{1}{2 f } f^1 = \dfrac{f^1}{f^0+f^1}, \nonumber
\end{align}
we have
\begin{align}
\dfrac{ f^0 }{ f^1 } = \dfrac{ 1-\eta }{ \eta } = \dfrac{ 1 }{ \eta } - 1, \nonumber
\end{align}
which further implies that
\begin{align}
\eta = \dfrac{ 1 }{ (f^0/f^1) + 1 }  = \dfrac{ 1 }{ e^{\log(f^0/f^1)} + 1 }.
\label{eq:eta_density_ratio}
\end{align}
We define the function $\bar\eta:\mathbb{R}\rightarrow\mathbb{R}$ as
\begin{align}
\bar\eta(t) = \dfrac{ 1 }{ e^{t} + 1 }. \nonumber
\end{align}
It is easy to deduce that $\bar\eta(0)=1/2$, and $|d\bar\eta(t)/dt|\le 1/4$ for all $t\in\mathbb{R}$.  Hence,
\begin{align}
|\bar\eta(t)-1/2|\le |t|/4.
\label{eq:eta_p_x}
\end{align}
From (\ref{eq:eta_density_ratio}) and (\ref{eq:eta_p_x}), we conclude that, for all $x\in\mathbb{R}^d$,
\begin{align}
|\eta(x)-1/2| \le \dfrac{1}{4} \left|\log(f^0/f^1)(x)\right|.
\label{eq:regression_function_to_density_ratio}
\end{align}

Now we are ready to derive the excess risk.  We have
\begin{align}
& \PP(\wh\delta_n(X)\neq Y)-\PP(\delta^*(X)\neq Y) = \EE\left( |2\eta(X)-1| \mathbbm{1}\left\{\wh \delta_n(X)\neq \delta^*(X)\right\} \right) \nonumber \\
& = \EE\left( |2\eta(X)-1| \1\left\{\wh \delta_n(X)\neq \delta^*(X)\right\} \1\left\{ X\notin A_{n,d,\gamma} \right\} \right) \nonumber \\
& + \EE\left( |2\eta(X)-1| \1\left\{\wh \delta_n(X)\neq \delta^*(X)\right\}  \1\left\{ \left|\log(f^0/f^1)(X)\right|\le \Delta(X) \right\} \1\left\{ X\in A_{n,d,\gamma} \right\} \right) \nonumber \\
& + \EE\left( |2\eta(X)-1| \1\left\{\wh \delta_n(X)\neq \delta^*(X)\right\}  \1\left\{ \left|\log(f^0/f^1)(X)\right|> \Delta(X) \right\} \1\left\{ X\in A_{n,d,\gamma} \right\} \right) \nonumber \\
& \le \PP\left(X\notin A_{n,d,\gamma}\right)  +  \dfrac{1}{2}\EE\left[ \left|\log(f^0/f^1)(X)\right| \1\left\{ \left|\log(f^0/f^1)(X)\right|\le \Delta(X) \right\}  \right] \nonumber \\
& + \EE\left( \1\left\{ \left| \left[ \wh{\log(f^0/f^1)} - \log(f^0/f^1)\right](X)\right| > \Delta(X) \right\} \1\left\{ X\in A_{n,d,\gamma} \right\} \right) \nonumber \\
& \le \PP\left(X\notin A_{n,d,\gamma}\right)  +  \dfrac{1}{2}\EE\left[ \Delta(X) \1\left\{ \left|\log(f^0/f^1)(X)\right|\le \Delta(X) \right\}  \right] \nonumber \\
& + \EE_X \left[ \PP^{\otimes 2n} \left( \1\left\{ \left| \left[\wh{\log(f^0/f^1)}-\log(f^0/f^1)\right](X)\right| > \Delta(X) \right\} \1\left\{ X\in A_{n,d,\gamma} \right\} \right) \right] \nonumber \\
& \le \PP\left(X\notin A_{n,d,\gamma}\right)  +  \dfrac{1}{2}\EE\left[ \Delta(X) \1\left\{ \left|\log(f^0/f^1)(X)\right|\le \Delta(X) \right\}  \right] + \EE\left[ 6 s_X'+ 9 s_X'' + 11 \right] n^{-\gamma/2}, \nonumber
%\label{eq:excess_risk_1_derivation}
\end{align}
which is Inequality~(\ref{eq:excess_risk_1}).  Here the first equality is a well known fact expressing the excess risk in terms of the regression function $\eta$ (e.g., \cite[Theorem~2.2]{Devroye96}), the first inequality follows by (\ref{eq:regression_function_to_density_ratio}), $|2\eta(X)-1|\le 1$, and the fact that $\wh\delta_n(X)\neq \delta^*(X)$ is possible only when $\left|\left[\wh{\log(f^0/f^1)} - \log(f^0/f^1)\right](X)\right|>|\log(f^0/f^1)(X)|$, in the second inequality $\PP^{\otimes 2n}$ denotes probability taken w.r.t. the $2n$ training samples and $\EE_X$ denotes expectation taken w.r.t. $X$, and the last inequality follows from Corollary~\ref{corollary_master_error_bound}.

Next, (\ref{eq:excess_risk_2}) follows from (\ref{eq:excess_risk_1}) by replacing $\|\beta^*(X)\|_{\ell_1}$, $s_X'$, $s_X''$ and $s^f_X$ by their constant bounds $C_{\beta^*}$, $s'$, $s''$ and $s^f$ respectively, and then invoking the margin assumption~\ref{assumption:margin_condition}.
\qed

%----------------------------------------------------------------------------------------------------

\subsection{Proof of Theorem~\ref{thm:exclusion_set_simple_Gaussian}}
\label{sec:proof_thm:exclusion_set_simple_Gaussian}

By (\ref{eq:wt_A_n_d_gamma}) and (\ref{eq:wt_A_F_n_d_gamma}), we have
\begin{align}
&\PP\left(X \notin A_{n,d,\gamma}\right) \nonumber \\
&\le \PP\left(X \notin A^{F,1}_{n,d,\gamma}\right) + \PP\left(X \notin A^{F,2}_{n,d,\gamma}\right) + \PP\left(X \notin A^F_{n,\beta^*,\gamma}\right) + \PP\left(X \notin A^{f,=}_{n,d,\gamma} \cap A^{f,\neq}_{n,d,\gamma} \right).
\label{eq:wt_A_n_d_gamma_c_prob}
\end{align}
We bound the four terms on the right hand side of (\ref{eq:wt_A_n_d_gamma_c_prob}) separately.

%----------------------------------------------------------------------------------------------------

\subsubsection{The term $\PP\left( X \notin A^{F,1}_{n,d,\gamma} \right)$}

%We first bound $\PP\left(X \notin A^{F,1}_{n,d,\gamma}\right)$.
We have
\begin{align}
\PP\left(X \notin A^{F,1}_{n,d,\gamma}\right) &= \dfrac{1}{2}\PP\left( \left. X \notin A^{F,1}_{n,d,\gamma}\right\vert Y=0\right) + \dfrac{1}{2}\PP\left(\left. X \notin A^{F,1}_{n,d,\gamma}\right\vert Y=1 \right) \nonumber \\
&= \PP\left( \left. X \notin A^{F,1}_{n,d,\gamma}\right\vert Y=0\right).
\label{eq:P_X_notin_A_F_1}
\end{align}
Here the second equality follows by symmetry.  Then, by (\ref{eq:wt_A_F_1}), we have
\begin{align}
\PP\left(\left. X \notin A^{F,1}_{n,d,\gamma} \right\vert Y=0 \right) &\le \sum_{i\in S''} \sum_{y\in\{0,1\}} \PP\left( \left. X_i \notin B_{n,\gamma,i,y} \right\vert Y=0 \right).
\label{eq:P_X_notin_A_F_1_Y_0}
\end{align}
We fix an arbitrary $i\in S''$.  First note that, we have that $F_{i|0}(X_i|Y=0)=\Phi(X_i|Y=0)$ follows a uniform distribution on $(0,1)$.  Hence, by (\ref{eq:B_n_gamma_i_y}), we have
\begin{align}
\PP\left( \left. X_i \notin B_{n,\gamma,i,0} \right\vert Y=0 \right) &= \PP\left( \left. F_{i|0}(X_i) < 8 g(2n,\gamma) \right\vert Y=0 \right) + \PP\left( \left. F_{i|0}(X_i) > 1 - 8 g(2n,\gamma) \right\vert Y=0 \right) \nonumber \\
& = 16 g(2n,\gamma).
\label{eq:P_X_i_notin_B_00}
\end{align}
On the other hand, the distribution of $F_{i|1}(X_i|Y=0)$ is no longer a uniform distribution and a more involved analysis is necessary.  We have
\begin{align}
&\PP\left( \left. X_i \notin B_{n,\gamma,i,1} \right\vert Y=0 \right) \nonumber \\
&= \PP\left( \left. F_{i|1}(X_i) < 8 g(2n,\gamma) \right\vert Y=0 \right) + \PP\left( \left. F_{i|1}(X_i) > 1 - 8 g(2n,\gamma) \right\vert Y=0 \right) \nonumber \\
&= \PP\left( \left. \Phi_{\mu}(X_i) < 8 g(2n,\gamma) \right\vert Y=0 \right) + \PP\left( \left. \Phi_{\mu}(X_i) > 1 - 8 g(2n,\gamma) \right\vert Y=0 \right) \nonumber \\
%&= \PP\left( \left. X_i < F_{i|1}^{-1}(8 g(2n,\gamma)) \right\vert Y=0 \right) + \PP\left( \left. X_i > F_{i|1}^{-1} (1 - 8 g(2n,\gamma)) \right\vert Y=0 \right) \nonumber \\
&= \PP\left( \left. X_i < \Phi_{\mu}^{-1}(8 g(2n,\gamma)) \right\vert Y=0 \right) + \PP\left( \left. X_i > \Phi_{\mu}^{-1} (1 - 8 g(2n,\gamma)) \right\vert Y=0 \right).
\label{eq:Phi_mismatch}
\end{align}
For the second term in (\ref{eq:Phi_mismatch}), using $\Phi_{\mu}^{-1}(t)=\Phi^{-1}(t)+\mu$, we have
\begin{align}
&\PP\left( \left. X_i > \Phi_{\mu}^{-1} (1 - 8 g(2n,\gamma)) \right\vert Y=0 \right) = \PP\left( \left. X_i > \Phi^{-1} (1 - 8 g(2n,\gamma)) + \mu \right\vert Y=0 \right) \nonumber \\
&\le \PP\left( \left. X_i > \Phi^{-1} (1 - 8 g(2n,\gamma)) \right\vert Y=0 \right) = \PP\left( \left. \Phi(X_i) > 1 - 8 g(2n,\gamma) \right\vert Y=0 \right) \nonumber \\
&= 8 g(2n,\gamma).
\label{eq:Phi_mismatch_1_0}
\end{align}
The first term in (\ref{eq:Phi_mismatch}) is more complicated.  First, we note that, for $t\le\min\{-1,-\mu\}$, we have
\begin{align}
\dfrac{\Phi(t)}{\Phi_{\mu}(t)} &= \dfrac{\Phi(t)}{\Phi(t-\mu)} \le \dfrac{\dfrac{1}{-t}\phi(t)}{\dfrac{-(t-\mu)}{1+(-(t-\mu))^2}\phi(t-\mu)} = \dfrac{1+(t-\mu)^2}{t(t-\mu)} e^{\mu^2/2} e^{-\mu t} \nonumber \\
&\le \dfrac{1+(2t)^2}{t^2} e^{\mu^2/2} e^{-\mu t} = \left(\dfrac{1}{t^2} + 4\right) e^{\mu^2/2} e^{-\mu t} \le 5 e^{\mu^2/2} e^{-\mu t}.
\label{eq:Phi_mismatch_2_0}
\end{align}
Here in the first inequality we have used (\ref{eq:Phi_phi_sym}) for $t\le 0$, and in the second inequality we have used the assumption $t\le-\mu$.  Hence, for $n$ large enough such that $\Phi_{\mu}^{-1}(8 g(2n,\gamma))\le\min\{-1,-\mu\}$, by (\ref{eq:Phi_mismatch_2_0}) with $t=\Phi_{\mu}^{-1}(8 g(2n,\gamma))$, we have
\begin{align}
&\PP\left( \left. X_i < \Phi_{\mu}^{-1}(8 g(2n,\gamma)) \right\vert Y=0 \right) = \Phi\left(\Phi_{\mu}^{-1}(8 g(2n,\gamma)\right) \nonumber \\
&\le 5 e^{\mu^2/2} e^{-\mu \Phi_{\mu}^{-1}(8 g(2n,\gamma))} \Phi_{\mu}\left( \Phi_{\mu}^{-1}(8 g(2n,\gamma)) \right) \nonumber \\
&= 5 e^{-\mu^2/2} e^{-\mu \Phi^{-1}(8 g(2n,\gamma))} (8 g(2n,\gamma)).
\label{eq:Phi_mismatch_2_1}
\end{align}
Then, invoking (\ref{eq:Phi_inv_sym}), we further deduce from (\ref{eq:Phi_mismatch_2_1}) that
\begin{align}
\PP\left( \left. X_i < \Phi_{\mu}^{-1}(8 g(2n,\gamma)) \right\vert Y=0 \right) &\le 5 e^{-\mu^2/2} \exp\left\{ \mu\sqrt{2\log \left(\dfrac{1}{2\cdot 8g(2n,\gamma)}\right)} \right\} (8 g(2n,\gamma)) \nonumber \\
&\le 5 e^{-\mu^2/2} \exp\left\{ \mu \sqrt{C\log( n^{\gamma/2}}) \right\} (8 g(2n,\gamma)) \nonumber \\
&= 5 e^{-\mu^2/2} e^{C\mu\sqrt{\gamma\log(n)}} (8 g(2n,\gamma)).
%&= 8 g(2n,\gamma) n^{o(1)}.
\label{eq:Phi_mismatch_2_2}
\end{align}
Plugging (\ref{eq:Phi_mismatch_2_2}) and (\ref{eq:Phi_mismatch_1_0}) into (\ref{eq:Phi_mismatch}), we have, for $J_{\mu}$ some constant dependent only on $\mu$,
\begin{align}
\PP\left( \left. X_i \notin B_{n,\gamma,i,1} \right\vert Y=0 \right) \le J_{\mu} e^{C\mu\sqrt{\gamma\log(n)}} g(2n,\gamma).
\label{eq:P_X_i_notin_B_mismatch}
\end{align}
Plugging (\ref{eq:P_X_i_notin_B_00}) and (\ref{eq:P_X_i_notin_B_mismatch}) into (\ref{eq:P_X_notin_A_F_1_Y_0}) and then in turn into (\ref{eq:P_X_notin_A_F_1}), we conclude that
\begin{align}
\PP\left( X \notin A^{F,1}_{n,d,\gamma} \right) \le J_{\mu}' s'' e^{C\mu\sqrt{\gamma\log(n)}} g(2n,\gamma).
\label{eq:P_X_notin_A_F_1_final}
\end{align}
Here $J_{\mu}'$ is another constant dependent only on $\mu$.

%----------------------------------------------------------------------------------------------------

\subsubsection{The term $\PP\left(X \notin A^{F,2}_{n,d,\gamma}\right)$}

We have
\begin{align}
&\PP\left(X \notin A^{F,2}_{n,d,\gamma}\right) \le \sum_{i\in S''} \PP\left( X_i \notin B^{\delta}_{n,\gamma,i} \right) \nonumber \\
&= \sum_{i\in S''} \PP\left( \text{None of}~(\ref{eq:F_ratio_1}), (\ref{eq:F_ratio_2}), (\ref{eq:F_ratio_3}), (\ref{eq:F_ratio_4})~\text{is satisfied with $t$ replaced by $X_i$} \right).
\label{eq:P_X_notin_A_F_2}
\end{align}
It is elementary to show that there exists some constant $J_{\mu}''>0$, which depends only on $\mu$, such that for all $i\in S''$ and for all $t\in\mathbb{R}$,
\begin{align}
\max\left\{ \dfrac{ F_{i|0}(t) }{ F_{i|1}(t) } , \dfrac{ 1-F_{i|1}(t) }{ 1-F_{i|0}(t) } \right\} \ge 1+J_{\mu}''. \nonumber
\end{align}
In addition, under Assumption~\ref{ass_d_n}, $\bar\delta_{n,d,\gamma}, \bar\delta_{n,1,\gamma}\rightarrow 0$ as $n\rightarrow\infty$.  Then, for all $n$ large enough, the probabilities in the last line of (\ref{eq:P_X_notin_A_F_2}) are identically zero, and so we have
\begin{align}
\PP\left(X \notin A^{F,2}_{n,d,\gamma}\right)=0.
\label{eq:P_X_notin_A_F_3}
\end{align}

%----------------------------------------------------------------------------------------------------

\subsubsection{The term $\PP\left( X \notin A^F_{n,\beta^*,\gamma} \right)$}

By (\ref{eq:wt_A_F_beta_n_gamma}) and (\ref{eq:B_subset_A_n}), we have
\begin{align}
A^F_{n,\beta^*,\gamma} &\supset \left\{x\in\mathbb{R}^d:\forall i\in S',\forall y\in\{0,1\}, x_i \in B_{n,\gamma,i,y}\right\} \nonumber \\
&=\cap_{i\in S'} \cap_{y\in\{0,1\}} \left\{x\in\mathbb{R}^d: x_i \in B_{n,\gamma,i,y}\right\} \nonumber
\end{align}
and thus
\begin{align}
\PP(X\notin A^F_{n,\beta^*,\gamma}) &= \PP\left( \left. X\notin A^F_{n,\beta^*,\gamma} \right\vert Y=0\right) \le \sum_{i\in S'} \sum_{y\in\{0,1\}} \PP\left( \left. X_i \notin B_{n,\gamma,i,y} \right\vert Y=0 \right).
\label{eq:P_X_i_notin_B_beta_star_0}
\end{align}
Here the equality follows by the same argument in the derivation of (\ref{eq:P_X_notin_A_F_1}).  We fix an arbitrary $i\in S'$.  If $i\in S''$ as well, then (\ref{eq:P_X_i_notin_B_00}) and (\ref{eq:P_X_i_notin_B_mismatch}) continue to hold.  On the other hand, if $i\notin S''$, then our job is easier, because then $(X_i|Y=0)$ and $(X_i|Y=1)$ have the same $N(0,1)$ distribution, $F_{i|0}(X_i)$ and $F_{i|1}(X_i)$ are both uniformly distributed on $(0,1)$, so (\ref{eq:P_X_i_notin_B_00}), and (\ref{eq:P_X_i_notin_B_00}) with the replacement of $B_{n,\gamma,i,0}$ by $B_{n,\gamma,i,1}$ and $F_{i|0}(X_i)$ by $F_{i|1}(X_i)$ all hold.  Combining the two cases, from (\ref{eq:P_X_i_notin_B_beta_star_0}), we conclude that
%Hence, for $i\notin S''$, $\PP\left( \left. X_i \notin B_{n,\gamma,i,y} \right\vert Y=0 \right)=16 g(2n,\gamma)$ for $y\in\{0,1\}$.
\begin{align}
\PP\left( X \notin A^F_{n,\beta^*,\gamma} \right) \le J_{\mu}' s' e^{C\mu\sqrt{\gamma\log(n)}} g(2n,\gamma).
\label{eq:P_X_notin_A_F_beta_star_final}
\end{align}

%----------------------------------------------------------------------------------------------------

\subsubsection{The terms $\PP\left(X \notin A^{f,=}_{n,d,\gamma} \cap A^{f,\neq}_{n,d,\gamma}\right)$}

Recall that $A^{f,=}_{n,d,\gamma}$ is as defined in (\ref{eq:wt_A_f_eq_n_d_gamma}) and $A^{f,\neq}_{n,d,\gamma}$ is as defined in (\ref{eq:wt_A_f_neq_n_d_gamma}).  Note that
\begin{align}
A^{f,=}_{n,d,\gamma} \cap A^{f,\neq}_{n,d,\gamma} = \Big\{ x\in\mathbb{R}^d: &\forall i\notin S^f_x, x_i\in\cap_{y\in\{0,1\}} B^f_{h_{n,i},i,y}, \nonumber \\
\text{and}~&\forall i\in S^f_x, x_i \in A^f_{n,i} \cap \left( \cap_{y\in\{0,1\}} B^f_{h_{n,i},i,y} \right) \Big\} \nonumber \\
= \Big\{ x\in\mathbb{R}^d: &\forall i\in\{s''+1,\dots,d\}, x_i\in\cap_{y\in\{0,1\}} B^f_{h_{n,i},i,y}, \nonumber \\
\text{and}~&\forall i\in\{1,\dots,s''\}~\text{such that}~x_i=\mu/2, x_i\in\cap_{y\in\{0,1\}} B^f_{h_{n,i},i,y}, \nonumber \\
\text{and}~&\forall i\in\{1,\dots,s''\}~\text{such that}~x_i\neq\mu/2, x_i\in A^f_{n,i} \cap \left( \cap_{y\in\{0,1\}} B^f_{h_{n,i},i,y} \right) \Big\} \nonumber
\end{align}
Here the second step follows because, under the simple $(d,s'',\mu,\Sigma)$ Gaussian classification model, for all $x\in\mathbb{R}^d$, $\{s''+1,\dots,d\}\subset (S^f_x)^c$, and for all $i\in\{1,\dots,s''\}$, $\Delta\log f_i(x_i)=0$ and so $i\in (S^f_x)^c$ if and only if $x_i=\mu/2$.  For $n$ large enough, for all $i\in\{1,\dots,s''\}$, we have that (\ref{eq:density_t_large_2}) holds with $t$ replaced by $\mu/2$, and so $\mu/2\in A^f_{n,i}$.  Hence, for $n$ large enough, we have a cleaner characterization of $A^{f,=}_{n,d,\gamma} \cap A^{f,\neq}_{n,d,\gamma}$ given by
\begin{align}
A^{f,=}_{n,d,\gamma} \cap A^{f,\neq}_{n,d,\gamma} = \Big\{x\in\mathbb{R}^d: &\forall i\in\{s''+1,\dots,d\}, x_i\in \cap_{y\in\{0,1\}} B^f_{h_{n,i},i,y}, \nonumber \\
\text{and}~&\forall i\in\{1,\dots,s''\}, x_i \in A^f_{n,i} \cap \left( \cap_{y\in\{0,1\}} B^f_{h_{n,i},i,y} \right) \Big\} \nonumber \\
= \Big\{x\in\mathbb{R}^d: &\forall i\in\{1,\dots,d\}, x_i\in \cap_{y\in\{0,1\}} B^f_{h_{n,i},i,y}, \nonumber \\
\text{and}~&\forall i\in\{1,\dots,s''\}, x_i \in A^f_{n,i} \Big\}. \nonumber
\end{align}
We will proceed with this characterization.

We first show that, for all $i\in\{1,\dots,d\}$ and for $y\in\{0,1\}$, we have $B^f_{h_{n,i},i,y}=\mathbb{R}$ (recall $B^f_{h_{n,i},i,y}$ as defined in (\ref{def_B_f_c})).  It suffices to show this for $y=0$.  In this case the density function $f_{i|0}=\phi$.  We assume that $n$ is large enough such that $\phi(h_{n,i})\ge\underline{f_{n,i}}$.  By symmetry of the density function $\phi$ around zero and the monotonicity of $\phi$ on $[0,\infty)$, it suffices to show that, if $t\ge h_{n,i}$ and $\phi(t)\ge\underline{f_{n,i}}$, then $\phi(t-h_{n,i}) \le 2 \phi(t)$.  We have
\begin{align}
\dfrac{\phi(t-h_{n,i})}{\phi(t)} &= e^{h_{n,i}t-h_{n,i}^2/2} < e^{h_{n,i}t}.
\label{eq:f_ratio}
\end{align}
It is easy to derive that, for an arbitrary constant $L$,
\begin{align}
\phi(t)\ge L \underline{f_{n,i}} \Longleftrightarrow |t| \le \left[ \gamma\log(n) + 2 \log\left(\dfrac{ H_i \log^{-1}(n) }{\sqrt{2\pi} L J_{\gamma,C_d}} \right) \right]^{1/2}  \equalscolon q(n,L).
\label{eq:A_f_t}
\end{align}
In the above, for brevity, we have suppressed the display of the dependence of the function $q$ on other parameters.  Then, the restriction $\phi(t)\ge\underline{f_{n,i}}$ enforces the bound $t\le q(n,1)$, which, when plugged into (\ref{eq:f_ratio}), yields that, for $n$ large enough,
\begin{align}
\dfrac{\phi(t-h_{n,i})}{\phi(t)} &< e^{h_{n,i}t} \le e^{\log(2)} = 2 \nonumber
%\label{eq:f_ratio_2}
\end{align}
as desired.  Here the second inequality follows by the choices (\ref{eq:h_n_i_concrete_super_smooth}) of $h_{n,i}$ and (\ref{eq:H_i_condition_super_smooth}) of $H_i$.

Hence, $A^{f,=}_{n,d,\gamma} \cap A^{f,\neq}_{n,d,\gamma} = \left\{ x\in\mathbb{R}^d: \forall i\in\{1,\dots,s''\}, x_i \in A^f_{n,i} \right\}$, and it remains to bound
\begin{align}
\PP\left(X \notin A^{f,=}_{n,d,\gamma} \cap A^{f,\neq}_{n,d,\gamma}\right) &= \PP\left( \exists i\in\{1,\dots,s''\}, X_i \notin A^f_{n,i} \right) \le \sum_{ i\in S'' } \PP\left( X_i \notin A^f_{n,i} \right) \nonumber \\
&= \sum_{ i\in S'' } \PP\left( \exists y\in\{0,1\}, f_{i|y}(X_i) < \dfrac{3}{1-\epsilon_n} \underline{f_{n,i}}\right).
\label{eq:P_A_f_decompose}
\end{align}
We fix an arbitrary $i\in S''$.  We have, for $n$ large enough such that $\epsilon_n\le 1/4$, that
\begin{align}
f_{i|y}(t) < \dfrac{3}{1-\epsilon_n} \underline{f_{n,i}} &\Longrightarrow f_{i|y}(t) < 4 \underline{f_{n,i}} \Longleftrightarrow |t-\mu_y| > q(n,4)
\label{eq:f_i_y_lower_bound_t}
\end{align}
for $\mu_0=0$ and $\mu_1=\mu$.  Here the second equivalence follows by (\ref{eq:A_f_t}).  Then, from (\ref{eq:f_i_y_lower_bound_t}), we further have
\begin{align}
\exists y\in\{0,1\}, f_{i|y}(t) < \dfrac{3}{1-\epsilon_n} \underline{f_{n,i}} \Longrightarrow t \notin \left[ -q(n,4) +\mu, q(n,4) \right].
\label{eq:f_i_y_lower_bound_t_2}
\end{align}
From (\ref{eq:f_i_y_lower_bound_t_2}), we then have, for $n$ large enough,
\begin{align}
&\PP\left( \exists y\in\{0,1\}, f_{i|y}(X_i) < \dfrac{3}{1-\epsilon_n} \underline{f_{n,i}} \right) \le \PP\left( X_i < - q(n,4) +\mu \right) + \PP\left( X_i > q(n,4) \right) \nonumber \\
&= \PP\left( \left. X_i < - q(n,4) +\mu \right| Y=0 \right) + \PP\left( \left. X_i > q(n,4) \right| Y=0 \right) \nonumber \\
&\le \dfrac{1}{q(n,4)-\mu} \phi\left(q(n,4)-\mu\right) + \dfrac{1}{q(n,4)} \phi\left(q(n,4)\right) \le \dfrac{2}{q(n,4)-\mu} \phi\left(q(n,4)-\mu\right) \nonumber \\
&\le \dfrac{4}{q(n,4)}\left[ \phi(q(n,4))e^{\mu q(n,4)} e^{-\mu^2} \right] = \dfrac{16 e^{-\mu^2}}{q(n,4)} \underline{f_{n,i}} e^{\mu q(n,4)} \nonumber \\
&\le J_{\gamma,C_d,\mu} \log(n) e^{\mu \sqrt{\gamma\log(n)}} g(2n,\gamma).
\label{eq:P_f_i_y_lower_bound}
\end{align}
Here the first equality follows by the symmetry given the cases $Y=0$ and $Y=1$, the second inequality follows from (\ref{eq:Phi_phi}) and (\ref{eq:Phi_phi_sym}), the second equality follows because $\phi(q(n,L))=L \underline{f_{n,i}}$ by (\ref{eq:f_i_y_lower_bound_t}), and in the last inequality $J_{\gamma,C_d,\mu}$ is some constant dependent only on $\gamma,C_d,\mu$.

Then, from (\ref{eq:P_A_f_decompose}) and (\ref{eq:P_f_i_y_lower_bound}), we conclude that, for $n$ large enough,
\begin{align}
\PP\left(X \notin A^{f,=}_{n,d,\gamma} \cap A^{f,\neq}_{n,d,\gamma}\right) \le J_{\gamma,C_d,\mu} s'' e^{C \mu \sqrt{\gamma\log(n)}}  g(2n,\gamma).
\label{eq:P_X_notin_A_f}
\end{align}

Therefore, by the overall bound (\ref{eq:wt_A_n_d_gamma_c_prob}) and the individual bounds (\ref{eq:P_X_notin_A_F_1_final}), (\ref{eq:P_X_notin_A_F_3}), (\ref{eq:P_X_notin_A_F_beta_star_final}), and (\ref{eq:P_X_notin_A_f}), we conclude that, for $C_{\gamma,C_d,\mu}$ some constant dependent only on $\gamma,C_d,\mu$,
\begin{align}
\PP\left(X \notin A_{n,d,\gamma}\right) \le C_{\gamma,C_d,\mu} (s'+s'') e^{C \mu \sqrt{\gamma\log(n)}} g(2n,\gamma), \nonumber
\end{align}
which is (\ref{eq:case_study_P_A}).
\qed

%------------------------------------------------------------------------------------------------------------------------------------------------------
%------------------------------------------------------------------------------------------------------------------------------------------------------
%------------------------------------------------------------------------------------------------------------------------------------------------------
%------------------------------------------------------------------------------------------------------------------------------------------------------

\section{Proofs for Section~\ref{sec_copula}}
\label{sec:proof_for_sec_copula}

\subsection{Proof of Lemma~\ref{thm:delta_alpha}}
\label{sec:proof_thm:delta_alpha}

We first prove some basic building blocks toward the proof of Lemma~\ref{thm:delta_alpha} and other results in the paper.
\begin{proposition}
\label{prop:basic_facts_Phi_Phi_inv}
For all $t\ge 0$, we have
\begin{align}
\dfrac{t}{1+t^2}\phi(t) \le 1-\Phi(t) \le \dfrac{1}{t}\phi(t),
\label{eq:Phi_phi}
\end{align}
and for all $0.5\le t\le 1$, we have
\begin{align}
\Phi^{-1}(t) \le \sqrt{2\log\dfrac{1}{2(1-t)}}.
\label{eq:Phi_inv}
\end{align}
Therefore, by symmetry, for all $t\le 0$, we have
\begin{align}
\dfrac{-t}{1+(-t)^2}\phi(t) \le \Phi(t) \le \dfrac{1}{-t}\phi(t),
\label{eq:Phi_phi_sym}
\end{align}
and for all $0\le t\le 0.5$, we have
\begin{align}
\Phi^{-1}(t) \ge -\sqrt{2\log\dfrac{1}{2t}}.
\label{eq:Phi_inv_sym}
\end{align}
%(See Proposition~1 of Mai~\&~Zou.)
\end{proposition}

\begin{proof}
The proof can be found in Appendix~\ref{sec:proof_prop:basic_facts_Phi_Phi_inv}.
\end{proof}

\begin{proposition}
\label{prop:A_n_prob_bound}
Recall that $n$ is large enough such that $a_n\ge 1$.  We have, for $t$ such that $\alpha_{i|y}(t)=a_n$, that
\begin{align}
%\PP\left(\alpha_{i|y}(X_i)\notin [-a_n,a_n]|Y=y\right) \le 4g(n,\gamma)
%\label{eq:A_n_c_bound}
g(n,\gamma)\le 1 - F_{i|y}(t)\le 2g(n,\gamma),
\label{eq:F_edge_bound}
\end{align}
and for all $t\in\mathbb{R}$ such that $\alpha_{i|y}(t)\in [-a_n,a_n]$, that
\begin{align}
g(n,\gamma) \le F_{i|y}(t) \le 1 - g(n,\gamma).
\label{eq:F_bound}
\end{align}
\end{proposition}

\begin{proof}
The proof can be found in Appendix~\ref{sec:proof_prop:A_n_prob_bound}.
\end{proof}

\begin{proof}[Proof of Lemma~\ref{thm:delta_alpha}]

We will start by attempting to derive a version of Inequality~(\ref{eq:ineq_delta_master_1}) but for $t$ uniformly over the interval specified by $t:\alpha_{i|y}(t)\in [-a_n,a_n]$, then specialize to our pointwise case midway.  We focus on the case $\alpha_{i|y}(t)\ge 0$ and so $\Phi(\alpha_{i|y}(t))=F_{i|y}(t)\ge 1/2$.  The analysis for the symmetric case $\alpha_{i|y}(t)<0$ is similar and is thus omitted.

We let $\xi^1, \dots, \xi^n$ be independent Uniform $(0,1)$ random variables, and let $G_n$ be their empirical distribution function.  We define the event
\begin{align}
E_F = \left\{ \sup_{u\in\left[\frac{1}{2},1-g(n,\gamma)\right]}1- G_n(u) > \frac{1}{2}(1-u) \right\}.
\label{def_E_F}
\end{align}
In words, $E_F$ is the event on which $1-G_n(u)$ is not too small relative to $1-u$, uniformly for $u$ over the interval $[1/2,1-g(n,\gamma)]$.  We can replace the constant $1/2$ in front of $(1-u)$ in (\ref{def_E_F}) by $1-\frac{1}{2}\sqrt{g(n,\gamma)/(1-u)}$, a quantity bounded below by $1/2$ for $u$ over the same interval, and Lemma~\ref{lemma:P_E_F} will continue to hold.  However, such a choice at most affects some constant multiplicative factor later on.  For brevity of display, we do not pursue such a choice.% and can be replaced by any positive number less than one.

\begin{lemma}
\label{lemma:P_E_F}
The event $E_F$ satisfies
\begin{align}
\PP(E_F) \ge 1 - 6\log(g^{-1}(n,\gamma)/2) \exp\left( -\dfrac{1}{32}n\cdot g(n,\gamma) \right).
\label{eq:prob_E_F_c}
\end{align}
\end{lemma}

\begin{proof}
For brevity we write $g=g(n,\gamma)$.  We have
\begin{align}
E_F &= \left\{ \sup_{u\in[\frac{1}{2},1-g]} \dfrac{1}{n}\sum_{j=1}^n \left[1-\1\{\xi^j\le u\}\right] > \frac{1}{2}(1-u) \right\} = \left\{ \sup_{u\in[\frac{1}{2},1-g]} \dfrac{1}{n}\sum_{j=1}^n \1\{\xi^j> u\} > \frac{1}{2}(1-u) \right\} \nonumber \\
& = \left\{ \sup_{u\in[g,\frac{1}{2}]} \dfrac{1}{n}\sum_{j=1}^n \1\{1-\xi^j<u\} > \frac{1}{2}u \right\}.
\label{eq:E_F_der_2}
\end{align}
Note that $1-\xi^1,\dots,1-\xi^n$ are again independent Uniform $(0,1)$ random variables, with the same joint distribution as $\xi^1,\dots,\xi^n$.  Then, from (\ref{eq:E_F_der_2}), we have
\begin{align}
\PP(E_F) &= \PP\left( \sup_{u\in[g,\frac{1}{2}]} \dfrac{1}{n}\sum_{j=1}^n \1\{\xi^j<u\} > \frac{1}{2}u \right) = \PP\left( \sup_{u\in[g,\frac{1}{2}]} \dfrac{1}{n}\sum_{j=1}^n \1\{\xi^j\le u\} > \frac{1}{2}u \right) \nonumber \\
&= \PP\left( \sup_{u\in[g,\frac{1}{2}]} G_n(u) - u > -\frac{1}{2}u \right) = \PP\left( \sup_{u\in[g,\frac{1}{2}]} \left| G_n(u) - u \right|^- < \frac{1}{2} u \right) \nonumber \\
&\ge \PP\left( \sup_{u\in[g,\frac{1}{2}]} \left|\sqrt{n}\dfrac{ G_n(u) - u}{\sqrt{u}}\right|^- < \frac{1}{2} \sqrt{n g} \right) \ge 1- 6\log(g^{-1}/2)\exp\left(-\dfrac{1}{32}n\cdot g\right), \nonumber
%\label{eq:A_F_der_2}
\end{align}
which is what we would like to show.  Here the in the fourth equality $|f|^-=-\min\{f,0\}$ for a generic function $f$, and the last inequality follows from \cite[Chapter~11, Section~2, Corollary~1]{Shorack09}.
\end{proof}

We let $A_n'\subset\mathbb{R}$ be an arbitrary interval such that $A_n'\subset[0,a_n]$.  We let $A_n''\subset\mathbb{R}$ be such that $A_n''=\left\{\Phi(\alpha_{i|y}(t))=F_{i|y}(t):\alpha_{i|y}(t)\in A_n'\right\}$.  It is easy to see that $\{t\in\mathbb{R}:\alpha_{i|y}(t)\in A_n'\} = \{t\in\mathbb{R}:F_{i|y}(t)\in A_n''\}$, and because $\Phi$ is strictly increasing, $A_n''$ is an interval in $\mathbb{R}$ as well.  In addition, by the first half of Inequality~(\ref{eq:F_edge_bound}) in Proposition~\ref{prop:A_n_prob_bound}, we have
\begin{align}
A_n''\subset\left[\dfrac{1}{2},1-g(n,\gamma)\right].
\label{eq:A_n_pp_range}
\end{align}
Now, we have
\begin{align}
&\PP\left(\sup_{t\in\mathbb{R}:\alpha_{i|y}(t)\in A_n'}|\wh\alpha_{i|y}(t)-\alpha_{i|y}(t)|\ge\epsilon\right) = \PP\left(\sup_{t\in\mathbb{R}:F_{i|y}(t)\in A_n''} \left|\Phi^{-1}(\wh F_{i|y}(t))-\Phi^{-1}(F_{i|y}(t))\right|\ge\epsilon\right) \nonumber \\
&= \PP\left(\sup_{t\in\mathbb{R}:F_{i|y}(t)\in A_n''} \left|\Phi^{-1}(G_n(F_{i|y}(t)))-\Phi^{-1}(F_{i|y}(t))\right|\ge\epsilon\right) \nonumber \\
&= \PP\left(\sup_{u\in A_n''} \left|\Phi^{-1}(G_n(u))-\Phi^{-1}(u)\right|\ge\epsilon\right).
\label{eq:delta_alpha_ineq_0}
\end{align}
%= \PP\left(\sup_{t:\alpha_{i|y}(t)\in A_n'} \left|\Phi^{-1}(\wh F_{i|y}(t))-\Phi^{-1}(F_{i|y}(t))\right|\ge\epsilon\right) \nonumber \\
Here the second equality follows because the random functions $\wh F_{i|y}(\cdot)$ and $G_n(F_{i|y}(\cdot))$ have the same probabilistic behavior.

Using the mean value theorem and Inequality~(\ref{eq:Phi_inv}), we have
\begin{align}
|\Phi^{-1}(G_n(u))-\Phi^{-1}(u)| &= (\Phi^{-1})'(\eta(u)) |G_n(u)-u| = \sqrt{2\pi}\exp\left(\dfrac{\Phi^{-1}(\eta(u))^2}{2}\right) |G_n(u)-u| \nonumber \\
&\le \sqrt{\dfrac{\pi}{2}}\dfrac{1}{1-\eta(u)} |G_n(u)-u|, \quad \forall u\in A_n''.
\label{eq:alpha_Taylor}
\end{align}
Here
\begin{align}
\eta(u) \in \left[ \min\{G_n(u),u\}, \max\{G_n(u),u\} \right], \quad \forall u\in A_n''.
\label{def_xi}
\end{align}
We proceed to obtain a lower bound on $1-\eta(u)$ for $u\in A_n''$.  Starting from relationship~(\ref{def_xi}), on the event $E_F$ (as defined in (\ref{def_E_F})), we have
\begin{align}
1-\eta(u) &\ge 1 - \max\{G_n(u),u\} = \min\{ 1-G_n(u), 1-u \} \nonumber \\
&\ge \min\left\{ \dfrac{1}{2}(1-u), 1-u \right\} =\dfrac{1}{2}(1-u), \quad \forall u\in A_n''.
\label{eq:xi_upper_bound}
\end{align}
Here the second inequality follows because, for all $u\in A_n''$ and so $u\in[1/2, g(n,\gamma)]$ by (\ref{eq:A_n_pp_range}), $1-G_n(u)\ge\dfrac{1}{2}(1-u)$ on the event $E_F$.
%This proves (\ref{eq:xi_upper_bound}).

Combining Inequalities~(\ref{eq:alpha_Taylor}) and (\ref{eq:xi_upper_bound}) yields, on the event $E_F$, that
\begin{align}
|\Phi^{-1}(G_n(u))-\Phi^{-1}(u)| &\le \sqrt{2\pi}\dfrac{1}{1-u} |G_n(u)-u|, \quad \forall u\in A_n''.
\label{eq:delta_alpha_ineq_1}
\end{align}
Hence we have, from (\ref{eq:delta_alpha_ineq_0}) and (\ref{eq:delta_alpha_ineq_1}),
\begin{align}
&\PP\left(\sup_{t\in\mathbb{R}:\alpha_{i|y}(t)\in A_n'}|\wh\alpha_{i|y}(t)-\alpha_{i|y}(t)|\ge\epsilon\right) \nonumber \\
&= \PP\left( \left\{\sup_{u\in A_n''} \left|\Phi^{-1}(G_n(u))-\Phi^{-1}(u)\right|\ge\epsilon \right\}\cap E_F\right) + \PP\left( \left\{\sup_{u\in A_n''} \left|\Phi^{-1}(G_n(u))-\Phi^{-1}(u)\right|\ge\epsilon\right\}\cap E_F^c\right) \nonumber \\
&\le \PP\left( \left\{ \sup_{u\in A_n''}\sqrt{2\pi}\dfrac{1}{1-u} |G_n(u)-u| \ge\epsilon \right\} \cap E_F \right) + \PP( E_F^c) \nonumber \\
&\le \PP\left( \sup_{u\in A_n''}|G_n(u)-u| \ge \frac{1}{\sqrt{2\pi}}(1-u)\epsilon  \right) + \PP( E_F^c).
\label{eq:delta_alpha_ineq_2}
\end{align}
Here in the first inequality we have invoked (\ref{eq:delta_alpha_ineq_1}) on the event $E_F$.

At this point we can invoke \cite[Chapter~11, Section~2, Corollary~1]{Shorack09}, and follow essentially the same reasoning as the proof of Lemma~\ref{lemma:P_E_F}, to continue deriving the uniform version of Inequality~(\ref{eq:ineq_delta_master_1}).  Because the pointwise version suffices for our purpose and offers a somewhat faster convergence rate, we switch to focus on this case instead.  We choose $A_n''$ to be the singleton set $A_n''=\{F_{i|y}(t)\}=\{u\}$.  Then, applying \cite[Chapter~11, Section~1, Inequality~(ii)]{Shorack09} to the first term in Inequality~(\ref{eq:delta_alpha_ineq_2}) yields
%Thus $G_n(u)-u$ has variance $\dfrac{1}{n}F_{i|y}(t)(1-F_{i|y}(t))$.  Then, applying Bernstein's inequality on the first term in Inequality~(\ref{eq:delta_alpha_ineq_2}) yields (see also \cite[Theorem~2(i)]{Okamoto59})
\begin{align}
\PP(|\wh\alpha_{i|y}(t)-\alpha_{i|y}(t)|\ge\epsilon) & \le \PP\left( \left. |G_n(u)-u| \ge \frac{1}{\sqrt{2\pi}}(1-u)\epsilon \right|_{u=F_{i|y}(t)} \right) + \PP( E_F^c) \nonumber \\
&\le 2 \exp\left( -\dfrac{1}{4\pi} (1-F_{i|y}(t)) n \epsilon^2 \Psi\left(\epsilon/\sqrt{2\pi}\right) \right) + \PP( E_F^c) \nonumber \\
&\le 2 \exp\left( -\dfrac{1}{6\pi} (1-F_{i|y}(t)) n \epsilon^2  \right) + \PP( E_F^c) .
\label{eq:ineq_delta_alpha}
\end{align}
Here in the second inequality the function $\Psi$ is as defined in \cite[Chapter~11, Section~1, (2)]{Shorack09}, and in the third inequality we have invoked \cite[Chapter~11, Section~1, Proposition~1(12)]{Shorack09} and the assumption that $\epsilon\le\sqrt{2\pi}$.  Then, Inequality~(\ref{eq:ineq_delta_alpha_1}) is obtained by incorporating Inequality~(\ref{eq:ineq_delta_alpha}) with the bound on $\PP( E_F^c)$ as in (\ref{eq:prob_E_F_c}), and with the symmetric case $\alpha_{i|y}(t)< 0$.  From Inequality~(\ref{eq:ineq_delta_alpha_1}), if we further lower bound $\min\left\{F_{i|y}(t),1-F_{i|y}(t)\right\}$ using Inequality~(\ref{eq:F_bound}), we obtain Inequality~(\ref{eq:ineq_delta_master_1}).
\end{proof}

%----------------------------------------------------------------------------------------------------

\subsection{Proof of Theorem~\ref{thm:S_x_p_c_est}}
\label{sec:proof_thm:S_x_p_c_est}

We fix an arbitrary $x\in\mathbb{R}^d$ and an arbitrary $i\notin S_x''$.  We let
\begin{align}
t=x_i. \nonumber
\end{align}
By the construction of our test, we have
\begin{align}
\PP(\wt\Delta\alpha_i(t)=0) \ge \min\left\{ \PP\left(\text{test}~(\ref{eq:F_test_1})~\text{succeeds}\right), \PP\left(\text{test}~(\ref{eq:F_test_2})~\text{succeeds}\right) \right\}.
\label{eq:P_wt_Delta_alpha_equal_zero}
\end{align}

First, suppose that $t$ satisfies
\begin{align}
F_i(t) \le g(2n,\gamma) \quad \text{or} \quad F_i(t) \ge 1 - g(2n,\gamma).
\label{eq:F_t_small}
\end{align}
Then, either $F_i(t) \le g(2n,\gamma)$ or $1 - F_i(t) \le g(2n,\gamma)$ and so $F_i(t) ( 1 - F_i(t) ) \le g(2n,\gamma)$.  In this case, we focus on test~(\ref{eq:F_test_1}).  Note that we would like one of the Inequalities in (\ref{eq:F_test_1}) to hold so that we set $\wt\Delta\alpha_i(t)=0$.  We set $\epsilon = 3 g(2n,\gamma)$.  By Bernstein's inequality with $\VV\left[\wh F_i(t)\right] = \dfrac{1}{2n} F_i(t) (1-F_i(t))$, we have
%series((a+x)*log((a+x)/a)+(1-a-x)*log((1-a-x)/(1-a)))
\begin{align}
\PP\left( \wh F_i(t) - F_i(t) >\epsilon \right) & \le \exp\left( - \dfrac{ 4 n^2 \epsilon^2 }{ 4 n F_i(t) ( 1 - F_i(t) ) + 8 n \epsilon / 3 } \right) \nonumber \\
&\le \exp\left( -3 n \cdot g(2n,\gamma) \right). \nonumber
\end{align}
%Hence this probability is exponentially small if $\gamma<2$.
Hence, we conclude that, by (\ref{eq:P_wt_Delta_alpha_equal_zero}) and test~(\ref{eq:F_test_1}), for $\epsilon = 3 g(2n,\gamma)$ as chosen above, for all $t$ such that $F_i(t)\le g(2n,\gamma)$ (i.e., the first half of (\ref{eq:F_t_small})), we have
\begin{align}
\PP(\wt\Delta\alpha_i(t)=0) &\ge \PP\left( \wh F_i(t) \le 4 g(2n,\gamma) \right) \ge \PP\left( \wh F_i(t) - F_i(t) \le \epsilon \right) \nonumber \\
&\ge 1 - \exp\left( - 3 n\cdot g(2n,\gamma) \right) \ge 1 - 8\dfrac{1}{d}n^{-\gamma/2}. \nonumber
\end{align}
Here the last line follows by (\ref{eq:thm:S_x_p_c_est_n_condition}).  By similar reasoning, the same conclusion follows for all $t$ such that $F_i(t) \ge 1 - g(2n,\gamma)$ (i.e., the second half of (\ref{eq:F_t_small})).  Therefore we conclude that (\ref{eq:H_i_prob_S_x_p_c_individual}) holds for $t=x_i$ in the regime specified by (\ref{eq:F_t_small}).

Next suppose that, in contrast to~(\ref{eq:F_t_small}), $t$ is such that
\begin{align}
g(2n,\gamma)<F_i(t) < 1 - g(2n,\gamma).
\label{eq:F_t_large}
\end{align}
In this case, test~(\ref{eq:F_test_1}) is more likely to fail, so we switch to study test (\ref{eq:F_test_2}).  Note that we set $\wt\Delta\alpha_i(t)=0$ (the desirable case) if Inequality~(\ref{eq:F_test_2}) holds, and so we upper bound the probability that Inequality~(\ref{eq:F_test_2}) fails.

Note that when Inequality~(\ref{eq:F_test_2}) fails, at least one of the following four inequalities
\begin{align}
\max\{ \wh F_{i|0}(t), \wh F_{i|1}(t) \} &> (1+\bar\delta_{n,d,\gamma}) F_i(t), \nonumber \\
\min\{ \wh F_{i|0}(t), \wh F_{i|1}(t) \} &< (1-\bar\delta_{n,d,\gamma}) F_i(t), \nonumber \\
\max\{ 1 - \wh F_{i|0}(t), 1 - \wh F_{i|1}(t) \} &> (1+\bar\delta_{n,d,\gamma}) ( 1 - F_i(t) ), \nonumber \\
\min\{ 1 - \wh F_{i|0}(t), 1 - \wh F_{i|1}(t) \} &< (1-\bar\delta_{n,d,\gamma}) ( 1 - F_i(t) ) \nonumber
\end{align}
must hold.  (If none of the these inequalities holds, it is easy to see that Inequality~(\ref{eq:F_test_2}) must hold.)  Thus,
\begin{align}
&\left\{  \max\left\{ \dfrac{ \max\{ \wh F_{i|0}(t),\wh F_{i|1}(t) \} }{ \min\{ \wh F_{i|0}(t),\wh F_{i|1}(t) \} } ,  \dfrac{ \max\{ 1-\wh F_{i|0}(t),1-\wh F_{i|1}(t) \} }{ \min\{ 1-\wh F_{i|0}(t),1-\wh F_{i|1}(t) \} }  \right\} \le \dfrac{1+\bar\delta_{n,d,\gamma}}{1-\bar\delta_{n,d,\gamma}}  \right\}^c \nonumber \\
& \subset \left\{ \max\{ \wh F_{i|0}(t),\wh F_{i|1}(t) \} > (1+\bar\delta_{n,d,\gamma}) F_i(t) \right\} \nonumber \\
& \cup \left\{ \min\{ \wh F_{i|0}(t),\wh F_{i|1}(t) \} < (1-\bar\delta_{n,d,\gamma}) F_i(t) \right\} \nonumber \\
& \cup \left\{ \max\{ 1 - \wh F_{i|0}(t), 1 - \wh F_{i|1}(t) \} > (1+\bar\delta_{n,d,\gamma}) ( 1 - F_i(t) ) \right\} \nonumber \\
& \cup \left\{ \min\{ 1 - \wh F_{i|0}(t), 1 - \wh F_{i|1}(t) \} < (1-\bar\delta_{n,d,\gamma}) ( 1 - F_i(t) ) \right\} \nonumber \\
& = \left\{ \wh F_{i|0}(t) > (1+\bar\delta_{n,d,\gamma}) F_{i|0}(t) \right\} \cup \left\{ \wh F_{i|1}(t) > (1+\bar\delta_{n,d,\gamma}) F_{i|1}(t) \right\} \nonumber \\
& \cup \left\{ \wh F_{i|0}(t) < (1-\bar\delta_{n,d,\gamma}) F_{i|0}(t) \right\} \cup \left\{ \wh F_{i|1}(t) < (1-\bar\delta_{n,d,\gamma}) F_{i|1}(t) \right\} \nonumber \\
& \cup \left\{ 1 - \wh F_{i|0}(t) > (1+\bar\delta_{n,d,\gamma}) ( 1 - F_{i|0}(t) ) \right\} \cup \left\{ 1 - \wh F_{i|1}(t) > (1+\bar\delta_{n,d,\gamma}) ( 1 - F_{i|1}(t) ) \right\} \nonumber \\
& \cup \left\{ 1 - \wh F_{i|0}(t) < (1-\bar\delta_{n,d,\gamma}) ( 1 - F_{i|0}(t) ) \right\} \cup \left\{ 1 - \wh F_{i|1}(t) < (1-\bar\delta_{n,d,\gamma}) ( 1 - F_{i|1}(t) ) \right\}.
\label{eq:F_t_ratio_breakdown}
\end{align}
Here the last step holds because in the current case $F_i(t)=F_{i|0}(t)=F_{i|1}(t)$.  Hence it suffices to bound the individual probabilities of the eight events whose union constitute the last step of the set relationship~(\ref{eq:F_t_ratio_breakdown}).  Recall that for $t$ in the regime specified by (\ref{eq:F_t_large}), both $F_i(t)$ and $1-F_i(t)$ are lower bounded by $g(2n,\gamma)$, which allows us to apply appropriate Chernoff bounds for relative deviations.  For example, for the first of the eight events, by considering i.i.d. Bernoulli random variables $\1\{X^{0,j}_i\le t\}$, $j\in\{1,\dots,n\}$ with mean $F_{i|0}(t)>g(2n,\gamma)$, we have from \cite[Inequality~(6)]{Hagerup90} that
\begin{align}
&\PP\left( \wh F_{i|0}(t) > (1+\bar\delta_{n,d,\gamma}) F_{i|0}(t) \right) \le \exp\left( - \dfrac{1}{3} n\bar\delta_{n,d,\gamma}^2 F_{i|y}(t)\right) \nonumber \\
&\le \exp\left( - \dfrac{1}{3} n\bar\delta_{n,d,\gamma}^2 g(2n,\gamma) \right)  = \dfrac{1}{d} n^{-\gamma/2},
\label{eq:F_dev}
\end{align}
while for the last term, by considering i.i.d. Bernoulli random variables $1-\1\{X^{1,j}_i\le t\}$, $j\in\{1,\dots,n\}$ with mean $1-F_{i|1}(t)>g(2n,\gamma)$, we have from \cite[Inequality~(7)]{Hagerup90} that
\begin{align}
&\PP\left( 1 - \wh F_{i|1}(t) < (1-\bar\delta_{n,d,\gamma}) ( 1 - F_{i|1}(t) ) \right) \le \exp\left( - \dfrac{1}{2} n\bar\delta_{n,d,\gamma}^2 ( 1-F_{i|1}(t) ) \right) \nonumber \\
&\le \exp\left( - \dfrac{1}{3} n\bar\delta_{n,d,\gamma}^2 g(2n,\gamma) \right)  = \dfrac{1}{d} n^{-\gamma/2}.
\label{eq:F_dev_2}
\end{align}
Here the last step of Inequalities~(\ref{eq:F_dev}) and (\ref{eq:F_dev_2}) hold by the choice of $\bar\delta_{n,d,\gamma}$ in (\ref{eq:def_bar_delta}).  Identical bounds are obtained for the other terms in the last step of (\ref{eq:F_t_ratio_breakdown}).

Hence, we conclude that, by (\ref{eq:P_wt_Delta_alpha_equal_zero}) and test~(\ref{eq:F_test_2}), for all $t$ such that $g(2n,\gamma)< F_{i|y}(t)<1-g(2n,\gamma)$, we have
\begin{align}
\PP(\wt\Delta\alpha_i(t)=0) &\ge \PP  \left(  \max\left\{ \dfrac{ \max\{ \wh F_{i|0}(t),\wh F_{i|1}(t) \} }{ \min\{ \wh F_{i|0}(t),\wh F_{i|1}(t) \} } ,  \dfrac{ \max\{ 1-\wh F_{i|0}(t),1-\wh F_{i|1}(t) \} }{ \min\{ 1-\wh F_{i|0}(t),1-\wh F_{i|1}(t) \} }  \right\} \le \dfrac{1+\bar\delta_{n,d,\gamma}}{1-\bar\delta_{n,d,\gamma}} \right) \nonumber \\
&\ge 1 - 8 \dfrac{1}{d} n^{-\gamma/2}. \nonumber
\end{align}
Therefore we conclude that (\ref{eq:H_i_prob_S_x_p_c_individual}) holds for $t=x_i$ in the regime specified by (\ref{eq:F_t_large}).  Combining with our earlier display, we conclude that (\ref{eq:H_i_prob_S_x_p_c_individual}) holds for all $x\in\mathbb{R}^d$ and $i\notin S_x''$.

Finally, as stated in the theorem, (\ref{eq:H_i_prob_S_x_p_c_union}) follows from (\ref{eq:H_i_prob_S_x_p_c_individual}) by a union bound argument.
\qed

%----------------------------------------------------------------------------------------------------

\subsection{Proof of Theorem~\ref{thm:S_x_p_est_1}}
\label{sec:proof_thm:S_x_p_est_1}

We fix an arbitrary $x\in\mathbb{R}^d$ satisfying Assumption~\ref{ass_x}, and an arbitrary $i\in S_x''$.  We let
$$t = x_i.$$
We first show that test~(\ref{eq:F_test_1}) fails with overwhelming probability.  Assumption~\ref{ass_x}, in particular (\ref{eq:F_t_large_2}), implies that
\begin{align}
8 g(2n,\gamma) \le F_i(t) \le 1 - 8 g(2n,\gamma).
\label{eq:F_t_large_3}
\end{align}
Then, on the one hand, we have $4 g(2n,\gamma)/F_i(t)\le 1/2$ by $(\ref{eq:F_t_large_3})$.  Thus,%So again by the Chernoff bound for relative deviations,
\begin{align}
&\PP\left( \wh F_i(t) \le 4 g(2n,\gamma) \right) = \PP\left( \wh F_i(t) \le \dfrac{4 g(2n,\gamma)}{F_i(t)} F_i(t) \right) \le \PP\left( \wh F_i(t) \le \dfrac{1}{2} F_i(t) \right) \nonumber \\
& = \PP\left( \dfrac{1}{2} \left[ \wh F_{i|0}(t) + \wh F_{i|1}(t) \right] \le \dfrac{1}{2} \cdot \dfrac{1}{2} \left[ F_{i|0}(t) + F_{i|1}(t) \right] \right) \le \sum_{ y\in\{0,1\} } \PP\left( \wh F_{i|y}(t) \le \dfrac{1}{2} F_{i|y}(t) \right) \nonumber \\
& \le 2 \exp\left( -\dfrac{1}{8} n F_{i|y}(t) \right) \le 2 \exp\left( - n\cdot g(2n,\gamma) \right). \nonumber
\end{align}
Here, in the third inequality we have used Chernoff bound for relative deviations, and in the last inequality we have used $F_{i|y}(t)\ge 8 g(2n,\gamma)$ as in (\ref{eq:F_t_large_2}).  On the other hand, we also have $1-F_{i|y}(t)\ge 8 g(2n,\gamma)$ by $(\ref{eq:F_t_large_3})$ and so $4 g(2n,\gamma)/(1-F_{i|y}(t))\le 1/2$.  Thus,
\begin{align}
&\PP\left( \wh F_i(t) \ge 1 - 4 g(2n,\gamma) \right) = \PP\left( 1 - \wh F_i(t) \le 4 g(2n,\gamma) \right) = \PP\left( 1 - \wh F_i(t) \le \dfrac{ 4 g(2n,\gamma)}{1 - F_i(t)} ( 1 - F_i(t) ) \right) \nonumber \\
&\le \PP\left( 1 - \wh F_i(t) \le \dfrac{1}{2} ( 1 - F_i(t) ) \right) \le \sum_{ y\in\{0,1\} } \PP\left( 1 - \wh F_{i|y}(t) \le \dfrac{1}{2} ( 1 - F_{i|y}(t) ) \right) \nonumber \\
& \le 2 \exp\left( -\dfrac{1}{8} n (1 -  F_{i|y}(t) ) \right) \le 2 \exp\left( - n\cdot g(2n,\gamma) \right). \nonumber
\end{align}
Here, in the third inequality we have again used Chernoff bound for relative deviations, and in the last inequality we have used $1-F_{i|y}(t)\ge 8 g(2n,\gamma)$ as in (\ref{eq:F_t_large_2}).  Combining the above displays, we conclude that
\begin{align}
\PP\left(\text{test}~(\ref{eq:F_test_1})~\text{fails}\right) \ge 1 - 4 \exp\left( - n\cdot g(2n,\gamma) \right) \ge 1 - n^{-\gamma/2}.
\label{eq:F_test_1_prob_S_x_p}
\end{align}
Here the second inequality follows by Assumption~\ref{ass_n}.

Next we discuss test~(\ref{eq:F_test_2}).  By Assumption~\ref{ass_x}, one of the inequalities (\ref{eq:F_ratio_1}), (\ref{eq:F_ratio_2}), (\ref{eq:F_ratio_3}), (\ref{eq:F_ratio_4}) hold.  First, let's assume that Inequality~(\ref{eq:F_ratio_1}) holds.  For test~(\ref{eq:F_test_2}) to fail, it suffices to have that both
\begin{align}
\wh F_{i|0}(t) \ge (1-\bar\delta_{n,1,\gamma})F_{i|0}(t)
\label{eq:F_dev_3}
\end{align}
and
\begin{align}
\wh F_{i|1}(t) \le (1+\bar\delta_{n,1,\gamma})F_{i|1}(t)
\label{eq:F_dev_4}
\end{align}
hold, because then we have
\begin{align}
\dfrac{ 1+\bar\delta_{n,d,\gamma} }{ 1-\bar\delta_{n,d,\gamma} } &< \dfrac{ (1-\bar\delta_{n,1,\gamma}) F_{i|0}(t) }{ (1+\bar\delta_{n,1,\gamma}) F_{i|1}(t) } \le \dfrac{ \wh F_{i|0}(t) }{ \wh F_{i|1}(t) } = \dfrac{ \max\{ \wh F_{i|0}(t),\wh F_{i|1}(t) \} }{ \min\{ \wh F_{i|0}(t),\wh F_{i|1}(t) \} }. \nonumber
\end{align}
Here the first inequality follows by (\ref{eq:F_ratio_1}), and the second inequality follows by (\ref{eq:F_dev_3}) and (\ref{eq:F_dev_4}).  By similar derivation as Inequalities~(\ref{eq:F_dev}) and (\ref{eq:F_dev_2}) with $\bar\delta_{n,d,\gamma}$ replaced by $\bar\delta_{n,1,\gamma}$, both Inequalities~(\ref{eq:F_dev_3}) and (\ref{eq:F_dev_4}) hold with probabilities at least $1- n^{-\gamma/2}$.  Hence
\begin{align}
\PP \left(  \dfrac{ \max\{ \wh F_{i|0}(t),\wh F_{i|1}(t) \} }{ \min\{ \wh F_{i|0}(t),\wh F_{i|1}(t) \} } > \dfrac{ 1+\bar\delta_{n,d,\gamma} }{ 1-\bar\delta_{n,d,\gamma} } \right) \ge 1- 2 n^{-\gamma/2},
\label{eq:max_ratio_bound_1}.
\end{align}
By a similar derivation, Inequality~(\ref{eq:F_ratio_2}) implies (\ref{eq:max_ratio_bound_1}) as well.  Now, let's assume that Inequality~(\ref{eq:F_ratio_4}) holds.  For test~(\ref{eq:F_test_2}) to fail, it suffices to have that both
\begin{align}
1 - \wh F_{i|1}(t) \ge (1-\bar\delta_{n,1,\gamma}) ( 1 - F_{i|1}(t) )
\label{eq:F_dev_5}
\end{align}
and
\begin{align}
1 - \wh F_{i|0}(t) \le (1+\bar\delta_{n,1,\gamma}) ( 1 - F_{i|0}(t) )
\label{eq:F_dev_6}
\end{align}
hold, because then we have
\begin{align}
\dfrac{ 1+\bar\delta_{n,d,\gamma} }{ 1-\bar\delta_{n,d,\gamma} } & < \dfrac{ (1-\bar\delta_{n,1,\gamma}) (1 - F_{i|1}(t) ) }{ (1+\bar\delta_{n,1,\gamma}) ( 1 - F_{i|0}(t) ) } \le \dfrac{ 1 - \wh F_{i|1}(t) }{ 1 - \wh F_{i|0}(t) } = \dfrac{ \max\{ 1 - \wh F_{i|0}(t), 1 - \wh F_{i|1}(t) \} }{ \min\{ 1 - \wh F_{i|0}(t), 1 - \wh F_{i|1}(t) \} } . \nonumber
\end{align}
By similar derivation as Inequalities~(\ref{eq:F_dev}) and (\ref{eq:F_dev_2}) with $\bar\delta_{n,d,\gamma}$ replaced by $\bar\delta_{n,1,\gamma}$, both Inequalities~(\ref{eq:F_dev_5}) and (\ref{eq:F_dev_6}) hold with probabilities at least $1- n^{-\gamma/2}$.  Hence,
\begin{align}
\PP \left(  \dfrac{ \max\{ 1 - \wh F_{i|0}(t), 1 - \wh F_{i|1}(t) \} }{ \min\{ 1 - \wh F_{i|0}(t), 1 - \wh F_{i|1}(t) \} } > \dfrac{ 1+\bar\delta_{n,d,\gamma} }{ 1-\bar\delta_{n,d,\gamma} } \right) \ge 1- 2 n^{-\gamma/2}.
\label{eq:max_ratio_bound_2}
\end{align}
By a similar derivation, Inequality~(\ref{eq:F_ratio_3}) implies (\ref{eq:max_ratio_bound_2}) as well.

Hence, we conclude that
\begin{align}
\PP\left(\text{test}~(\ref{eq:F_test_2})~\text{fails}\right) \ge 1- 2 n^{-\gamma/2}.
\label{eq:F_test_2_prob_S_x_p}
\end{align}

By (\ref{eq:F_test_1_prob_S_x_p}) and (\ref{eq:F_test_2_prob_S_x_p}), and the fact that if (\ref{eq:F_test_2}) is violated then necessarily $\wh F_{i|0}(t) \neq \wh F_{i|1}(t)$ and so $\wt\Delta\alpha_i(t)\neq 0$ (recall the definition of $\wt\Delta\alpha_i(t)$ as in (\ref{eq:def_wt_Delta_alpha_i_t})), we conclude that
\begin{align}
\PP\left(\wt\Delta\alpha_i(t)\neq 0\right) \ge 1- 3 n^{-\gamma/2}. \nonumber
\end{align}
Therefore we conclude that Inequality~(\ref{eq:H_i_prob_S_x_p_individual}) holds for $t=x_i$.  Then, as stated in the theorem, (\ref{eq:S_x_p_consistency}) follows from (\ref{eq:H_i_prob_S_x_p_individual}) by a union bound argument, and Theorem~\ref{thm:S_x_p_c_est}, in particular (\ref{eq:H_i_prob_S_x_p_c_union}).
%and the definition of $H_{i,x_i}$ as in (\ref{eq:def_H_i_t}))

Next we prove (\ref{eq:H_x_p_prob}).  Note that
\begin{align}
&\left\{ |\wt\Delta\alpha_i(t) - \Delta\alpha_i(t) | \ge 2\epsilon \right\} \nonumber \\
&= \left( \left\{ |\wt\Delta\alpha_i(t) - \Delta\alpha_i(t) | \ge 2\epsilon \right\}\cap \left\{\wt\Delta\alpha_i(t)=0\right\}\right) \cup \left( \left\{ |\wt\Delta\alpha_i(t) - \Delta\alpha_i(t) | \ge 2\epsilon \right\}\cap \left\{\wt\Delta\alpha_i(t)\neq0\right\}\right) \nonumber \\
&= \left( \left\{ |\wt\Delta\alpha_i(t) - \Delta\alpha_i(t) | \ge 2\epsilon \right\}\cap \left\{\wt\Delta\alpha_i(t)=0\right\}\right) \nonumber \\
& \cup \left( \left\{ | \wh\alpha_{i|0}(t) - \wh\alpha_{i|1}(t) - (\alpha_{i|0}(t) - \alpha_{i|1}(t) ) | \ge 2\epsilon \right\}\cap \left\{\wt\Delta\alpha_i(t)\neq0\right\}\right) \nonumber \\
&\subset \left\{\wt\Delta\alpha_i(t)=0\right\} \cup \left\{ |(\wh\alpha_{i|0}(t) - \wh\alpha_{i|1}(t) ) - (\alpha_{i|0}(t) - \alpha_{i|1}(t) ) | \ge 2\epsilon \right\} \nonumber \\
&\subset \left\{\wt\Delta\alpha_i(t)=0\right\} \cup \left\{ |\wh\alpha_{i|0}(t) - \alpha_{i|0}(t)| \ge \epsilon \right\} \cup \left\{ |\wh\alpha_{i|1}(t) - \alpha_{i|1}(t) | \ge \epsilon \right\}. \nonumber
\end{align}
Hence, by De Morgan's law,
\begin{align}
\left\{ |\wt\Delta\alpha_i(t) - \Delta\alpha_i(t) | < 2\epsilon \right\} \supset \left\{\wt\Delta\alpha_i(t)\neq 0\right\} \cap \left\{ |\wh\alpha_{i|0}(t) - \alpha_{i|0}(t)| < \epsilon \right\} \cap \left\{ |\wh\alpha_{i|1}(t) - \alpha_{i|1}(t) | < \epsilon \right\}, \nonumber
\end{align}
and thus, after taking intersections over $i\in S_x''$, we have
\begin{align}
&\cap_{i\in S_x''} \left( \left\{ |\wh\alpha_{i|0}(x_i)-\alpha_{i|0}(x_i)|<\epsilon \right\} \cap \left\{ |\wh\alpha_{i|1}(x_i)-\alpha_{i|1}(x_i)|<\epsilon \right\} \cap \{ |\wt\Delta\alpha_i(x_i) - \Delta\alpha_i(x_i) | < 2\epsilon \}\right) \nonumber \\
&\supset \cap_{i\in S_x''} \left( \left\{\wt\Delta\alpha_i(x_i)\neq 0\right\} \cap \left\{ |\wh\alpha_{i|0}(x_i) - \alpha_{i|1}(x_i)| < \epsilon \right\} \cap \left\{ |\wh\alpha_{i|1}(x_i) - \alpha_{i|1}(x_i) | < \epsilon \right\} \right).
\label{eq:ineq_intermediate_1}
\end{align}
Set relationship (\ref{eq:ineq_intermediate_1}) further implies that
\begin{align}
H'_{x,\epsilon} \supset \{\wt S_x''=S_x''\} \cap \left( \cap_{i\in S_x''} \cap_{y\in\{0,1\}}  \left\{ |\wh\alpha_{i|y}(x_i) - \alpha_{i|y}(x_i)| < \epsilon \right\}   \right).
\label{eq:H_x_p_intermediate}
\end{align}
Then, Inequality~(\ref{eq:H_x_p_prob}) follows from set relationship~(\ref{eq:H_x_p_intermediate}), Inequality~(\ref{eq:S_x_p_consistency}), Assumption~\ref{ass_x} and the observation following (\ref{eq:B_subset_A_n}) for $i\in S_x''$, $y\in\{0,1\}$, and Assumption~\ref{ass_n}.
\qed

%----------------------------------------------------------------------------------------------------

\subsection{Proof of Theorem~\ref{thm_L}}
\label{sec:proof_thm_L}
We fix an arbitrary $i\in s_x'$.  We have
%We assume that $n$ is large enough.  Then, Assumption~\ref{ass_n} holds under Assumption~\ref{ass_d_n}.
\begin{align}
&\left\{ |\wh\beta_i(x)-\beta^*_i(x)| \le 2 (M-1) \epsilon + 2 J_2 \kappa M s \sqrt{\gamma \log(n)} \lambda_n + 2 J_2 \kappa M s \lambda_n \epsilon \right\} \cap H'_{x,\epsilon} \cap E_n \nonumber \\
&= \left\{ \left| \left[ \wt\Omega - I_d \right]_{i\cdot} \wt\Delta\alpha(x) - \left[ \Omega - I_d \right]_{i\cdot} \Delta\alpha(x) \right| \le  2 (M-1) \epsilon + 2 J_2 \kappa M s \sqrt{\gamma \log(n)} \lambda_n + 2 J_2 \kappa M s \lambda_n \epsilon \right\} \nonumber \\
&\cap H'_{x,\epsilon} \cap E_n \nonumber \\
&\supset \left\{ \left|  \left[ \Omega - I_d \right]_{i\cdot} (\wt\Delta\alpha(x) - \Delta\alpha(x)) \right| \le 2 (M-1) \epsilon \right\} \cap H'_{x,\epsilon} \nonumber \\
&\cap \left\{ \left| \left[ \wt\Omega - \Omega \right]_{i\cdot} \Delta\alpha(x) \right| \le 2 J_2 \kappa M s \sqrt{\gamma \log(n)} \lambda_n \right\} \cap E_n \nonumber \\
&\cap \left\{ \left| \left[ \wt\Omega - \Omega \right]_{i\cdot} (\wt\Delta\alpha(x)-\Delta\alpha(x)) \right| \le 2 J_2 \kappa M s \lambda_n \epsilon \right\} \cap H'_{x,\epsilon} \cap E_n.
\label{eq:proof_thm_L_der_1}
\end{align}
As mentioned earlier, the diagonal elements of $\Omega$ are bounded below by one, and we are assuming $\Omega\in\calU(s,M,\kappa)$; hence, $\|\left[ \Omega - I_d \right]_{i\cdot}\|_{\ell_1}=\|\left[ \Omega \right]_{i\cdot}\|_{\ell_1} -1\le M-1$.  Also note that, for two vectors $u,v\in\mathbb{R}^d$, we have $|u^T v|\le \|u\|_{\ell_1} \|v\|_{\max}$.  Finally, on the event $H'_{x,\epsilon}\subset\{ \wt S_x'' = S_x'' \}$, $|\wt\Delta\alpha_i(x_i) - \Delta\alpha_i(x_i)|$ can be nonzero only if $i\in S_x''$.  Then, from (\ref{eq:proof_thm_L_der_1}), for $n$ large enough,
\begin{align}
&\left\{ |\wh\beta_i(x)-\beta^*_i(x)| \le 2 (M-1) \epsilon + 2 J_2 \kappa M s \sqrt{\gamma \log(n)} \lambda_n + 2 J_2 \kappa M s \lambda_n \epsilon \right\} \cap H'_{x,\epsilon} \cap E_n \nonumber \\
&\supset  \left\{ \left(M-1\right) \max_{i\in S_x''} |\wt\Delta\alpha_i(x_i) - \Delta\alpha_i(x_i)| \le 2 (M-1) \epsilon \right\} \cap H'_{x,\epsilon} \nonumber \\
&\cap \left\{ \| \wt\Omega - \Omega\|_{\infty} \max_{i\in S_x''}|\Delta\alpha_i(x_i)| \le 2 J_2 \kappa M s \sqrt{\gamma \log(n)} \lambda_n \right\} \cap E_n \nonumber \\
&\cap \left\{ \| \wt\Omega - \Omega\|_{\infty} \max_{i\in S_x''} |\wt\Delta\alpha_i(x_i)-\Delta\alpha_i(x_i) | \le 2 J_2 \kappa M s \lambda_n \epsilon \right\} \cap H'_{x,\epsilon} \cap E_n \nonumber \\
%&\supset \left\{ \max_{i\in S_x''} |\wt\Delta\alpha_i(x_i) - \Delta\alpha_i(x_i)| \le 2\epsilon \right\} \cap H'_{x,\epsilon} \nonumber \\
%&\cap \left\{ \| \wt\Omega - \Omega\|_{\infty} \le J_2 \kappa M s \lambda_n \right\} \cap E_n \nonumber \\
%&\cap \left\{ \| \wt\Omega - \Omega\|_{\infty} \max_{i\in S_x''} |\wt\Delta\alpha_i(x_i)-\Delta\alpha_i(x_i) | \le 2 J_2 \kappa M s \lambda_n \epsilon \right\} \cap H'_{x,\epsilon} \cap E_n \nonumber \\
&\supset H'_{x,\epsilon} \cap E_n . \nonumber
\end{align}
Here the last set step follows by the definition of $H'_{x,\epsilon}$ as in (\ref{eq:H_x_p}), Proposition~\ref{thm_Omega} regarding $\| \wt\Omega - \Omega\|_{\infty}$ on $E_n$, and the fact that $\max_{i\in S_x''}|\Delta\alpha_i(x_i)|\le\max_{i\in S_x''}(|\alpha_{i|0}(x_i)|+|\alpha_{i|1}(x_i)|)\le 2a_n=2\sqrt{\gamma\log(n)}$, which follows by Assumptions~\ref{ass_x}, in particular (\ref{eq:B_subset_A_n}).  In addition, by the choices of $H'_{x,\epsilon}$, $E_n$ and Proposition~\ref{thm_Omega}, we have, for $n$ large enough,
\begin{align}
H'_{x,\epsilon} \cap E_n &\subset \{ \wt S_x'' = S_x'' \}\cap\{ \sgn(\wt\Omega-I_d) =\sgn(\Omega-I_d) \} = \{\wt S_x' = S_x'\} . \nonumber
\end{align}
Thus, we conclude that
$$ H'_{x,\epsilon} \cap E_n \subset L_{x,\epsilon} $$
and hence
$$\PP(L_{x,\epsilon})\ge\PP(H'_{x,\epsilon} \cap E_n).$$
Then, (\ref{eq:L_prob}) follows from Inequality~(\ref{eq:H_x_p_prob}) in Theorem~\ref{thm:S_x_p_est_1} (which applies because Assumption~\ref{ass_n} holds under Assumption~\ref{ass_d_n} for $n$ large enough) for $H'_{x,\epsilon}$ and Inequality~(\ref{eq:P_Omega_estimation}) in Proposition~\ref{prop:Zhao14_Thm_IV.5} for $E_n$.
\qed

%----------------------------------------------------------------------------------------------------

\subsection{Proof of Theorem~\ref{thm_est_copula_part}}
\label{sec:proof_thm_est_copula_part}
We assume that $n$ is large enough.  We have
\begin{align}
& \left| \left[(\wh\alpha_0 + \wh\alpha_1)^T \wh\beta - (\alpha_0 + \alpha_1)^T \beta^*\right](x) \right| \nonumber \\
&\le |(\wh\alpha_0(x) + \wh\alpha_1(x) - \alpha_0(x) - \alpha_1(x))^T \beta^*(x)| +  |(\alpha_0(x) + \alpha_1(x))^T(\wh\beta(x) - \beta^*(x))| \nonumber \\
&+ |(\wh\alpha_0(x) + \wh\alpha_1(x) - \alpha_0(x) - \alpha_1(x))^T(\wh\beta(x) - \beta^*(x))| \nonumber \\
& \le \max_{i\in S_x}\left(|\wh\alpha_{i|0}(x_i) - \alpha_{i|0}(x_i)| + |\wh\alpha_{i|1}(x_i) - \alpha_{i|1}(x_i)|\right) \|\beta^*(x)\|_{\ell_1} \nonumber \\
& + |S_x'\cup \wt S_x'| \max_{i \in S_x'\cup\wt S_x'} | \alpha_{i|0}(x_i) + \alpha_{i|1}(x_i) |  \max_{i \in S_x'\cup\wt S_x'} |\wh\beta_i(x) - \beta^*_i(x)| \nonumber \\
& + |S_x'\cup \wt S_x'| \max_{i \in S_x'\cup\wt S_x'} \left(|\wh\alpha_{i|0}(x_i) - \alpha_{i|0}(x_i)| + |\wh\alpha_{i|1}(x_i) - \alpha_{i|1}(x_i)|\right) \max_{i \in S_x'\cup\wt S_x'} |\wh\beta_i(x) - \beta^*_i(x)| . \nonumber
%\label{eq:bound_copula_part}
\end{align}
Here in the second inequality we have invoked (\ref{eq:wh_beta_support}).  Thus, by Theorem~\ref{thm_L}, on the event $L'_{x,\epsilon}$ (on which $\{\wt S_x'=S_x'\}$ through the event $L_{x,\epsilon}$ as defined in (\ref{eq:def_L_x_gamma})), we have from the above
\begin{align}
&\left| \left[(\wh\alpha_0 + \wh\alpha_1)^T \wh\beta - (\alpha_0 + \alpha_1)^T \beta^*\right](x) \right| \nonumber \\
&\le \max_{i\in S_x}\left(|\wh\alpha_{i|0}(x_i) - \alpha_{i|0}(x_i)| + |\wh\alpha_{i|1}(x_i) - \alpha_{i|1}(x_i)|\right) \|\beta^*(x)\|_{\ell_1} \nonumber \\
&+ s_x' \max_{i \in S_x'} | \alpha_{i|0}(x_i) + \alpha_{i|1}(x_i) | \max_{i \in S_x'} |\wh\beta_i(x) - \beta^*_i(x)| \nonumber \\
&+ s_x' \max_{i \in S_x'}\left(|\wh\alpha_{i|0}(x_i) - \alpha_{i|0}(x_i)| + |\wh\alpha_{i|1}(x_i) - \alpha_{i|1}(x_i)|\right) \max_{i \in S_x'} |\wh\beta_i(x) - \beta^*_i(x)| \nonumber \\
&\le 2\epsilon\|\beta^*(x)\|_{\ell_1} + 2 s_x' \sqrt{\gamma\log(n)} \left[ 2 (M-1) \epsilon + 2 J_2 \kappa M s \sqrt{\gamma \log(n)} \lambda_n + 2 J_2 \kappa M s \lambda_n \epsilon \right] \nonumber \\
& + 2 s_x' \epsilon \left[ 2 (M-1) \epsilon + 2 J_2 \kappa M s \sqrt{\gamma \log(n)} \lambda_n + 2 J_2 \kappa M s \lambda_n \epsilon \right]. \nonumber
\end{align}
Here in the second inequality we have invoked Assumption~\ref{ass_x_beta_star}.  Hence, we have shown (\ref{eq:L_p_x_error_bound}).

It remains to establish (\ref{eq:L_p_x_gamma_prob}).  Note that $L'_{x,\epsilon}$ differs from $L_{x,\epsilon}$ by at most a set in the parenthesis on the right hand side of (\ref{eq:def_L_p_x}), which has probability at least
\begin{align}
1 - 4 s_x' \exp\left( - J_1 \dfrac{n^{1-\gamma/2} \epsilon^2}{\sqrt{\gamma\log n}} \right) - 2 s_x' n^{-\gamma/2} \nonumber
\end{align}
by Lemma~\ref{thm:delta_alpha} and Assumptions~\ref{ass_d_n} and \ref{ass_x_beta_star}.  Combining this result with (\ref{eq:L_prob}), Inequality~(\ref{eq:L_p_x_gamma_prob}) then follows.
\qed

%----------------------------------------------------------------------------------------------------

\section{Proofs for Section~\ref{sec_naive_bayes}}
\label{sec:proof_for_sec_naive_bayes}

\subsection{Proof of Proposition~\ref{prop:f_rel_dev_variance}}
\label{sec:proof_prop:f_rel_dev_variance}
We have
\begin{align}
\wh f_{i|y}(t) - \EE \wh f_{i|y}(t) = \dfrac{1}{n} \sum_{j=1}^{n} \left\{ \dfrac{1}{h_{n,i}} K_i\left(\dfrac{X^{y,j}_i-t}{h_{n,i}}\right) - \EE \left[ \dfrac{1}{h_{n,i}} K_i\left(\dfrac{X^{y,j}_i-t}{h_{n,i}}\right) \right] \right\}. \nonumber
\end{align}
Note that
\begin{align}
&\VV\left\{ K_i\left(\dfrac{X^{y,j}_i-t}{h_{n,i}}\right) - \EE \left[ K_i\left(\dfrac{X^{y,j}_i-t}{h_{n,i}}\right) \right] \right\} \le \EE\left[ K_i^2\left(\dfrac{X^{y,j}_i-t}{h_{n,i}}\right) \right] \nonumber \\
& = \int K_i^2\left(\dfrac{z-t}{h_{n,i}}\right) f_{i|y}(z) dz = \int_{t-h_{n,i}}^{t+h_{n,i}} K_i^2\left(\dfrac{z-t}{h_{n,i}}\right) f_{i|y}(z) dz \nonumber \\
& \le \int_{t-h_{n,i}}^{t+h_{n,i}} K_i^2\left(\dfrac{z-t}{h_{n,i}}\right) \left[\sup_{z'\in[t-h_{n,i},t+h_{n,i}]} f_{i|y}(z')\right] dz \le 2 \int K_i^2\left(\dfrac{z-t}{h_{n,i}}\right) f_{i|y}(t) dz \nonumber \\
& = 2 f_{i|y}(t) \int K_i^2\left(\dfrac{z-t}{h_{n,i}}\right) dz = 2 \|K_i\|_{L^2}^2 f_{i|y}(t) h_{n,i}. \nonumber
\end{align}
Here the second equality follows from the fact that $K_i$ is supported on $[-1,1]$, and in the third inequality we have invoked (\ref{eq:t_in_B_f}).  Hence, we conclude that
%and the last equality follows from the definition of $C_{K_i^2}$ as in (\ref{eq:def_C_K_i_2}).
\begin{align}
\VV\left\{ \dfrac{1}{h_{n,i}} K_i\left(\dfrac{X^{y,j}_i-t}{h_{n,i}}\right) - \EE \left[\dfrac{1}{h_{n,i}} K_i\left(\dfrac{X^{y,j}_i-t}{h_{n,i}}\right) \right] \right\} \le 2 \|K_i\|_{L^2}^2 \dfrac{f_{i|y}(t)}{h_{n,i}}.
\label{eq:Bernstein_variance_term}
\end{align}
%Also, by the definition of $K_{\max,i}$ in (\ref{eq:def_K_max_i}), we have
We also have
\begin{align}
\left| \dfrac{1}{h_{n,i}} K_i\left(\dfrac{X^{y,j}_i-t}{h_{n,i}}\right) - \EE \dfrac{1}{h_{n,i}} K_i\left(\dfrac{X^{y,j}_i-t}{h_{n,i}}\right) \right| \le \dfrac{2 \|K_i\|_{L^{\infty}} }{h_{n,i}}.
\label{eq:Bernstein_bounded_term}
\end{align}
Then, by Bernstein's inequality,
\begin{align}
&\PP\left\{ \dfrac{ |\wh f_{i|y}(t) - \EE \wh f_{i|y}(t)| }{ f_{i|y}(t) } \ge \epsilon' \right\} = \PP\left\{ |\wh f_{i|y}(t) - \EE \wh f_{i|y}(t)| \ge \epsilon' f_{i|y}(t) \right\} \nonumber \\
&\le 2\exp\left(- \dfrac{ n^2 \epsilon'^2 f^2_{i|y}(t) }{ 2 n \VV \left\{ \dfrac{1}{h_{n,i}} K_i\left(\dfrac{X^{y,j}_i-t}{h_{n,i}}\right) - \EE \left[ \dfrac{1}{h_{n,i}} K_i\left(\dfrac{X^{y,j}_i-t}{h_{n,i}}\right) \right] \right\} + \dfrac{4}{3} \|K_i\|_{L^{\infty}} \dfrac{ n \epsilon' f_{i|y}(t) }{h_{n,i}} }\right) \nonumber \\
&\le 2\exp\left(- \dfrac{ n \epsilon'^2 f^2_{i|y}(t) }{ 4 \|K_i\|_{L^2}^2 \dfrac{f_{i|y}(t)}{h_{n,i}}  + \dfrac{4}{3} \|K_i\|_{L^{\infty}} \epsilon' \dfrac{ f_{i|y}(t) }{h_{n,i}}  }\right) \nonumber \\
&\le 2\exp\left(- \dfrac{ 3 }{8\max\left\{ \|K_i\|_{L^{\infty}} \epsilon', 3\|K_i\|_{L^2}^2 \right\}} n \epsilon'^2 f_{i|y}(t) h_{n,i} \right), \nonumber
%\label{eq:f_rel_dev_variance}
\end{align}
which is the conclusion of the proposition.
\qed

%----------------------------------------------------------------------------------------------------

\subsection{Proof of Theorem~\ref{thm:f_rel_dev}}
\label{sec:proof_thm:f_rel_dev}

We first prove case for the H\"{o}lder class.  We use the decomposition
\begin{align}
\wh f_{i|y}(t) - f_{i|y}(t) = \left[ \wh f_{i|y}(t) - \EE \wh f_{i|y}(t) \right] + \left[ \EE \wh f_{i|y}(t) - f_{i|y}(t) \right]. \nonumber
\end{align}
By standard derivation (e.g., \cite[Proposition~1.2]{TsybakovBook09}), for the bias part, we have
\begin{align}
| \EE \wh f_{i|y}(t) - f_{i|y}(t) | &\le \dfrac{L_i}{l!} h_{n,i}^{\beta_i} \int_{-1}^{1} |K_i(u)| |u^{\beta_i}| du \le 2 \dfrac{L_i}{l!} \|K_i\|_{L^{\infty}} h_{n,i}^{\beta_i} = \dfrac{1}{2^+ C_i^{\beta_i}} h_{n,i}^{\beta_i}.
\label{eq:f_rel_dev_bias}
\end{align}

Next, because $t$ satisfies (\ref{eq:t_in_B_f}), by Proposition~\ref{prop:f_rel_dev_variance}, Inequality~(\ref{eq:f_rel_dev_variance}) holds.  Combining (\ref{eq:f_rel_dev_variance}) and (\ref{eq:f_rel_dev_bias}), we have
\begin{align}
&\PP\left\{ \dfrac{ | \wh f_{i|y}(t) - f_{i|y}(t) | }{ f_{i|y}(t) } \ge \epsilon_n \right\} \nonumber \\
&\le \PP\left\{ \dfrac{|\wh f_{i|y}(t) - \EE \wh f_{i|y}(t)|}{ f_{i|y}(t) } \ge \dfrac{ \epsilon_n }{2} \right\} + \1\left\{ \dfrac{|\EE \wh f_{i|y}(t) - f_{i|y}(t)|}{f_{i|y}(t)} \ge \dfrac{ \epsilon_n }{2} \right\} \nonumber \\
&\le 2 \exp\left(- \dfrac{3}{32\max\left\{ \|K_i\|_{L^{\infty}} \epsilon_n, 3\|K_i\|_{L^2}^2 \right\}} n \epsilon_n^2 f_{i|y}(t) h_{n,i} \right) + \1\left\{ \dfrac{1}{ 2^+ C_i^{\beta_i}} h_{n,i}^{\beta_i} \ge \dfrac{1}{2}\epsilon_n f_{i|y}(t) \right\} \nonumber \\
&\le 2 \exp\left( - \dfrac{3}{32\max\left\{ \|K_i\|_{L^{\infty}} \epsilon_n, 3\|K_i\|_{L^2}^2 \right\}} n \epsilon_n^2 \underline{f_{n,i}} h_{n,i} \right) + \1\left\{ \left(\dfrac{2}{2^+}\right)^{\frac{1}{\beta_i}} h_{n,i} \ge C_i \left(\epsilon_n \underline{f_{n,i}} \right)^{\frac{1}{\beta_i}} \right\} \nonumber \\
&= 2 \exp\left( - \dfrac{3}{32\max\left\{ \|K_i\|_{L^{\infty}} \epsilon_n, 3\|K_i\|_{L^2}^2 \right\}} n \epsilon_n^2 \underline{f_{n,i}} h_{n,i} \right).
\label{eq:f_rel_dev_concrete_1}
\end{align}
Here the last equality follows from (\ref{eq:h_n_i_concrete}).  Then, Inequality~(\ref{eq:f_rel_dev_concrete}) follows from Inequality~(\ref{eq:f_rel_dev_concrete_1}) by the choices (\ref{eq:choice_epsilon}) of $\epsilon_n$, (\ref{eq:f_i_lower_bound}) of $\underline{f_{n,i}}$ with a large enough constant $J_{\beta_i,\gamma,C_d}$, and (\ref{eq:h_n_i_concrete}) of $h_{n,i}$.

Next we prove the case for the super-smooth densities.  By Definition~\ref{def:super_smooth_densities} and the choice (\ref{eq:h_n_i_concrete_super_smooth}) of the bandwidth $h_{n,i}$, we have
\begin{align}
\left| \EE \wh f_{i|y}(t) - f_{i|y}(t) \right| \le c_{\lceil\log(n)\rceil} h_{n,i}^{\lceil\log(n)\rceil} = c_{\lceil\log(n)\rceil} H_i^{\lceil\log(n)\rceil} \log^{-\frac{1}{2}\lceil\log(n)\rceil}(n).
\label{eq:f_rel_dev_bias_super_smooth_1}
\end{align}
Then, from (\ref{eq:f_rel_dev_bias_super_smooth_1}), the assumption $c_{\lceil\log(n)\rceil}\rightarrow 0$ as $n\rightarrow\infty$, and the fact that $\log^{-\frac{1}{2}\lceil\log(n)\rceil}(n)=o(n^{-\epsilon})$ for all $\epsilon>0$, we have
\begin{align}
\1\left\{ \dfrac{|\EE \wh f_{i|y}(t) - f_{i|y}(t)|}{f_{i|y}(t)} \ge \dfrac{ \epsilon_n }{2} \right\} &\le \1\left\{ c_{\lceil\log(n)\rceil} H_i^{\lceil\log(n)\rceil} \log^{-\frac{1}{2}\lceil\log(n)\rceil}(n) \ge \dfrac{1}{2} \epsilon_n \underline{f_{n,i}} \right\} = 0
\label{eq:f_rel_dev_bias_super_smooth_2}
\end{align}
for all $n$ large enough.  Then, replacing the second term in the second line of (\ref{eq:f_rel_dev_concrete_1}) by (\ref{eq:f_rel_dev_bias_super_smooth_2}), and upper bounding the term $\max\left\{ \|K_i\|_{L^{\infty}} \epsilon_n, 3\|K_i\|_{L^2}^2 \right\}$ in the third line of (\ref{eq:f_rel_dev_concrete_1}) by $3\|K^{\lceil\log(n)\rceil}\|_{L^2}^2\le3\lceil\log(n)\rceil$ for $n$ large enough, again yield Inequality~(\ref{eq:f_rel_dev_concrete}).
\qed

%----------------------------------------------------------------------------------------------------

\subsection{Proof of Theorem~\ref{thm:S_f_x_c_est}}
\label{sec:proof_thm:S_f_x_c_est}
We fix an arbitrary $x\in\mathbb{R}^d$ satisfying Assumption~\ref{ass_density_x_S_f_x_c}, and an arbitrary $i\notin S^f_x$.  We let
\begin{align}
t=x_i. \nonumber
\end{align}

By the construction of our test, we have
\begin{align}
\PP(\wt\Delta\log f_i(t)=0) \ge \min\left\{ \PP\left(\text{test}~(\ref{eq:density_test_1})~\text{succeeds}\right), \PP\left(\text{test}~(\ref{eq:density_test_2})~\text{succeeds}\right) \right\}.
\label{eq:P_wt_Delta_log_f_equal_zero}
\end{align}

First, suppose that $t$ satisfies
\begin{align}
f_i(t) = f_{i|0}(t) = f_{i|1}(t) <\underline{f_{n,i}}.
\label{eq:density_t_small}
\end{align}
In this case, we focus on test (\ref{eq:density_test_1}).  Note that we would like Inequality~(\ref{eq:density_test_1}) to hold that so that we set $\wt\Delta\log f_i(x_i) = 0$.

For the case of the H\"{o}lder class, as in (\ref{eq:f_rel_dev_bias}), the bias term $\EE \wh f_{i}(t) - f_{i}(t)$ satisfies
\begin{align}
| \EE \wh f_{i}(t) - f_{i}(t) | &\le \sum_{y\in\{0,1\}} \frac{1}{2}| \EE \wh f_{i|y}(t) - f_{i|y}(t) | \le \dfrac{1}{2^+ C_i^{\beta_i}} h_{n,i}^{\beta_i} = \dfrac{1}{2^+} \epsilon_n\underline{f_{n,i}} = o(\underline{f_{n,i}}).
\label{eq:sparsity_density_equal_bias}
\end{align}
Here the first equality follows from our choice (\ref{eq:h_n_i_concrete}) of $h_{n,i}$, and the second equality follows by (\ref{eq:choice_epsilon}).  For the case of super-smooth densities, the conclusion of Inequality~(\ref{eq:sparsity_density_equal_bias}) holds as well by a derivation similar to that of (\ref{eq:f_rel_dev_bias_super_smooth_1}) and (\ref{eq:f_rel_dev_bias_super_smooth_2}).

%(\ref{eq:f_i_lower_bound}) of $\underline{f_{n,i}}$.
%Note that, although now $f_i(t)\le\underline{f_{n,i}}$ so condition imposed by (\ref{def_B_f_c}) does not apply, one moment's though reveal that by the continuity of $f_i$ and (\ref{def_B_f_c}) together imply that $\sup_{z'\in[t-h_{n,i},t+h_{n,i}]} f_{i|y}(z')\le 2\underline{f_{n,i}}$.

Next we discuss the variance part.  By Assumption~\ref{ass_density_x_S_f_x_c} and the restriction on $t$ by (\ref{eq:density_t_small}), we have $\sup_{z'\in[t-h_{n,i},t+h_{n,i}]} f_{i|y}(z')\le 2\underline{f_{n,i}}$.  Then, by the derivation of (\ref{eq:Bernstein_variance_term}), we have
%with the factor $\sup_{z'\in[t-h_{n,i},t+h_{n,i}]} f_{i|y}(z')$ upper bounded by $\bar{f}$
\begin{align}
\VV\left[ \dfrac{1}{h_{n,i}} K_i\left(\dfrac{X^{y,j}_i-t}{h_{n,i}}\right) - \EE \dfrac{1}{h_{n,i}} K_i\left(\dfrac{X^{y,j}_i-t}{h_{n,i}}\right) \right] \le 2 \|K_i\|_{L^2}^2 \dfrac{\underline{f_{n,i}}}{h_{n,i}}, \nonumber
\end{align}
and we also recall (\ref{eq:Bernstein_bounded_term}).  Then, by Bernstein's inequality,
\begin{align}
&\PP\left\{ \wh f_{i}(t) - \EE \wh f_{i}(t) > \underline{f_{n,i}} \right\} \le \sum_{y\in\{0,1\}} \PP\left\{ \wh f_{i|y}(t) - \EE \wh f_{i|y}(t) \ge \underline{f_{n,i}} \right\} \nonumber \\
&\le 2 \exp\left(-\dfrac{ C n^2 \underline{f_{n,i}^2} }{ n \VV \left\{ \dfrac{1}{h_{n,i}} K_i\left(\dfrac{X^{y,j}_i-t}{h_{n,i}}\right) - \EE \left[\dfrac{1}{h_{n,i}} K_i\left(\dfrac{X^{y,j}_i-t}{h_{n,i}}\right) \right] \right\}  + \dfrac{\|K_i\|_{L^{\infty}} }{h_{n,i}} n \underline{f_{n,i}} }\right) \nonumber \\
%&\le \exp\left(-\dfrac{C}{ \max\{ \|K_i\|_{L^2}^2, \|K_i\|_{L^{\infty}}\} } \dfrac{ n \underline{f_{n,i}^2} }{ \dfrac{ \underline{f_{n,i}} }{h_{n,i}} + \dfrac{\epsilon}{h_{n,i}} } \right) \nonumber \\
&\le 2 \exp\left(-\dfrac{C }{ \max\{ \|K_i\|_{L^2}^2, \|K_i\|_{L^{\infty}}\} } n \underline{f_{n,i}} h_{n,i} \right) \le \exp\left( - c_i' n^{c_i} \right).
\label{eq:sparsity_density_equal_variance}
\end{align}
for some constants $c_i,c_i'>0$ dependent only on the parameters $\beta_i, \gamma, C_d, L_i, \|K_i\|_{L^2}, \|K_i\|_{L^{\infty}}$.

Hence, we conclude that, for $n$ large enough,
\begin{align}
\PP(\wt\Delta\log f_i(t)=0) &\ge \PP\left( \wh f_{i}(t) \le 3 \underline{f_{n,i}} \right) \ge \PP\left( \wh f_{i}(t) - f_{i}(t) \le 2 \underline{f_{n,i}} \right) \nonumber \\
&\ge \PP\left( \left\{ \wh f_{i}(t) - \EE \wh f_{i}(t) \le \underline{f_{n,i}} \right\} \cap \left\{ | \EE \wh f_{i}(t) - f_{i}(t) | \le \underline{f_{n,i}} \right\} \right) \nonumber \\
&\ge \PP\left( \wh f_{i}(t) - \EE \wh f_{i}(t) \le \underline{f_{n,i}} \right) - \1\left\{ | \EE \wh f_{i}(t) - f_{i}(t) | > \underline{f_{n,i}} \right\} \nonumber \\
&\ge 1 - \exp\left(-c_i' n^{c_i} \right). \nonumber
\end{align}
Here, the first Inequality follows from (\ref{eq:P_wt_Delta_log_f_equal_zero}) and (\ref{eq:density_test_1}), the second inequality follows from (\ref{eq:density_t_small}), and the last inequality follows from (\ref{eq:sparsity_density_equal_bias}) and its counterpart for super-smooth densities, and (\ref{eq:sparsity_density_equal_variance}).  Therefore we conclude that (\ref{eq:G_i_prob_S_f_x_c_individual}) holds for $t=x_i$ specified by the regime~(\ref{eq:density_t_small}).

Next suppose that, in contrast to (\ref{eq:density_t_small}), $t$ is such that
\begin{align}
f_i(t) = f_{i|0}(t) = f_{i|1}(t) \ge \underline{f_{n,i}}.
\label{eq:density_t_large}
\end{align}
In this case, test~(\ref{eq:density_test_1}) is more likely to fail, so we switch to study test~(\ref{eq:density_test_2}).  Note that we set $\wt\Delta\log f_i(t)=0$ (the desirable case) if Inequality~(\ref{eq:density_test_2}) holds.  For brevity we let, for $y\in\{0,1\}$,
$$\delta_y = \wh f_{i|y}(t) - f_{i|y}(t).$$
Using the mean value theorem, we have
\begin{align}
\left| \log \wh f_{i|0}(t) - \log \wh f_{i|1}(t) - \left(\log f_{i|0}(t) - \log f_{i|1}(t) \right) \right| \le \left| \dfrac{\delta_0}{\wt f_{i|0}(t)} \right| + \left| \dfrac{\delta_1}{\wt f_{i|1}(t)} \right|.
\label{eq:log_density_ineq_pre}
\end{align}
Here $\wt f_{i|y}(t)$ is some number sandwiched between $f_{i|y}(t)$ and $\wh f_{i|y}(t)$.  We define the event
\begin{align}
L_{n,i,\gamma,t} = \cap_{y\in\{0,1\}}\left\{ \dfrac{ | \wh f_{i|y}(t) - f_{i|y}(t) | }{ f_{i|y}(t) } < \epsilon_n \right\}.
\label{eq:L_n_i_gamma_t}
\end{align}
Then, we further deduce from (\ref{eq:log_density_ineq_pre}) that, on the event $L_{n,i,\gamma,t}$,
\begin{align}
&\left| \log \wh f_{i|0}(t) - \log \wh f_{i|1}(t) - \left(\log f_{i|0}(t) - \log f_{i|1}(t) \right) \right| \le \dfrac{ |\delta_0| }{ f_{i|0}(t)-|\delta_0| } + \dfrac{ |\delta_1| }{ f_{i|1}(t)-|\delta_1|} \nonumber \\
&= \dfrac{ \dfrac{|\delta_0|}{f_{i|0}(t)} }{ 1-\dfrac{|\delta_0|}{f_{i|0}(t)} } + \dfrac{ \dfrac{|\delta_1|}{f_{i|1}(t)} }{ 1-\dfrac{|\delta_1|}{f_{i|1}(t)} } < 2 \dfrac{\epsilon_n}{1-\epsilon_n} = \tilde\delta_{n,\gamma}
\label{eq:log_f_diff_wt_delta}
\end{align}
(we recall $\tilde\delta_{n,\gamma}$ as defined in (\ref{eq:def_density_delta})).  Therefore, from (\ref{eq:P_wt_Delta_log_f_equal_zero}), (\ref{eq:density_test_2}) and (\ref{eq:log_f_diff_wt_delta}), we have
\begin{align}
&\PP(\wt\Delta\log f_i(t)=0) \ge \PP\left( \left| \log \wh f_{i|0}(t) - \log \wh f_{i|1}(t) \right| < \tilde\delta_{n,\gamma} \right) \nonumber \\
&= \PP\left( \left| \log \wh f_{i|0}(t) - \log \wh f_{i|1}(t) - \left(\log f_{i|0}(t) - \log f_{i|1}(t) \right) \right| < \tilde\delta_{n,\gamma} \right) \ge \PP(L_{n,i,\gamma,t}).
\label{eq:density_test_2_prob_1}
\end{align}
Here the equality follows because in the current case $f_{i|0}(t)= f_{i|1}(t)$.

Because in the current case specified by (\ref{eq:density_t_large}), condition~(\ref{eq:t_in_B_f}) holds for $y\in\{0,1\}$, we can apply Inequality~(\ref{eq:f_rel_dev_concrete}) in Theorem~\ref{thm:f_rel_dev} to conclude that
\begin{align}
\PP(L_{n,i,\gamma,t}) \ge 1 - \dfrac{2}{d} n^{-\gamma/2}.
\label{eq:P_L_n_i_gamma_t}
\end{align}
Therefore, from (\ref{eq:density_test_2_prob_1}) and (\ref{eq:P_L_n_i_gamma_t}), we conclude that (\ref{eq:G_i_prob_S_f_x_c_individual}) holds for $t=x_i$ specified by the regime~(\ref{eq:density_t_large}).  Combining with our earlier display, we conclude that (\ref{eq:G_i_prob_S_f_x_c_individual}) holds.
%for all $x\in\mathbb{R}^d$ and for all $i$ such that $f_{i|0}(x_i)=f_{i|1}(x_i)$, and $x_i$ satisfies condition (\ref{eq:t_in_B_f}) with $t$ replaced by $x_i$ and with $B^f_{h_{n,i},i,y}$ replaced by $B^f_{h_{n,i},i,0}$ or $B^f_{h_{n,i},i,1}$.

Finally, as stated in the theorem, (\ref{eq:G_i_prob_S_f_x_c_union}) follows from (\ref{eq:G_i_prob_S_f_x_c_individual}) by a union bound argument.
\qed

%----------------------------------------------------------------------------------------------------

\subsection{Proof of Theorem~\ref{thm:est_naive_bayes_part}}
\label{sec:proof_thm:est_naive_bayes_part}

We fix an arbitrary $x\in\mathbb{R}^d$ satisfying Assumptions~\ref{ass_density_x_S_f_x_c} and \ref{ass_density_x_S_f_x}, and an arbitrary $i\in S^f_x$.  We let
$$t = x_i.$$

We have
\begin{align}
&\left\{ \left| \wt\Delta\log f_i(t) - \Delta\log f_i(t) \right| < 2\tilde\delta_{n,\gamma} \right\} \nonumber \\
&\supset\left( \left\{ \left| \wt\Delta\log f_i(t) - \Delta\log f_i(t) \right| < \tilde\delta_{n,\gamma} \right\} \cap \left\{ \wt\Delta\log f_i(t)\neq 0 \right\} \right) \nonumber \\
&\cup \left( \left\{ \left| \wt\Delta\log f_i(t) - \Delta\log f_i(t) \right| < 2\tilde\delta_{n,\gamma} \right\} \cap \left\{ \wt\Delta\log f_i(t)=0 \right\} \right) \nonumber \\
&=\left( \left\{ \left| \log \wh f_{i|0}(t) - \log \wh f_{i|1}(t) - \left(\log f_{i|0}(t) - \log f_{i|1}(t) \right) \right| < \tilde\delta_{n,\gamma} \right\} \cap \left\{ \wt\Delta\log f_i(t)\neq 0 \right\} \right) \nonumber \\
&\cup\left( \left\{ \left| \Delta\log f_i(t) \right| < 2\tilde\delta_{n,\gamma} \right\} \cap \left\{ \wt\Delta\log f_i(t)=0 \right\} \right).
\label{eq:Delta_log_f_decomposition}
\end{align}
We discuss separately the cases
\begin{align}
\left| \Delta\log f_i(t) \right| < 2\tilde\delta_{n,\gamma}
\label{eq:Delta_log_f_small}
\end{align}
and
\begin{align}
\left| \Delta\log f_i(t) \right| \ge 2\tilde\delta_{n,\gamma}.
\label{eq:Delta_log_f_large}
\end{align}
First, we suppose that (\ref{eq:Delta_log_f_small}) holds.  Then, from (\ref{eq:Delta_log_f_decomposition}), we have
\begin{align}
&\left\{ \left| \wt\Delta\log f_i(t) - \Delta\log f_i(t) \right| < 2\tilde\delta_{n,\gamma} \right\} \nonumber \\
&\supset\left( \left\{ \left| \log \wh f_{i|0}(t) - \log \wh f_{i|1}(t) - \left(\log f_{i|0}(t) - \log f_{i|1}(t) \right) \right| < \tilde\delta_{n,\gamma} \right\} \cap \left\{ \wt\Delta\log f_i(t)\neq 0 \right\} \right) \nonumber \\
&\cup \left\{ \wt\Delta\log f_i(t)=0 \right\} \nonumber \\
&\supset \left( \left\{ \left| \log \wh f_{i|0}(t) - \log \wh f_{i|1}(t) - \left(\log f_{i|0}(t) - \log f_{i|1}(t) \right) \right| < \tilde\delta_{n,\gamma} \right\} \cap \left\{ \wt\Delta\log f_i(t)\neq 0 \right\} \right) \nonumber \\
&\cup \left( \left\{ \left| \log \wh f_{i|0}(t) - \log \wh f_{i|1}(t) - \left(\log f_{i|0}(t) - \log f_{i|1}(t) \right) \right| < \tilde\delta_{n,\gamma} \right\} \cap \left\{ \wt\Delta\log f_i(t)= 0 \right\} \right) \nonumber \\
&= \left\{ \left| \log \wh f_{i|0}(t) - \log \wh f_{i|1}(t) - \left(\log f_{i|0}(t) - \log f_{i|1}(t) \right) \right| < \tilde\delta_{n,\gamma} \right\} \supset L_{n,i,\gamma,t}.
\label{eq:Delta_log_f_event_small}
\end{align}
Here the last step follows because (\ref{eq:log_f_diff_wt_delta}) holds on the event $L_{n,i,\gamma,t}$ introduced in (\ref{eq:L_n_i_gamma_t}).

Next, suppose that (\ref{eq:Delta_log_f_large}) holds instead of (\ref{eq:Delta_log_f_small}).  Then, from (\ref{eq:Delta_log_f_decomposition}), we have
\begin{align}
&\left\{ \left| \wt\Delta\log f_i(t) - \Delta\log f_i(t) \right| < 2 \tilde\delta_{n,\gamma} \right\} \nonumber \\
&\supset \left\{ \left| \log \wh f_{i|0}(t) - \log \wh f_{i|1}(t) - \left(\log f_{i|0}(t) - \log f_{i|1}(t) \right) \right| < \tilde\delta_{n,\gamma} \right\} \cap \left\{ \wt\Delta\log f_i(t)\neq 0 \right\} \nonumber \\
&\supset \left\{ \left| \log \wh f_{i|0}(t) - \log \wh f_{i|1}(t) - \left(\log f_{i|0}(t) - \log f_{i|1}(t) \right) \right| < \tilde\delta_{n,\gamma} \right\} \cap \left\{ \wh f_i(t) > 3 \underline{f_{n,i}} \right\} \nonumber \\
&\supset L_{n,i,\gamma,t} \cap \left\{ \wh f_i(t) > 3 \underline{f_{n,i}} \right\} = L_{n,i,\gamma,t} \cap L_{n,i,\gamma,t}'.
\label{eq:Delta_log_f_event_large}
\end{align}
Here the second step follows because, by tests (\ref{eq:density_test_1}) and (\ref{eq:density_test_2}), the dual conditions $\wh f_i(t) > 3 \underline{f_{n,i}}$, and $\left| \log \wh f_{i|0}(t) - \log \wh f_{i|1}(t) - \left(\log f_{i|0}(t) - \log f_{i|1}(t) \right) \right| < \tilde\delta_{n,\gamma}$ when (\ref{eq:Delta_log_f_large}) holds, together implies $\wt\Delta\log f_i(t)\neq 0$, and in the last step we have introduced the event
\begin{align}
L_{n,i,\gamma,t}' = \left\{\wh f_i(t) > 3 \underline{f_{n,i}} \right\}. \nonumber
%\label{eq:L_p_n_i_gamma_t}
\end{align}
We have
\begin{align}
&L_{n,i,\gamma,t}'^c \subset \cup_{ y\in\{0,1\} } \left\{ \wh f_{i|y}(t) \le  3 \underline{f_{n,i}} \right\} \subset \cup_{ y\in\{0,1\} } \left\{ \wh f_{i|y}(t) \le  ( 1-\epsilon_n ) f_{i|y}(t) \right\} \nonumber \\
&= \cup_{ y\in\{0,1\} } \left\{ \wh f_{i|y}(t) - f_{i|y}(t) \le  -\epsilon_n f_{i|y}(t) \right\} \subset \cup_{y\in\{0,1\}} \left\{ | \wh f_{i|y}(t) - f_{i|y}(t) | \ge  \epsilon_n f_{i|y}(t) \right\}. \nonumber
\end{align}
Here the second inequality follows from (\ref{eq:density_t_large_2}).  The above implies that
\begin{align}
&L_{n,i,\gamma,t}' \supset \cap_{ y\in\{0,1\} } \left\{ | \wh f_{i|y}(t) - f_{i|y}(t) | <  \epsilon_n f_{i|y}(t) \right\} = L_{n,i,\gamma,t}, \nonumber
\end{align}
which, together with (\ref{eq:Delta_log_f_event_large}), further implies that, for the case (\ref{eq:Delta_log_f_large}),
\begin{align}
\left\{ \left| \wt\Delta\log f_i(t) - \Delta\log f_i(t) \right| < 2 \tilde\delta_{n,\gamma} \right\} \supset L_{n,i,\gamma,t}.
\label{eq:Delta_log_f_event_large_2}
\end{align}

Combining (\ref{eq:Delta_log_f_event_small}) and (\ref{eq:Delta_log_f_event_large_2}) for the cases (\ref{eq:Delta_log_f_small}) and (\ref{eq:Delta_log_f_large}) respectively, and taking the intersection over $i\in S_x^f$, we conclude that
\begin{align}
\cap_{i\in S_x^f} \left\{ \left| \wt\Delta\log f_i(x_i) - \Delta\log f_i(x_i) \right| \le 2 \tilde\delta_{n,\gamma} \right\} \supset \cap_{i\in S_x^f} L_{n,i,\gamma,x_i} \nonumber
\end{align}
which further implies
\begin{align}
L^{\text{bayes}}_{x,n} \supset \left\{\wh S^f_x \subset S^f_x \right\} \cap \left( \cap_{i\in S_x^f} L_{n,i,\gamma,x_i} \right).
\label{eq:G_x_p_intermediate}
\end{align}
Then, Inequality~(\ref{eq:L_n_b_x_n_gamma_prob}) follows from set relationship~(\ref{eq:G_x_p_intermediate}), Inequality~(\ref{eq:G_i_prob_S_f_x_c_union}) in Theorem~\ref{thm:S_f_x_c_est}, and Inequality~(\ref{eq:P_L_n_i_gamma_t}) for $t=x_i$ with $i\in S_x^f$ (which holds because (\ref{eq:t_in_B_f}) holds for $t=x_i$ with $i\in S_x^f$ and for $y\in\{0,1\}$ by Assumption~\ref{ass_density_x_S_f_x}).

%in Theorem~\ref{thm:f_rel_dev}, which holds because condition (\ref{eq:t_in_B_f}) holds (by Assumption~\ref{ass_density_x_S_f_x}, in particular (\ref{eq:density_t_large_2}), on $t=x_i$).

Finally we prove (\ref{eq:error_bound_bayes}) on the event $L^{\text{bayes}}_{x,n}$.  On this event we have
\begin{align}
\left\| \left[\wt \Delta\log f - \Delta\log f\right](x) \right\|_{\ell_1} &= \sum_{i\in S^f_x} \left|\wt \Delta\log f_i(x_i) - \Delta\log f_i(x_i) \right| \le 2 s^f_x \tilde\delta_{n,\gamma}. \nonumber
\end{align}
Here the equality follows because $\wh S^f_x \subset S^f_x$, and the inequality follows because for all $i\in S^f_x$ we have $\left|\wt \Delta\log f_i(x_i) - \Delta\log f_i(x_i) \right| \le 2 \tilde\delta_{n,\gamma}$, all by the definition of $L^{\text{bayes}}_{x,n}$ as in (\ref{eq:L_n_b_x_n_gamma}).
\qed

%--------------------------------------------------
%--------------------------------------------------
%--------------------------------------------------
%--------------------------------------------------
%--------------------------------------------------
%--------------------------------------------------
%--------------------------------------------------
%--------------------------------------------------
%--------------------------------------------------
%--------------------------------------------------

\section*{Acknowledgements}

This research is supported in part by NSF Grant DMS~1310119.

\bigskip

\appendix

\section{Auxiliary proofs}

\subsection{Proof of Proposition~\ref{prop:basic_facts_Phi_Phi_inv}}
\label{sec:proof_prop:basic_facts_Phi_Phi_inv}

(\ref{eq:Phi_phi}) is well known, see for instance Inequality~(9) in \cite{Gordon41}.  Result analogous to (\ref{eq:Phi_inv}) in terms of the closely related complementary error function is well known too; here for completeness we give the derivation of (\ref{eq:Phi_inv}).  Starting from Equation~(2) and Inequality~(5) in \cite{Chi03}, we have
\begin{align}
1-\Phi(x)\le\dfrac{1}{2}e^{-x^2/2}, \nonumber
\end{align}
which further implies that
\begin{align}
\log(2(1-\Phi(x)))\le-x^2/2 \Rightarrow x\le\sqrt{2\log\dfrac{1}{2(1-\Phi(x))}} \Rightarrow \Phi^{-1}(x)\le\sqrt{2\log\dfrac{1}{2(1-x)}}. \nonumber
\end{align}
\qed

\subsection{Proof of Proposition~\ref{prop:A_n_prob_bound}}
\label{sec:proof_prop:A_n_prob_bound}
We let $t$ be such that $\alpha_{i|y}(t)=a_n$.  We have
\begin{align}
g(n,\gamma) = \dfrac{1}{2 a_n}\phi(a_n)\le \dfrac{a_n}{1+a_n^2}\phi(a_n) & \le 1 - \Phi(a_n) \le \dfrac{1}{a_n}\phi(a_n) = 2 g(n,\gamma)
\label{eq:prop:A_n_prob_bound_1}
\end{align}
Here the first inequality follows because $a_n\ge 1$ by assumption, the second and third inequalities follow by (\ref{eq:Phi_phi}).  Then, substituting $\Phi(a_n) = \Phi(\alpha_{i|y}(t)) = F_{i|y}(t)$ into (\ref{eq:prop:A_n_prob_bound_1}) yields (\ref{eq:F_edge_bound}).
%\begin{align}
%\PP(\alpha_{i|y}(X_i)\in [-a_n,a_n]|Y=y) &= \PP\left(\Phi\left(\alpha_{i|y}(X_i)\right)\in \Phi([-a_n,a_n])|Y=y\right) \nonumber \\
%&= \PP\left(F_{i|y}(X_i)\in \Phi([-a_n,a_n])|Y=y\right) \nonumber \\
%&\ge \PP\left(F_{i|y}(X_i)\in \left[\dfrac{\phi(a_n)}{a_n},1-\dfrac{\phi(a_n)}{a_n}\right]|Y=y\right) \nonumber \\
%&= \PP\left(F_{i|y}(X_i)\in \left[2g(n,\gamma),1-2g(n,\gamma)\right]|Y=y\right) \nonumber \\
%&= 1-4g(n,\gamma). \nonumber
%\end{align}
%Here in the second equality we have used $\Phi(\alpha_{i|y}(t)) = \Phi( \Phi^{-1}(F_{i|y}(t)) ) = F_{i|y}(t)$, in the inequality we have used the second half of Inequality~(\ref{eq:Phi_phi}) with $t$ replaced by $a_n$, in the third equality we have used the definition of $g(n,\gamma)$ as in (\ref{eq:g_n_gamma}), and in the last equality we have used the fact that, given $Y=y$, $F_{i|y}(X_i)$ is uniformly distributed on $[0,1]$.

By symmetry and (\ref{eq:F_edge_bound}), we have that, for $t$ be such that $\alpha_{i|y}(t)=-a_n$,
\begin{align}
g(n,\gamma)\le F_{i|y}(t) \le 2g(n,\gamma).
\label{eq:prop:A_n_prob_bound_2}
\end{align}
Then, (\ref{eq:F_bound}) follows from the first halfs of (\ref{eq:F_edge_bound}) and (\ref{eq:prop:A_n_prob_bound_2}), and the monotonicity of $\alpha_{i|y}$ and $F_{i|y}$.
%Next, note that for all $t$ such that $\alpha_{i|y}(t)\in [-a_n,a_n]$, we have $-a_n\le \alpha_{i|y}(t)\le a_n$ and so by the monotonicity of $\Phi$, we have
%\begin{align}
%\Phi(-a_n) \le\Phi(\alpha_{i|y}(t)) = F_{i|y}(t) = \Phi(\alpha_{i|y}(t)) \le \Phi(a_n). \nonumber
%\end{align}
%Then, (\ref{eq:F_bound}) follows from the first half of Inequality~(\ref{eq:Phi_phi}).
\qed

%--------------------------------------------------

\subsection{The margin assumption for Gaussian classification}
\label{sec:margin_assumption_Gaussian}
In this section we consider the margin assumption for classifying two Gaussian distributions with the same covariance matrix $\Sigma\in\mathbb{R}^{d\times d}$.  Without loss of generality we assume that $(X|Y=0)\sim N(0,\Sigma)$, $(X|Y=1)\sim N(\mu,\Sigma)$ for some $\mu\in\mathbb{R}^d$.  It is straightforward to derive that, for $x\in\mathbb{R}^d$,
\begin{align}
|\log(f^0/f^1)(x)| =|\mu^T \Sigma^{-1} x - \mu^T \Sigma^{-1}\mu/2|.
\label{eq:margin_assumption_Gaussian_1}
\end{align}
Note that $( \mu^T \Sigma^{-1} X |Y=0)\sim N(0, \mu^T \Sigma^{-1} \mu)$.  Hence, we have
\begin{align}
\PP(0<|\log(f^0/f^1)(X)|\le t) &= \PP(0<|\log(f^0/f^1)(X)|\le t |Y=0) \nonumber \\
&= \PP(0<|\mu^T \Sigma^{-1} X - \mu^T \Sigma^{-1}\mu/2)|\le t |Y=0) \nonumber \\
&\le \dfrac{2}{\sqrt{2\pi \mu^T \Sigma^{-1} \mu} } t. \nonumber
\end{align}
Here the first equality follows by symmetry, the second equality follows by (\ref{eq:margin_assumption_Gaussian_1}), and the inequality follows because the density of the $N(0, \mu^T \Sigma^{-1} \mu)$ distribution is bounded above by $1/\sqrt{2\pi \mu^T \Sigma^{-1} \mu}$.  Hence we conclude from the above that in this case the margin assumption, i.e., Assumption~\ref{assumption:margin_condition}, is fulfilled with $\alpha=1$.
%\PP(0<|\log(f^0/f^1)(X)|\le t)\le C_0 t^{\alpha}, \forall t>0.

\bibliographystyle{plain}
\bibliography{sigma_short_journal_title}

\end{document}